\documentclass[11pt]{article}
\usepackage{amsmath}
\usepackage{amsthm}
\usepackage{amsfonts}
\usepackage{graphicx}
\usepackage{latexsym}
\usepackage{url}
\usepackage{bbm}
\usepackage{mathtools}
\mathtoolsset{showonlyrefs}

 \usepackage{todonotes}

\newcommand{\R}{\mathbb{R}}

\newcommand{\E}{\mathbb{E}}
\newcommand{\PP}{\mathbb{P}}
\newcommand{\N}{\mathbb{N}}
\newcommand{\Z}{\mathbb{Z}}

\newcommand{\F}{\mathbb{F}}
\newcommand{\1}{\mathbbm{1}}
\renewcommand{\P}{\mathbb P}

\newcommand{\g}{\mathcal G}

\newcommand{\f}{\mathcal F}

\newcommand{\toop}{\stackrel{\PP}{\longrightarrow}}
\newcommand{\schw}{\stackrel{d}{\longrightarrow}}
\newcommand{\eqschw}{\stackrel{d}{=}}
\newcommand{\stab}{\stackrel{\mathcal{L}-s}{\longrightarrow}}

\newcommand{\toas}{\stackrel{\mbox{\tiny a.s.}}{\longrightarrow}}

\renewcommand{\epsilon}{\varepsilon}

\newcommand{\bee}{\begin{equation}}
\newcommand{\eee}{\end{equation}}
\newcommand{\beea}{\begin{array}}
\newcommand{\eeea}{\end{array}}

\renewcommand{\theequation}{\arabic{section}.\arabic{equation}}

\theoremstyle{plain}
\newtheorem{prop}{Proposition}[section]

\newtheorem{theo}[prop]{Theorem}
\newtheorem{lem}[prop]{Lemma}

\theoremstyle{definition}

\newtheorem{rem}[prop]{Remark}

\begin{document}

\title{Limit theorems for a class of \\ stationary increments L\'evy driven moving averages}
\author{Andreas Basse-O'Connor\thanks{Department
of Mathematics, Aarhus University, 
E-mail: basse@math.au.dk.} \and 
Rapha\"el Lachi\`eze-Rey\thanks{Department
of Mathematics, University Paris Decartes, 
E-mail: raphael.lachieze-rey@parisdescartes.fr.} \and
Mark Podolskij\thanks{Department
of Mathematics, Aarhus University,
E-mail: mpodolskij@math.au.dk.}}

\maketitle

\begin{abstract}
In this paper we present some new limit theorems for power variation of $k$th order increments of  
stationary increments L\'evy driven moving averages. In the infill asymptotic setting, where the sampling frequency converges to zero while the time span remains fixed, 
the asymptotic theory
gives very surprising results, which (partially) have no counterpart in the theory of discrete moving averages.
More specifically, we will show that the first order limit theorems and the mode of convergence 
strongly depend on the interplay between the given order of the increments,
the considered power $p>0$, the Blumenthal--Getoor index $\beta \in (0,2)$ of the driving pure jump L\'evy process $L$ and the behaviour of the kernel function
$g$ at $0$ determined by the power $\alpha $. 
First order asymptotic theory essentially comprises three cases: stable convergence towards a certain infinitely divisible distribution,
an ergodic type limit theorem and convergence in probability towards an integrated random process. We also prove the second order limit theorem 
connected to the ergodic type result. When the driving L\'evy process $L$ is a symmetric $\beta$-stable process we  obtain two different limits: 
a central limit theorem and convergence in distribution towards a  $(1-\alpha )\beta$-stable 
totally right skewed random variable.

\ \

{\it Key words}: \
Power variation, limit theorems, moving averages, fractional processes,  stable convergence, high frequency data.\bigskip

{\it AMS 2010 subject classifications.} Primary~60F05,~60F15,~60G22;
secondary~60G48, ~60H05.

\end{abstract}

\section{Introduction and main results} \label{Intro}
\setcounter{equation}{0}
\renewcommand{\theequation}{\thesection.\arabic{equation}}

In the recent years there has been an increasing interest in limit theory for power variations of stochastic processes. Power variation functionals 
and related statistics play a major role in analyzing the fine properties of the underlying model, in stochastic integration concepts and statistical
inference. In the last decade asymptotic theory for power variations of various classes of stochastic processes has been intensively 
investigated in the literature.
We refer e.g.\  to \cite{BGJPS,J,JP,PV} for limit theory for power variations of It\^o semimartingales, to \cite{BNCP09,BNCPW09,Coeu,gl89,nr09} 
for the asymptotic results in the framework of fractional Brownian motion and related processes, and to \cite{Rosenblatt, CTV11, Rosenblatt-1} for investigations 
of power variation of the Rosenblatt process. 
  
In this paper we study the power variation for a class of \textit{stationary increments L\'evy driven moving averages}. More specifically, we consider 
an infinitely divisible process with  stationary increments  $(X_t)_{t\geq 0}$, defined on a  probability space $(\Omega, \mathcal F,   \mathbb P)$, 
given as
\begin{align} \label{def-of-X-43}
X_t= \int_{-\infty}^t \big\{g(t-s)-g_0(-s)\big\}\, dL_s.
\end{align} 
Here $L=(L_t)_{t\in \R}$ is a symmetric L\'evy process on $\R$ with  $L_0=0$ and without Gaussian
component. Furthermore,  $g,g_0: \R \to \R$ are deterministic  functions vanishing on $(-\infty,0)$. 
In the further discussion
we will need the notion of \textit{Blumenthal--Getoor index} of $L$, which is defined via
\begin{align} \label{def-B-G}
\beta:=\inf\left\{r\geq 0: \int_{-1}^1 |x|^r\,\nu(dx)<\infty\right\}\in [0,2],
\end{align} 
where $\nu$ denotes the L\'evy measure of $L$. When $g_0=0$, the process $X$ is a moving average,  and in this case $X$ is a stationary process.  If $g(s)=g_0(s)=s^\alpha_+$,
$X$ is a so called \text{fractional L\'evy process}. In particular, when $L$ is a $\beta$-stable L\'evy motion with $\beta \in (0,2)$,
$X$ is called a linear fractional stable motion and it is self-similar with index $H=\alpha+1/\beta$; see e.g.\ \cite{SamTaq}
(since in this case the stability index and the Blumenthal--Getoor index coincide, they are both denoted by $\beta$). 

Probabilistic analysis of stationary increments L\'evy driven moving averages such as semimartingale property, fine scale structure
and integration concepts, have been investigated in several papers. We refer to the work
of \cite{BasRosSM,BCI, bls12,  bm08,FK} among many others. 
However, only  few results on the power variations of such processes are presently available. Exceptions to this are \cite[Theorem~5.1]{BCI} and   \cite[Theorem~2]{sg14};  see     Remark~\ref{rem3} for a closer discussion of a result from \cite[Theorem~5.1]{BCI}. These two results are concerned with certain power variations of fractional L\'evy process and have some overlap with our Theorem~\ref{maintheorem}(ii) for the linear fractional stable motion, but we apply different proofs.  The aim of this paper is to derive a rather complete picture of the first order asymptotic theory for power variation of the process $X$, and,
in some cases, the associated second order limit theory. We will see that the type of convergence and the limiting random variables/distributions
are quite surprising and novel in the literature. Apart from pure probabilistic interest, limit theory for power variations of 
stationary increments L\'evy driven moving averages give rise to a variety of statistical  
methods. In particular, the theoretical results can be applied to identify and estimate the parameters 
$\alpha$ and  $\beta$ of the model (cf. Section  \ref{stat}).  We refer to e.g. \cite{BCI, Rosenblatt, CTV11, glt15} for related statistical procedures. Furthermore, the asymptotic results  
provide a first step towards limit theory for power variation of stochastic processes,
which contain $X$ as a building block. In this context let us mention stochastic integrals with respect to $X$ and L\'evy semi-stationary processes,
which have been introduced in \cite{BNBV}.  

To describe our main results we need to introduce some notation and a set of assumptions. In this work    
we consider the $k$th order increments $\Delta_{i,k}^{n} X$ of $X$, $k\in \N$, that are defined by 
\begin{align} \label{filter}
\Delta_{i,k}^{n} X:= \sum_{j=0}^k (-1)^j \binom{k}{j} X_{(i-j)/n}, \qquad i\geq k.
\end{align}
 For instance, we have that $\Delta_{i,1}^n X = X_{\frac{i}{n}}-X_{\frac{i-1}{n}}$ and $\Delta_{i,2}^n X = X_{\frac{i}{n}}-
2X_{\frac{i-1}{n}}+X_{\frac{i-2}{n}}$. Our main functional  is the power variation computed on the basis of $k$th order filters:
\begin{align} \label{vn}
V(p;k)_n := \sum_{i=k}^n |\Delta_{i,k}^{n} X|^p, \qquad p>0.
\end{align}
Now, we introduce the following set of assumptions on $g$ and $\nu$:
   
\noindent \textbf{Assumption~(A):}
The function $g\!:\R\to\R$ satisfies $g\in C^k((0, \infty))$ and
\begin{align}\label{kshs}
g(t)\sim c_0 t^\alpha\qquad \text{as } t\downarrow 0\quad \text{for some }\alpha>0\text{ and }c_0\neq  0, 
\end{align}
where $g(t)\sim f(t)$ as $t\downarrow 0$ means that $\lim_{t\downarrow 0}g(t)/f(t)= 1$. 
For some $\theta\in (0,2]$, $\limsup_{t\to \infty} \nu(x\!:|x|\geq t) t^{\theta}<\infty$ and $g - g_0$ is a bounded function in $L^{\theta}(\R_+)$. Finally,
there  exists a $\delta>0$ such that  $|g^{(k)}(t)|\leq K t^{\alpha-k}$ for all $t\in (0,\delta)$,  $|g'|$ and $|g^{(k)}|$ are in $L^\theta((\delta,\infty))$ and  decreasing on $(\delta,\infty)$.

\noindent \textbf{Assumption~(A-log):}
In addition to  (A) suppose that $\int_\delta^\infty |g^{(k)}(s)|^\theta \log(1/|g^{(k)}(s)|)\,ds<\infty$. 

Assumption~(A) ensures in particular  that the process $X$ is well-defined, cf.\ Section~\ref{secPrel}. When $L$ is a $\beta$-stable L\'evy process, we always choose $\theta = \beta$ in assumption (A).
Before we introduce the main results, we need some more notation. 
Let $h_k\!:\R\to\R$ be given by 
\begin{align}\label{def-h-13}
h_k(x)&=  \sum_{j=0}^k (-1)^j \binom{k}{j} (x-j)_{+}^{\alpha},\qquad x\in \R,
\end{align} 
where  $y_+=\max\{y,0\}$ for all $y\in \R$. 
Let $\F=(\f_t)_{t\geq 0}$ be the  filtration generated by $(L_t)_{t\geq 0}$, $(T_m)_{m\geq 1}$  be a   sequence of $\F$-stopping times  that exhausts the jumps of $(L_t)_{t\geq 0}$. That is,  $ \{T_m(\omega):m\geq 1\}\cap \R_+ = \{t\geq 0: \Delta L_t(\omega)\neq 0\}$ and  $T_m(\omega)\neq T_n(\omega)$ for all $m\neq n$ with $T_m(\omega)<\infty$. Let     $(U_m)_{m\geq 1}$ be independent and uniform $[0,1]$-distributed random variables, defined on an extension $(\Omega', \mathcal F', \mathbb P')$ of the original probability space,
which are independent of $\f$. 

The following two theorems summarize the first and  second order limit theory for the power variation $V(p;k)_n$. We would like to emphasize 
part (i) of Theorem \ref{maintheorem} and part (i) of Theorem \ref{sec-order}, which are quite unusual  probabilistic results. We refer to \cite{Aldous,ren}
and to Section \ref{secPrel} for the definition of $\mathcal F$-stable convergence in law which will be denoted $\stab$.    
   
\begin{theo}[First order asymptotics] \label{maintheorem}
Suppose (A) is satisfied and    assume that the Blumenthal--Getoor index satisfies $\beta<2$. 
We obtain the following three cases:
 
\begin{itemize}
\item[(i)] \label{mt-case-1} Suppose that (A-log) holds if  $\theta=1$. If $\alpha <k-1/p$ and $p>\beta$ we obtain the $\f$-stable convergence
\begin{equation} \label{part1}
n^{\alpha p}V(p;k)_n \stab |c_0|^p\!\!\!\!\! \!\sum_{m:\, T_m\in [0,1]} |\Delta L_{T_m}|^p V_m\quad\text{where}\quad 
V_m=\sum_{l=0}^{\infty} |h_k(l+U_m)|^p.
\end{equation} 
\item[(ii)] \label{mt-case-2} Suppose that  $L$ is a symmetric $\beta$-stable L\'evy process with scale parameter $\sigma>0$, i.e. the characteristic function of $L_1$ is given by 
$\E[\exp(iuL_1)]=\exp(-\sigma^{\beta}|u|^{\beta})$. If  $\alpha <k-1/\beta$ and   $p<\beta$ then it holds
that
\begin{align} \label{part2}
n^{-1+p(\alpha + 1/\beta)}V(p;k)_n \toop m_p
\end{align}
where  $m_p=|c_0|^p \sigma^p (\int_\R |h_k(x)|^\beta\,dx)^{p/\beta}\E[|Z|^p]$ and $Z$ is a symmetric $\beta$-stable random variable with scale parameter $1$. 

\item[(iii)] \label{mt-case-3} Suppose that $p\geq 1$. If  $p= \theta$ suppose in addition that (A-log) holds. For all    $\alpha>k-1/( \beta\vee p)$ 
%(and  $p\neq \beta$)  
we deduce that
\begin{equation} \label{part3}
  n^{-1+pk}V(p;k)_n \toop \int_0^1 |F_u|^p\, du 
 \end{equation}
where $(F_u)_{u\in \R}$ is a measurable process satisfying
\begin{equation}
F_u=  \int_{-\infty}^u g^{(k)}(u-s) \,dL_s\quad \text{a.s.\ for all }u\in \R \quad \text{and}\quad \int_0^1 |F_u|^{p}\,du<\infty\quad \text{a.s.}
\end{equation} 
\end{itemize}
\end{theo}

We remark that,  
except the critical cases where $p= \beta$, $\alpha= k-1/p$ and $\alpha= k-1/\beta$, 
Theorem~\ref{maintheorem} covers all possible choices of $\alpha>0,\beta\in [0,2)$ and  $p\geq 1$. We also note that the limiting random variable in 
\eqref{part1} is infinitely divisible, see Section~\ref{gencom} for more details. In addition, we note that there is no convergence in probability in \eqref{part1} 
due to the fact that the random variables $V_m$, $m\geq 1$, are independent of $L$ and the properties of stable convergence. 
To be used in the next theorem we recall that a   totally right skewed $\rho$-stable random variable $S$ with $\rho>1$, mean zero and scale parameter $\sigma>0$ has characteristic function given by 
\begin{equation}
\E[ e^{i \theta S} ] =\exp\Big( -\sigma^\rho |\theta |^\rho \big( 1- i \mathrm{sign}(\theta)\tan(\pi \rho/2)\big)\Big),\qquad \theta\in \R. 
\end{equation}

For part (ii) of Theorem \ref{maintheorem}, which we will refer to as the ergodic case, we also show the second order asymptotic results under the additional condition $p<\beta/2$. We remark that for 
$k=1$ we are automatically in the regime of Theorem~\ref{sec-order}(i).
\begin{theo}[Second order assymptotics]\label{sec-order}
Suppose that assumption (A) is satisfied and  $L$ is a symmetric $\beta$-stable L\'evy process with scale parameter $\sigma>0$. 
Let  $f:[0,\infty )\mapsto \R$ be given by  $f(t)=g(t)t^{-\alpha}$ for  $t>0$ and $f(0)= c_0$. Fix  $k\geq 1$ and  assume that $f$ is $k$-times  continuously right differentiable at $0$ and  $| g^{(k)}(t)|\leq K t^{\alpha-k}$ for all $t>0$. For all $p<\beta/2$ we  have the following two cases:
\begin{itemize}
\item [(i)]   Suppose that  $\alpha\in (k-2/\beta,k-1/\beta)$. If  $\beta<1/2$ assume in addition that  $\alpha>k-\frac{1}{\beta(1-\beta)}$. Then it  holds that
\[
n^{1-\frac{1}{(k-\alpha)\beta}}\Big(n^{-1+p(\alpha + 1/\beta)}V(p;k)_n- m_p\Big) \schw S,
\]
where $S$ is a totally right skewed  $(k-\alpha)\beta$-stable random variable with mean zero and scale parameter  $\widetilde{\sigma}$, which is defined in  Remark~\ref{rem-const}(i).   
\item[(ii)]  \label{mt-case-6} If   $\alpha \in (0,k-2/\beta)$ we deduce that 
\begin{align} \label{part5}
\sqrt{n} \Big( n^{-1+p(\alpha + 1/\beta)}V(p;k)_n - m_p \Big) \schw \mathcal N(0, \eta^2),
\end{align} 
where the quantity $\eta^2\in (0,\infty)$ is defined via
\begin{align} \label{etaformula}
\eta^2 = \lim_{n \to \infty} \left( \text{\rm var}\left( |n^{\alpha +1/\beta}\Delta_{k,k}^{n} X|^p \right) 
+ 2 \sum_{l=1}^{n-k} \text{\rm cov}\left(  |n^{\alpha +1/\beta}\Delta_{k,k}^{n} X|^p, 
 |n^{\alpha +1/\beta}\Delta_{k+l,k}^{n} X|^p \right)  \right).
\end{align}
\end{itemize}
\end{theo}

This paper is structured as follows. The basic ideas and methodology of the proofs are demonstrated 
in Section ~\ref{sec3a}. Section~\ref{sec2} presents some remarks about the nature and applicability of the main results.   Section~\ref{secPrel} introduces some preliminaries.  
We state the proof of Theorem \ref{maintheorem} in Section~\ref{proofs-w3lkhj}, while the proof of Theorem \ref{sec-order} is demonstrated in Section \ref{sec5}.

%%%%%%%%%%%%%%%%%%%
\section{Basic ideas and methodology} \label{sec3a}
%%%%%%%%%%%%%
\setcounter{equation}{0}
\renewcommand{\theequation}{\thesection.\arabic{equation}}

\subsection{First order asymptotics} 
In this section we explain the basic intuition and the methodology of the proofs of Theorem~\ref{maintheorem}. For simplicity of exposition we only consider the case
$g_0=0$, $k=1$ and we set $\Delta_{i}^{n} X:=\Delta_{i,1}^{n} X$, $h:=h_1$ and $V(p)_n:=V(p;1)_n$.

In order to uncover the path properties of the process $X$ we perform a formal differentiation with respect to time. Since $g(0)=0$ we obtain a formal representation 
\begin{align} \label{formalsde}
dX_t = g(0) dL_t + \left(\int_{-\infty}^t g'(t-s) \,dL_s \right) dt = F_t \,dt.
\end{align}  
We remark that the path $(F_t(\omega))_{t \in [0,1]}$ 
is not necessarily bounded under assumption (A). However, 
under conditions of Theorem \ref{maintheorem}(iii), the process $X$ is differentiable almost everywhere and $X'=F\in
L^p([0,1])$; see Lemma \ref{abs-cont-sdf} for a detailed exposition.  Thus,
under conditions of Theorem~\ref{maintheorem}(iii),  an application of the mean value theorem gives
an intuitive proof  of \eqref{part3}:
\begin{align*}
\text{$\P$-}\!\!\!\lim_{n\to \infty} n^{-1+p} V(p)_n = 
\text{$\P$-}\!\!\!\lim_{n\to \infty} \frac{1}{n}   \sum_{i=1}^n |F_{\xi_i^n}|^p 
=\int_0^1 |F_u|^p\, du,
\end{align*}
where $\xi_i^n \in ((i-1)/n, i/n)$; we refer to Lemma \ref{lem-f-65} for a formal argument. 
This gives a sketch
of the proof of the asymptotic result at \eqref{part3}.

Now, we turn our attention to the small scale behaviour of the stationary increments L\'evy driven moving averages $X$. Recall that under conditions of Theorem \ref{maintheorem}(ii), 
$\alpha<1-1/\beta$ and thus $g'$ has an explosive behaviour at $0$. Hence,
we intuitively deduce the 
following approximation for the increments of $X$ for a small $\Delta>0$:
\begin{align}
X_{t+\Delta} - X_t &= \int_{\R} \{g(t+\Delta -s) - g(t -s)\}\, dL_s \nonumber \\
& \label{tangent} \approx \int_{t+\Delta -\epsilon }^{t+\Delta} \{g(t+\Delta -s) - g(t -s)\} \,dL_s \\
& \approx c_0 \int_{t+\Delta -\epsilon }^{t+\Delta} \{(t+\Delta -s)_{+}^\alpha  - (t -s)_{+}^\alpha \} \,dL_s 
 \nonumber \\
& \approx c_0 \int_{\R} \{(t+\Delta -s)_{+}^\alpha  - (t -s)_{+}^\alpha \} \,dL_s = \widetilde{X}_{t+\Delta} - \widetilde{X}_t,  \nonumber
\end{align} 
where 
\begin{align} \label{flm1}
\widetilde{X}_t := c_0 \int_{\R} \{(t-s)_{+}^\alpha  - ( -s)_{+}^\alpha \}\, dL_s,
\end{align} 
and $\epsilon >0$ is an arbitrary small real number with $\epsilon \gg \Delta$. In the classical terminology $\widetilde{X}$ is called
the \textit{tangent process of $X$}.    
In the framework of  
Theorem \ref{maintheorem}(ii) the process $\widetilde{X}$ 
is a symmetric fractional $\beta$-stable motion. 
We recall that $(\widetilde{X}_t)_{t \geq 0}$ has stationary increments, symmetric $\beta$-stable marginals,  H\"older index $\alpha$ (cf. \cite[Theorem~3.4]{Takashima})  and 
it is self-similar with index $H=\alpha + 1/\beta \in (1/2,1)$, i.e.
\[
(\widetilde{X}_{at})_{t \geq 0} \eqschw a^H(\widetilde{X}_t)_{t \geq 0}
\qquad \text{for any } a\in \R_+.
\] 
Furthermore,   the symmetric fractional $\beta$-stable noise 
$(\widetilde{X}_t - \widetilde{X}_{t-1})_{t \geq 1}$ is mixing; 
see e.g.\ \cite{Ergodic-Cam}. Thus, using Birkhoff's ergodic theorem we conclude that 
\begin{align*}
n^{-1+p(\alpha + 1/\beta)}V(p)_n & = \frac{1}{n} \sum_{i=1}^n |n^H \Delta_{i}^{n} X|^p \\
& \approx \frac{1}{n} \sum_{i=1}^n |n^H \Delta_{i}^{n} \widetilde{X}|^p \\
&  \eqschw \frac{1}{n} \sum_{i=1}^n |\widetilde{X}_i - \widetilde{X}_{i-1}|^p \toop 
\E [|\widetilde{X}_1 - \widetilde{X}_{0}|^p]. 
\end{align*}  
This method sketches the proof of the convergence at  
\eqref{part2}.

\begin{rem} \label{rem3.1} \rm
We recall that $L$ is assumed to be a symmetric $\beta$-stable L\'evy process   in Theorems \ref{maintheorem}(ii) and \ref{sec-order}. 
We conjecture that this assumption
can be relaxed following the discussion of tangent processes at \eqref{tangent}. 
Indeed, the small jumps of the driving L\'evy process $L$ are dominating in the asymptotic results
of Theorems \ref{maintheorem}(ii) and \ref{sec-order}. Thus, it seems to suffice when 
small jumps of $L$ are in the domain of attraction of a symmetric $\beta$-stable L\'evy
process, e.g.\   the  L\'evy measure of the process $L$ satisfies the decomposition
\[
\nu (dx) = \left(\text{const} \cdot |x|^{-1-\beta} + \varphi(x) \right) dx,
\]  
where the function $\varphi$ satisfies the conditions $\int_{\R} (1 \wedge  x^2) \varphi(x) dx <
\infty$ and  $\varphi(x)=o(|x|^{-1-\beta})$ as $x\rightarrow 0$. Such processes include for instance  tempered or truncated symmetric $\beta$-stable L\'evy motions. 
We believe that the statement of Theorem \ref{maintheorem}(ii) remains valid for the above 
form of the L\'evy density. Theorem  \ref{sec-order} would probably require a stronger condition
on the function $\varphi$.
\qed
\end{rem}

At this stage we need to better understand the fine scale behaviour of the process $X$ in order
to describe the intuition behind the non-standard result of Theorem  \ref{maintheorem}(i). For the ease 
of exposition we will discuss the following simple framework: Assume that the driving motion $L$ jumps only once at random time $T$, which has a density on the interval $(0,1)$ (note that 
$L$ is not a L\'evy process in this case). That is, $L$ has the representation 
\[
L_t = \1_{(-\infty, T)} (t) \Delta L_T.
\]   
Now, we will describe  the asymptotic distribution of the 
scaled increments $n^{\alpha} \Delta_{i}^{n} X$. Let $j_n$ be a random index satisfying
$T \in [(j_n-1)/n,j_n/n)$. Similarly to the approximation at \eqref{tangent}, we obtain that  
\begin{align*}
\Delta_{i}^{n} X  \approx A_i^n := c_0 
\int_{0}^{\frac{i}{n}} \Big\{ \Big( \frac{i}{n} -s \Big)^{\alpha}_{+} - 
\Big( \frac{i-1}{n} -s \Big)^{\alpha}_{+}  \Big \} \, dL_s.
\end{align*}
Since $T \in [(j_n-1)/n,j_n/n)$ 
is the only jump time of $L$, we observe that $A_i^n = 0$  
for all $i<j_n$.  More precisely, we deduce that 
\begin{align*}
\Delta_{j_n+l}^{n} X \approx
c_0 \Delta L_T \left( \Big( \frac{j_n+l}{n} -T \Big)_+^{\alpha} - 
\Big( \frac{j_n+l-1}{n} -T \Big)_+^{\alpha} \right), \qquad  l \geq 0.
\end{align*}
Now, we use the following result, which is essentially due to Tukey  \cite{Tukey} (see also
\cite{Jacod-round-off} and Lemma \ref{frac-part-lem} below): Let $Z$ be a random 
variable with an absolutely continuous distribution and let $\{x\}:=x-\lfloor x\rfloor\in [0,1)$ denote
the fractional part of $x\in \R$. Then it holds that 
\[
\{nZ\} \stab U \sim \mathcal{U}([0,1]),
\]   
where $U$ is defined on the extended probability space and $U$ is independent of $Z$.
Since $j_n-nT=1- \{n T\}$ and $1-U \sim \mathcal{U}([0,1])$,  we conclude the stable convergence 
in law 
\[
n^{\alpha} \Delta_{j_n+l}^{n} X \stab c_0 \Delta L_T \left((l+U)_+^{\alpha} -  (l-1+U)_+^{\alpha} \right),
\qquad l \geq 0.
\]
Thus, in this setting we obtain the result of \eqref{part1} as follows:
\begin{align} \label{stable1jump}
n^{\alpha p} V(X,p)_n \approx \sum_{i=j_n}^n |n^{\alpha} A_i^n |^p \stab  |c_0 \Delta L_T|^p
\sum_{l=0}^{\infty} \left|(l+U)_+^{\alpha} -  (l-1+U)_+^{\alpha} \right|^p,
\end{align} 
which gives an intuitive proof of Theorem \ref{maintheorem}(i). A formal proof of the stable convergence at \eqref{part1} for a general L\'evy motion $L$ 
requires a decomposition of the driving jump measure associated with 
$L$ into big and small jumps, and a certain time separation between the big jumps.

\subsection{Weak limit theorems}
In this section we highlight the basic methodology behind the proof of Theorem~\ref{sec-order}. 
For the sake of exposition, we will rather consider the power variation 
$V(\widetilde{X}, p;k)_n$ of the tangent process 
$\widetilde{X}$ defined at \eqref{flm1} driven by a symmetric $\beta$-stable
L\'evy motion $L$.

Weak limit theory for statistics of \textit{discrete} moving averages has been a subject of a deep investigation during the last thirty years. 
In a functional framework a variety of different limit distributions may appear. They include
Brownian motion, $m$th order Hermite  processes, stable L\'evy processes with various stability indexes and fractional Brownian motion. 
We refer to the papers \cite{at87,hh97,h99,ks01,s02,s04} for an overview. 

Let us start with the treatment of Theorem~\ref{sec-order}(i). By self-similarity of the symmetric fractional $\beta$-stable motion $\widetilde{X}$, we conclude that 
\begin{align} \label{distreq}
n^{1-\frac{1}{(k-\alpha)\beta}}\Big(n^{-1+p(\alpha + 1/\beta)}V(\widetilde{X},p;k)_n- m_p\Big)
\eqschw n^{-\frac{1}{(k-\alpha)\beta}} \sum_{i=k}^n H \left(\Delta_{i,k} \widetilde{X}\right), 
\end{align}  
where 
\[
\Delta_{i,k} \widetilde{X}:= \sum_{j=0}^k (-1)^j \binom{k}{j} \widetilde{X}_{i-j} \qquad 
\text{and} \qquad  H(x):=|x|^p -m_p.
\]
In this framework the most important ingredient is the \textit{Appell rank}
of the function $H$ (cf.\ \cite{at87}). We recall that for a general function $H:\R \to \R$
with $\E[H(\Delta_{k,k} \widetilde{X})]=0$ 
the Appell rank of $H$ is defined via
\begin{align*}
m^{\star} := \min_{m\geq 1} \{H_\infty^{(m)} (0) \not= 0\} \qquad \text{with} \qquad H_\infty (x):= \mathbb E[H(\Delta_{k,k} \widetilde{X}+x)].  
\end{align*} 
Notice that in the setting $H(x)=|x|^p -m_p$ we have that $m^{\star}=2$, since $\widetilde{X}$ has 
symmetric distribution and the function $x \mapsto |x|^p$ is even.  
It turns out that Appell rank $m^{\star}$ together with the parameter $\alpha$ and the tail behaviour 
of the noise $L_1$ determines the limiting behaviour of the statistic on the right hand side 
of \eqref{distreq}. The weak limit theory for $m^{\star}=1,2,3$ in the context of discrete moving average  
processes has been investigated in  \cite{hh97,h99, s02,s04} among others. We remark however that
for $m^{\star}\geq 2$ the authors
 only consider bounded functions $H$ and existence of second moments of the noise process. 
Both assumptions are obviously not satisfied in our setting since $\mathbb E[L_1^2]=\infty$. 

The key to proving Theorem~\ref{sec-order}(i) are several projection techniques that are described in
details in Section \ref{sec6.1} (cf. Eq. \eqref{eq:7836638-22}). In particular, they show the 
decomposition 
\[
n^{-\frac{1}{(k-\alpha)\beta}} \sum_{i=k}^n H \left(\Delta_{i,k} \widetilde{X}\right) =
n^{-\frac{1}{(k-\alpha)\beta}} \sum_{i=k}^n Z_i + o_{\mathbb P} (1),
\]
where $(Z_i)_{i \geq k}$ is a certain sequence of i.i.d. random variables. Thus, by  \cite[Theorem~1.8.1]{SamTaq}, it is sufficient to determine the tail behaviour of $Z_1$.
That is, we prove the convergence 
\[
\lim_{x\to\infty}x^{(k-\alpha)\beta}\P( Z_1>x)= \gamma \qquad \text{and}\qquad \lim_{x\to\infty}x^{(k-\alpha)\beta} \P(Z_1<-x)=0
\]
for a constant $\gamma \in (0,\infty)$, which completes the proof of Theorem~\ref{sec-order}(i)
(cf. Section \ref{sec6.3}). At this stage we remark that Theorem~\ref{sec-order}(i) 
is similar in spirit to the results of \cite{s04}. Indeed, we apply a similar proof strategy to show the weak convergence. However, strong modifications due to unboundedness of $H$, triangular nature of summands in \eqref{vn}, stochastic integrals instead of sums,  and the different set of conditions are required. 

In order to describe the main ideas behind the proof of Theorem~\ref{sec-order}(ii) we again observe  the identity in distribution
\begin{align}
\sqrt{n}\Big(n^{-1+p(\alpha + 1/\beta)}V(\widetilde{X},p;k)_n- m_p\Big)
\eqschw \frac{1}{\sqrt{n}}\sum_{i=k}^n H \left(\Delta_{i,k} \widetilde{X}\right)
=: \frac{1}{\sqrt{n}} S_n. 
\end{align} 
In the first step the term $S_n$ is approximated by the quantity $S_{n,m}$, which is a sum of $m$-dependent identically distributed random variables. This approximation is obtained by a proper cut off 
in the integration region of the integral $\Delta_{i,k} \widetilde{X}$. In the second step we will show that
\[
\frac{1}{\sqrt{n}} S_{n,m} \schw \mathcal N(0, \eta^2_m) \quad \text{as } n \to \infty
\qquad \text{and} \qquad \lim_{m\to \infty} \eta^2_m = \eta^2. 
\]  
Hence, the proof of Theorem~\ref{sec-order}(ii) is complete if we show the convergence
\[
\lim_{m\to \infty} \limsup_{n\to \infty} \big(n^{-1} \E[(S_{n}-S_{n,m})^2]\big)  = 0,
\]
which is the main part of the proof. It again relies on rather complex projection techniques similar to 
the proof of  Theorem~\ref{sec-order}(i) (cf. Eq. \eqref{eq:283723}).

\begin{rem} \label{rem2} \rm
The symmetry condition on the L\'evy process $L$ is assumed for sake of assumption simplification. Most asymptotic results of this paper
would not change if we dropped this condition. However, the Appell rank of the function 
$H(x)=|x|^p - m_p$ might be $1$ when $L$ is not symmetric
and this does change the result of Theorem \ref{sec-order}(i). We conjecture that the limiting distribution becomes $\beta$-stable in this framework (see e.g.\  \cite{ks01} for the discrete moving average setting). 
However, 
we dispense with the exact exposition of this case. \qed  
\end{rem}

%%%%%%%%%%%%%%%%%%%
\section{Some remarks and applications} \label{sec2}
%%%%%%%%%%%%%
\setcounter{equation}{0}
\renewcommand{\theequation}{\thesection.\arabic{equation}}

\subsection{General comments} \label{gencom}
In this section we discuss the set of assumptions and the various statements of the main 
theoretical results. We start by commenting on the set of conditions introduced in assumption (A).   
First of all, assumption (A) ensures that the process $X$ introduced at  \eqref{def-of-X-43} is well-defined (cf. Section \ref{secPrel}). 
More importantly, the conditions of assumption (A) guarantee that the quantity
\begin{align} \label{impcond}
\int_{-\infty}^{t-\epsilon} g^{(k)}(t-s) \,dL_s, \qquad \epsilon >0,
\end{align}            
is well-defined (typically, this does not hold true for $\epsilon =0$). 
The latter is crucial for the proof of Theorem \ref{maintheorem}(i).
We recall that the condition $p\geq 1$ is imposed in Theorem \ref{maintheorem}(iii). 
We think that this condition might not be necessary, but the results
of \cite{SamBra} applied in our proofs require $p\geq 1$. Finally, we note that 
assumption (A-$\log$) will be used only for the case $\theta =1$ (resp.\ $\theta=p$) in part (i) (resp.\ part (iii)) of Theorem \ref{maintheorem}.

The conditions $\alpha \in (0,k-1/p)$ and $p>\beta$ of Theorem \ref{maintheorem}(i) seem to be sharp.
Indeed, Taylor expansion implies that $|h_k(x)|\leq K|x|^{\alpha -k}$ for large $x$. Consequently, 
we obtain from \eqref{part1} that 
\[
\sup_{m\geq 1} V_m <\infty  
\]
when $\alpha \in (0,k-1/p)$. On the other hand $\sum_{m:T_m\in [0,1]} |\Delta L_{T_m}|^p<\infty$ for $p>\beta$, which follows from the definition
of the Blumenthal--Getoor index at \eqref{def-B-G}. Notice that under assumption $\alpha \in (0, k-1/2)$
the case $p=2$, which corresponds to quadratic variation, always falls under Theorem \ref{maintheorem}(i).  

We remark that the distribution of the limiting variable in \eqref{part1} does not depend
on the chosen sequence $(T_m)_{m\geq 1}$ of stopping times which exhausts the jump times of $L$. Furthermore, 
the limiting random variable $Z$ in \eqref{part1} is infinitely divisible with L\'evy 
measure $(\nu\otimes \eta)\circ \big((y,v)\mapsto  |c_0 y|^p v\big)^{-1}$, where $\eta$ denotes the law of $V_1$. In fact, $Z$ has characteristic function
given by 
\begin{equation}
\E[\exp(i \theta Z)]=\exp\Big( \int_{\R_0\times \R} (e^{i \theta |c_0 y|^p v}-1)  \,\nu(dy)\,\eta(dv)   \Big).
\end{equation}
To show this, let  $\Lambda$ be the Poisson random measure given by 
$\Lambda=\sum_{m=1}^\infty \delta_{(T_m,\Delta L_{T_m})}$ on $[0,1]\times \R_0$  which has intensity measure $\lambda\otimes \nu$. Here  $\R_0:=\R\setminus \{0\}$ and  $\lambda$ denotes the Lebesgue measure on $[0,1]$. 
Set  $\Theta=\sum_{m=1}^\infty \delta_{(T_m, \Delta L_{T_m},V_m)}$. Then $\Theta$ is a Poisson random measure with intensity  measure $\lambda\otimes \nu \otimes \eta$, due to  \cite[Theorem~36]{Serfozo}. Thus, the  above claim  follows from the stochastic integral representation
\begin{equation}
Z= \int_{[0,1]\times \R_0\times \R} \big(|c_0 y|^p v\big) \, \Theta(ds,dy,dv).
\end{equation} 
As for Theorem \ref{maintheorem}(iii), we remark that for values of $\alpha$ close to $k-1/p$ or $k-1/\beta $, 
the function $g^{(k)}$ explodes at $0$. This leads to unboundedness of the path 
$(F_t(\omega))_{t \in [0,1]}$. Nevertheless,  the limiting random variable in \eqref{part3}
is still finite.

\begin{rem}\label{rem-const} (Definition of $\tilde \sigma$)
In order to introduce the constant $\tilde \sigma$  appearing in Theorem~\ref{sec-order}(i)
we set 
\begin{align}
\kappa = {}&\frac{k_\alpha^{1/(k-\alpha)} }{k-\alpha} 
\int_0^{\infty} \Phi (y) y^{-1-1/(k-\alpha )} \,dy,
\end{align}
where   $\Phi(y):= \E[|\Delta_{k,k} \widetilde{X}+y|^p - |\Delta_{k,k} \widetilde{X}|^p ] $, $y\in\R$,    $\widetilde{X}_t$ is a linear fractional stable motion  defined in \eqref{flm1} with $c_0=1$ and $L$ being a standard symmetric $\beta$-stable L\'evy process, and $k_\alpha = \alpha (\alpha-1)(\alpha-2)\cdots (\alpha-k+1)$. 
In addition, set 
\begin{equation}\label{def-tau-rho}
 \tau_{\rho}=  \frac{\rho-1}{\Gamma(2-\rho)|\cos(\pi \rho/2)|},\qquad \text{for all }\rho\in (1,2),
\end{equation}
where $\Gamma$ denotes the gamma function. Then, the scale parameter $\tilde \sigma$
is defined via
\begin{equation}
\tilde \sigma = |c_0|^p \sigma^p  \Big(\frac{\tau_\beta}{\tau_{(1-\alpha)\beta}}\Big)^{\frac{1}{(1-\alpha)\beta}}\kappa .
\end{equation}
The function $\Phi(y)$  
can be computed explicitly, see \eqref{rep-H-est-2}.  
This representation shows, in particular, that $\Phi (y)>0$ for all $y>0$, and hence the limiting variable 
$S$
in Theorem \ref{sec-order}(i) is not degenerate, because $\tilde \sigma >0$.
\qed
\end{rem}

\begin{rem} \label{rem3} \rm
Theorem~5.1 of \cite{BCI} studies the first order asymptotic of the power variation  of some fractional fields $(X_t)_{t\in \R^d}$. In the  case  $d=1$, they consider  fractional L\'evy processes $(X_t)_{t\in \R}$ of the form 
\begin{equation}\label{eq:7237}
 X_t= \int_{\R} \big\{|t-s|^{H-1/2} - |s|^{H-1/2}\big\}\,dL_s
\end{equation} where $L$ is a truncated $\beta$-stable L\'evy process. This  setting is close to fit into the framework of the present paper \eqref{def-of-X-43} with $\alpha= H-1/2$ except for the fact that  the stochastic integral \eqref{eq:7237} is over the hole real line. However, the proof of Theorem~\ref{maintheorem}(i) still holds for $X$ in \eqref{eq:7237} with  the obvious modifications of $h_k$ and $V_m$ in \eqref{def-h-13} and \eqref{part1}, respectively.  
Notice also that   \cite{BCI} considers  the power variation along the subsequence $2^n$, which corresponds to dyadic partitions, and their setting includes second order increments ($k=2$).  For $p<\beta$,  Theorem~5.1 of \cite{BCI} claims that 
$2^{n \alpha p} V(p;2)_{2^n}\to C$ almost surely, where $C$ is a positive constant. However, this contradicts Theorem~\ref{maintheorem}(i) together with the remark following it, namely that,   convergence in probability can not take place under the conditions of  Theorem~\ref{maintheorem}(i), not even trough a subsequence. It seems that the   last three lines of the   proof of  \cite[Theorem~5.1]{BCI}  are erroneous, since the derived   estimates are not uniform in the parameters which are required for the stated conclusion to hold, see  \cite[p.~372]{BCI}.  
\qed
\end{rem}

\subsection{Statistical applications} \label{stat}
The asymptotic theory of this paper has a variety of potential applications in statistics. We have
seen  in Section~ \ref{sec3a} that in the framework of a symmetric fractional $\beta$-stable motion
$(\widetilde{X}_t)_{t\geq 0}$ the parameter $H=\alpha +1/\beta \in(1/2,1)$ is the self-similarity 
index of the process $\widetilde{X}$ while $\alpha>0$ determines the H\"older continuity index of 
$\widetilde{X}$. 

Having understood the role of the parameters $\alpha>0$ and $H=\alpha+1/\beta \in (1/2,1)$ from 
the modelling perspective, it is obviously important to investigate estimation methods for these parameters when the underlying process is given by  $(X_t)_{t\geq 0}$.  We start with a direct estimation procedure that identifies the convergence 
rates in Theorem~\ref{maintheorem}(i)-(iii).  We apply these convergence results only for 
$k=1$. Since we have assumed that $\alpha>0$ and $H=\alpha+1/\beta \in (1/2,1)$, it must hold that 
$\beta \in (1,2)$. Notice also that the condition $p>1$ is required in  Theorem~\ref{maintheorem}(i) when $k=1$. Now, we define the statistic 
\begin{align} \label{statistic}
S_{\alpha, \beta} (n,p):= -\frac{\log V(p)_n}{\log n} \qquad \text{with} \qquad V(p)_n= V(p;1)_n.
\end{align} 
Assume that the underlying L\'evy motion $L$ is symmetric $\beta$-stable, in which case 
Theorems~\ref{maintheorem}(i)-(iii) are all applicable. Then, under assumptions of 
Theorem~\ref{maintheorem},   for $p>1$ it holds that 
\begin{align} \label{logscale}
S_{\alpha, \beta} (n,p) \toop  S_{\alpha, \beta} (p)  :=
\left \{
\begin{array} {ll}
\alpha p: & \alpha < 1-1/p \text{ and } p>\beta \\
pH-1: &\alpha < 1-1/\beta \text{ and } p<\beta \\
p -1:  & \alpha > 1-1/\max(p,\beta)
\end{array} 
\right. 
\end{align}
for any fixed $p \in (1,2)$. Indeed, the result of Theorem~\ref{maintheorem}(i) implies that 
\[
\frac{\alpha p \log n + \log V(p)^n}{\log n} \stab 0 \qquad \Rightarrow \qquad
\frac{\alpha p \log n + \log V(p)^n}{\log n} \toop 0,
\]
which explains the first line in  \eqref{logscale}. Similarly, Theorem~\ref{maintheorem}(ii) and  (iii)
imply the other convergence results of \eqref{logscale}. At this stage we remark that the limit 
$S_{\alpha, \beta} (p)$ is a piecewise linear function in $p \in (1,2)$ with different slopes. Indeed,
it suffices to only consider $p \in (1,2)$ to uncover all three slopes.
Now, it is natural to consider 
the $L^2$-distance between the observed scale function $S_{\alpha, \beta} (n,p)$ and the theoretical
$S_{\alpha, \beta} (p)$:
\begin{align} \label{scaleest}
( \hat{\alpha}_n, \hat{\beta}_n) := \text{argmin}_{\alpha>0, ~\alpha + 1/\beta \in (1/2,1)} 
\int_1^2 \left( S_{\alpha, \beta} (n,p) -   S_{\alpha, \beta} (p) \right)^2 dp. 
\end{align} 
In practice the integral in \eqref{scaleest} needs to be discretised.   
This approach is somewhat similar to the estimation method proposed in \cite{glt15}. 

If we are interested in the estimation of the self-similarity parameter  $H=\alpha+1/\beta \in (1/2,1)$,
then there is an alternative estimator based on a ratio statistic. Recalling that  $\beta \in (1,2)$, 
we deduce for any $p \in (0,1]$
\begin{align} \label{ratiostatistic}
R(n,p) := \frac{\sum_{i=2}^n |X_{\frac{i}{n}} - X_{\frac{i-2}{n}}|^p}{\sum_{i=1}^n |X_{\frac{i}{n}} - X_{\frac{i-1}{n}}|^p} \toop 2^{pH}  
\end{align}
by a direct application of Theorem~\ref{maintheorem}(ii). Thus, we immediately conclude that 
\[
\hat{H}_n := \frac{\log R(n,p)}{ p \log 2} \toop H. 
\]
This type of idea is rather standard in the framework of fractional Brownian motion with Hurst 
parameter $H$.  Theorem~\ref{sec-order}(i) suggests that the statistic  $\hat{H}_n$ has convergence rate $n^{1 - 1/(1-\alpha)\beta}$ when $p\in (0,1/2]$. By Theorem~\ref{sec-order}(ii)
this convergence rate can be improved to $\sqrt{n}$ when the first order increments are replaced by 
$k$th order increments, $k\geq 2$, in the definition of the statistic $R(n,p)$.

%%%%%%%%%%%%%
\section{Preliminaries} \label{secPrel}
%%%%%%%%%%%%%%%%¤¤¤¤¤¤¤%%

Throughout the following sections all positive constants will be denoted by $K$, although they may change from line to line. 
Also the notation might change from subsection to subsection, but the meaning will be clear from the context. Throughout all the next sections we assume, without loss of generality,  that $c_0=\delta=\sigma= 1$. Recall that  $g(t)=g_0(t)=0$ for all $t<0$ by assumption.

For a sequences of random variables  $(Y_n)_{n\in \N}$ defined on the probability space $(\Omega,\mathcal F,\P)$ 
we write $Y_n\stab Y$ if $Y_n$ converges $\f$-stably in law to $Y$.  That is, $Y$ is a random variable defined on an extension  of $(\Omega,\mathcal F,\P)$ such that 
for all $\f$-measurable random variables $U$ we have the joint convergence in law  
\[
(Y_n,U)  \schw (Y,U).
\] 
In particular, $Y_n\stab Y$ implies $Y_n\schw Y$.  For $A\in \f$ we will say that $Y_n\stab Y$ on $A$, if $Y_n\stab Y$ under $\P_{|A}$, where  $\P_{|A}$ denotes the conditionally probability measure $B\mapsto \P(B \cap A)/\P(A)$, when $\P(A)>0$.  We refer to the work \cite{Aldous,ren} for a detailed exposition of stable convergence. In addition,  $\toop$ will denote convergence in probability. We will  write $V(Y,p;k)_n=\sum_{i=k}^n |\Delta^n_{i,k} Y|^p$ when we want to stress that the power variation is built from a process $Y$. 
On the other hand, when  $k$ and $p$ are fixed we will sometimes write $V(Y)_n=V(Y,p;k)_n$ to simplify the notation. 

First of all, it follows from \cite[Theorem~7]{RajRos} that the process $X$ introduced in 
\eqref{def-of-X-43}  is well-defined if and only if for all $t\geq 0$, 
\begin{equation}\label{sdkfhsd;kf}
\int_{-t}^\infty \int_\R \Big(\big| f_t(s)x\big|^2\wedge 1\Big)\,\nu(dx)\,ds<\infty,
\end{equation}
where $f_t(s)=g(t+s)-g_0(s)$.
By adding  and subtracting  $g$ to $f_t$ it follows by assumption (A) and the  mean value theorem  that  $f_t\in L^\theta(\R_+)$ and $f_t$ is  bounded. 
For all $\epsilon>0$, assumption (A) implies that 
\begin{equation}\label{eq:74}
\int_{\R} (|y x|^2\wedge 1) \,\nu(dx)\leq K\Big( \1_{\{|y|\leq 1\}} |y|^\theta +\1_{\{|y|>1\}}|y|^{\beta+\epsilon}\Big), 
\end{equation}
which shows  \eqref{sdkfhsd;kf} since $f_t\in L^\theta(\R_+)$  is bounded.

Now, for all $n,i\in \N$, we set
\begin{align}\label{def-g-i-n}
g_{i,n}(x) ={}&  \sum_{j=0}^k (-1)^j \binom{k}{j} g\big((i-j)/n-x\big),\\ 
\label{def-h-i-n}
h_{i,n}(x) = {}& \sum_{j=0}^k (-1)^j \binom{k}{j} \big((i-j)/n-x\big)_+^\alpha,\\
\label{def-g-n}
g_n(x)={}& n^{\alpha}g(x/n), \qquad x\in \R.
\end{align}
In addition,  for each function $\phi\!:\R\to \R$ define  $D^k \phi\!:\R\to\R$  by 
\begin{align} \label{dkdef}
D^k \phi(x)=\sum_{j=0}^k (-1)^j \binom{k}{j} \phi(x-j),\qquad x\in \R.
\end{align}
In this notation the function $h_k$, defined in \eqref{def-h-13}, is given by $h_k=D^k \phi$ with $\phi: x\mapsto x_+^\alpha$.

%
%For $n,i\in \N$ and a function 
%$q:\R\to \R$ we define  
%\begin{equation} \label{def-g-i-n}
%q_{i,n}(x):= \sum_{j=0}^k (-1)^j \binom{k}{j} q\big((i-j)/n-x\big), \qquad q_n(x):=n^{\alpha}
%q(x/n), \qquad x\in \R.
%\end{equation}
%
%
%In addition let 
%\begin{equation}
%h(x)= x_+^\alpha,\qquad x\in \R. 
%\end{equation}
%In particular, for $q=g$ or $h_k$. 
%For simplicity we will write $h_{i,n}= (h_k)_{i,n}$ and $h_n=(h_k)_n$ suppressing the dependence of functions on $k$. 
%Recall the definition of the function $g_{i,n}$ and $h_{i,n}$ in  \eqref{def-g-i-n}.

\begin{lem} \label{helplem}
Assume that $g$ satisfies condition (A). Then we obtain the following estimates 
\begin{align}
\label{lemest1} |g_{i,n}(x)|&\leq K (i/n -x)^{\alpha}, \qquad x \in [(i-k)/n,i/n], \\
\label{lemest2} |g_{i,n}(x)|&\leq K n^{-k}((i-k)/n -x)^{\alpha-k} , \qquad x \in (i/n - 1, (i-k)/n), \\
\label{lemest3} |g_{i,n}(x)|&\leq K n^{-k}
\left( \1_{[(i-k)/n - 1, i/n - 1]}(x)  +
g^{(k)} ((i-k)/n -x) \1_{(-\infty, (i-k)/n - 1)} (x)  \right),\qquad  \\ 
& x \in (- \infty, i/n - 1].
\end{align}
The same estimates trivially hold for the function $h_{i,n}$. 
\end{lem}

\begin{proof}
The inequality \eqref{lemest1} follows directly from condition \eqref{kshs} of (A). The second inequality
\eqref{lemest2} is a straightforward consequence of Taylor expansion of order $k$ and the condition
$|g^{(k)}(t)|\leq K t^{\alpha-k}$ for $t\in (0,1)$. The third inequality \eqref{lemest3} follows again through Taylor expansion and the fact that the function $g^{(k)}$ is decreasing on $(1, \infty)$. 
\end{proof}

%%%%%%%%%%%%%
\section{Proof of Theorem \ref{maintheorem}}\label{proofs-w3lkhj}
%%%%%%%%%%%%%%
\setcounter{equation}{0}
\renewcommand{\theequation}{\thesection.\arabic{equation}}

In this section we will prove the assertions of Theorem \ref{maintheorem}.

%%%%%%%%%%%%%%%%%%
\subsection{Proof of Theorem~\ref{maintheorem}(i)} 
%%%%%%%%%%%%%%%%%%%%

The proof of Theorem~\ref{maintheorem}(i) is divided into the following three steps.
In Step~(i) we show Theorem~\ref{maintheorem}(i) for the compound Poisson case, which stands for the treatment of big jumps of $L$. 
Step (ii) consists of an approximating lemma, which proves that the small jumps of $L$ are asymptotically negligible.
Step (iii) combines the previous results to obtain the general theorem. 

Before proceeding with the proof  we will need the following preliminary lemma. Let $\{x\}:=x-\lfloor x\rfloor\in [0,1)$ denote the fractional
part of $x\in \R$. The lemma below seems to be essentially known (cf.\ \cite{Jacod-round-off,Tukey}), however, we have not been able 
to find this particular formulation. Therefore it is stated below for completeness. 
%For  the one dimensional case, see Tukey~\cite{Tukey}. 

\begin{lem}\label{frac-part-lem}
For $d\geq 1$ let $V=(V_1,\dots,V_d)$ be an absolutely continuous random vector in $\R^d$ with a density $v\!:\R^d\to \R_+$. Suppose that there exists an open convex set $A\subseteq \R^d$ such that $v$ is continuous differentiable on $A$ 
%(with a locally bounded gradient $\nabla v=(\partial v/\partial x_j)_{j=1}^k$?) 
and vanish outside $A$.  Then, as $n\to \infty$, 
\begin{equation}\label{sflj}
\big(\{n V_1\},\dots,\{n V_d\}\big)\stab U=\big(U_1,\dots,U_d\big)
\end{equation}
 where $U_1,\dots,U_d$ are independent   $\mathcal U([0,1])$-distributed random variables which are independent of $\f$. 
\end{lem}

\begin{proof}
For $x=(x_1,\dots,x_d)\in \R^k$ let $\{x\}=(\{x_1\},\dots,\{x_d\})$ be the fractional parts of its components. 
Let  $f:\R^d\times \R^d\to \R$ be a $C^1$-function, which vanishes outside some closed ball  in $A\times \R^d$.  We claim that 
for all $\rho>0$ there exists a constant 
$K>0$ such that  
\begin{align}\label{eq-est-f-g}
D_\rho:= \Big| \int_{\R^d}  f(x,\{ x/\rho \})\, v(x) \,dx-\int_{\R^k}\Big(\int_{[0,1]^d} f(x,u)\,du\Big)v(x)\,dx\Big|  \leq   K \rho. 
\end{align}
Indeed, by \eqref{eq-est-f-g} used for $\rho=1/n$ we obtain that 
\begin{equation}\label{sdkfh}
\E[f(V,\{n V\})]\to \E[f(V,U)]\qquad \text{as }n\to \infty,
\end{equation}
with $U=(U_1,\dots, U_d)$ given in the lemma. Moreover, due to \cite[Proposition~2(D'')]{Aldous},  \eqref{sdkfh}  implies the stable convergence  $\{n V\}\stab U$ as $n\to\infty$,  and the proof  is  complete. Thus, it only remains to prove the inequality \eqref{eq-est-f-g}. At this stage we use a similar 
technique as in  \cite[Lemma~6.1]{Jacod-round-off}. 
 Define   $\phi(x,u): =f(x,u)v(x)$. Then it holds by substitution that  
\begin{align}
\int_{\R^d}  f(x,\{  x/\rho \})v(x)\,dx= 
\sum_{j\in \Z^d} \int_{(0,1]^d} \rho^d \phi(\rho j + \rho u , u)\,du 
\end{align}
and
\begin{align}
\int_{\R^d}\Big(\int_{[0,1]^d} f(x,u)\,du\Big)v(x)\,dx={}& 
\sum_{j\in \Z^d}  \int_{[0,1]^d} \Big(\int_{(\rho j,\rho (j+1)]}\phi(x,u)\,dx\Big)\,du.
\end{align}
Hence, we conclude that  
\begin{align}
D_\rho\leq {}& \sum_{j\in \Z^d} \int_{(0,1]^d} \Big|\int_{(\rho j,\rho (j+1)]}\phi(x,u)\,dx-\rho^d \phi(\rho j + \rho u , u)\Big|\,du 
 \\  \leq {}& \sum_{j\in \Z^d} \int_{(0,1]^d} \int_{(\rho j,\rho (j+1)]}\Big|\phi(x,u)- \phi(\rho j + \rho u , u)\Big|\,dx\,du .
\end{align}
By mean value theorem there exists a positive constant $K$ and a compact set $B\subseteq \R^d\times \R^d$ such that for all $j\in \Z^d,\ x\in (\rho j,\rho(j+1)]$ and $u\in (0,1]^d$ we have 
\begin{equation}
\Big|\phi(x,u)- \phi(\rho j + \rho u , u)\Big|\leq K \rho \1_B(x,u). 
\end{equation}
Thus, 
$
D_\rho\leq K \rho \int_{(0,1]^d} \int_{\R^d}\1_{B}(x,u)\,dx\,du
$,
which shows \eqref{eq-est-f-g}. 
\end{proof}

\begin{proof}[Step~(i): The compound Poisson case]
Let $L=(L_t)_{t\in \R}$ be a compound Poisson process
and let $0\leq T_1<T_2<\ldots$ denote the jump times of the L\'evy process 
$(L_t)_{t\geq 0}$ chosen in increasing order.  Consider a fixed  $\epsilon>0$ and let $n\in \N$ satisfy $\varepsilon n  >4 k $. We define 
\begin{align}
\Omega_{\varepsilon}:= \Big\{\omega \in \Omega : {}& \text{for all }j\geq 1\text{ with $T_j(\omega)\in [0,1]$ we have  }|T_{j+1}(\omega )-T_{j}(\omega)|> \varepsilon/2 \\ {}& \text{and }\Delta L_s(\omega)=0\text{ for all }s\in [-\epsilon,\epsilon]\cup [1-\epsilon,1]\Big\}. 
% \text{ and for all $i=1,\dots,n$}\\ {}& \text{there is at most one $j$ such that } T_j(\omega)\in [(i-1)/n, i/n]\Big\}. 
\end{align}
Notice that $\P(\Omega_{\varepsilon}) \uparrow 1$ as  $\varepsilon \downarrow 0$. Now, we decompose for $i=k,\dots,n$
\[\label{rep-X-42}
\Delta_{i,k}^n X = M_{i,n,\varepsilon} + R_{i,n,\varepsilon}, 
\]
where 
\begin{align*}
M_{i,n,\varepsilon} = {}& 
\int_{\frac{i}{n}-\varepsilon}^{\frac{i}{n}} g_{i,n}(s)\,dL_s, \qquad \qquad R_{i,n,\varepsilon} = \int_{-\infty }^{\frac{i}{n}-\varepsilon} g_{i,n}(s)\,dL_s,
\end{align*}
and the function  $g_{i,n}$ is introduced in  \eqref{def-g-i-n}.
 The term $M_{i,n,\varepsilon}$ represents the dominating quantity, while $R_{i,n,\varepsilon}$ turns out to be negligible.

\noindent \textit{The dominating term:}
We claim that on $\Omega_\epsilon$ and as $n\to \infty$,
\begin{equation}\label{main-con}
 n^{\alpha p} \sum_{i=k}^n |M_{i,n,\epsilon}|^p\stab Z\qquad\text{where}\qquad Z=\sum_{m:\, T_m\in (0,1]} |\Delta L_{T_m}|^p V_m,
% \left( \sum_{l=0}^{\infty} |h(l+U_m)|^p \right)
\end{equation}
where $V_m$, $m\geq 1$, are  defined in \eqref{part1}.
To show \eqref{main-con} let  $i_m=i_m(\omega,n)$ denote the random index such that $T_m\in ((i_m-1)/n, i_m/n]$. The following representation will be crucial:  On $\Omega_{\varepsilon}$ we have that
\begin{align}
\label{cru-decomp}
{}& n^{\alpha p} \sum_{i=k}^n | M_{i,n,\varepsilon}|^p =  V_{n,\varepsilon}\qquad\qquad \text{with}\\
%V_{n,\varepsilon} (1) &= n^{\alpha p} \sum_{m} |\Delta L_{T_m}|^p | g(i_m/n - T_m) |^p \\[1.5 ex]
{}&  V_{n,\varepsilon}  = n^{\alpha p} \sum_{m:\,T_m\in (0,1]} |\Delta L_{T_m}|^p \left(
\sum_{l=0}^{[\varepsilon n]+v_m} |g_{i_m+l,n}(T_m)|^p \right)
\end{align}
for some random indexes $v_m=v_m(\omega,n,\epsilon)\in \{-2,-1,0\}$ which are measurable with respect to $T_m$.
Indeed,  on $\Omega_\epsilon$ and for each $i=k,\dots,n$, $L$ has at most one jump in  $(i/n-\epsilon/2,i/n]$. For each $m\in \N$ with $T_m\in (0,1]$ we have  $T_m\in (i/n-\epsilon,i/n]$ if and only if $i_m\leq i <n(T_m+\epsilon)$  (recall that $\epsilon n > 4k $). Thus,
\begin{equation}\label{sldfj}
\sum_{i\in \{k,\dots,n\}:\, T_m\in (i/n-\epsilon,i/n]} | M_{i,n,\varepsilon}|^p=|\Delta L_{T_m}|^p\left(\sum_{l=0}^{[\varepsilon n]+v_m} |g_{i_m+l,n}(T_m)|^p \right)
\end{equation}
for some $T_m$-measurable random variable $v_m\in \{-2,-1,0\}$.
Thus, by summing \eqref{sldfj} over all  $m\in \N$ with $T_m\in (0,1]$,   \eqref{cru-decomp} follows. In the following we will show that 
\begin{equation}\label{con-Z-V-n}
V_{n,\epsilon} \stab Z
%:=\sum_{m:\, T_m\in [0,1]} |\Delta L_{T_m}|^p \left( \sum_{l=0}^{\infty} | h(l+U_m)|^p \right)
\qquad \text{as }n\to\infty. 
\end{equation}
For $d\geq 1$ it is well-known that  the random vector 
$(T_1,\dots,T_d)$ is absolutely continuous with a $C^1$-density   
on  the open convex set $A:=\{(x_1,\dots,x_d)\in \R^d\!: 0<x_1<x_2<\dots<x_d\}$,  which is vanishing outside $A$. Thus, by Lemma~\ref{frac-part-lem}  we have  
\begin{equation}\label{eq-con-T_k-U}
(\{nT_m\})_{m\leq d}\stab (U_m)_{m\leq d}\qquad \text{as } n\to \infty
\end{equation}
where  
$(U_i)_{i\in \N}$ are  i.i.d.\ $\mathcal U([0,1])$-distributed random variables. By \eqref{kshs} we may write  $g(x)=x^\alpha_+ f(x)$  where $f\!:\R\to\R$  satisfies $f(x)\to 1$ as $x\downarrow 0$. 
By definition of $i_m$ we have that $\{n T_m\}=nT_m-(i_m-1)$ and therefore  
for all $l=0,1,2,\dots$ and $j=0,\dots,k$,  
\begin{align}\label{eq-sdlfj-1}
  {}&  n^\alpha g\Big(\frac{l+i_m-j}{n}-T_m\Big)= n^\alpha \Big(\frac{l+i_m-j}{n}-T_m\Big)^\alpha_+ f\Big(\frac{l+i_m-j}{n}-T_m\Big)\\ \label{eq-sdlfj-2}  {}& \qquad  =\Big(l-j+(i_m-nT_m)\Big)^\alpha_+ f\Big(\frac{l-j}{n}+n^{-1}(i_m-nT_m)\Big)
  \\ {}&\qquad = \Big(l-j+1-\{nT_m\}\Big)^\alpha_+ f\Big(\frac{l-j}{n}+n^{-1}(1-\{nT_m\})\Big).
\end{align}
By  \eqref{eq-con-T_k-U}, $(U_m)_{m\leq d}\eqschw (1-U_m)_{m\leq d}$ 
and $f(x)\to 1$ as $x\downarrow 0$  we obtain that  
\begin{equation}\label{con-g-U}
\Big\{n^\alpha g\Big(\frac{l+i_m-j}{n}-T_m\Big)\Big\}_{l,m\leq d}
\stab \Big\{\big(l-j+U_m\big)^\alpha_+\Big\}_{l,m\leq d}  \qquad \text{as } n\rightarrow \infty. 
\end{equation}
 Eq.~\eqref{con-g-U} implies that 
\begin{align}\label{stab-con-g-23}
\big\{n^\alpha  g_{i_m+l,n}(T_m)\big\}_{l,m\leq d} \stab \big\{h_k(l+U_m)\big\}_{l,m\leq d},
\end{align}
with $h_k$ being defined at \eqref{def-h-13}. Due to the $\f$-stable convergence in \eqref{stab-con-g-23}  we obtain 
by the continuous mapping theorem that for each fixed $d\geq 1$  and as $n\to \infty$, 
\begin{align}
{}& V_{n,\epsilon,d}:=n^{\alpha p} \sum_{m:\,m\leq d,\,T_m\in [0,1]} |\Delta L_{T_m}|^p \left(
\sum_{l=0}^{[\varepsilon d]+v_m} |g_{i_m+l,n}(T_m)|^p \right)
\\ {}& \qquad  \stab Z_{d}=\sum_{m:\, m\leq d,\, T_m\in [0,1]} |\Delta L_{T_m}|^p \left( \sum_{l=0}^{[\varepsilon d]+v_m} | h_k(l+U_m)|^p \right).\label{eq-erlkj}
\end{align}
Moreover, for $\omega\in \Omega$ we have as $d\to \infty$,
\begin{equation}
Z_{d}(\omega)\uparrow Z(\omega).
%\qquad \text{with}\qquad  Z:=\sum_{m:\, T_m\in [0,1]} |\Delta L_{T_m}|^p \left( \sum_{l=0}^{\infty} | h(l+U_m)|^p \right).
\end{equation}
Recall that  $|h_k(x)|\leq K (x-k)^{\alpha-k}$ for $x>k+1$, which implies that $Z<\infty$ a.s.\ since $p(\alpha-k)<-1$.
For all $l\in \N$ with $k\leq l\leq n$, we have 
\begin{equation}\label{eq:7365}
n^{\alpha p } |g_{i_m+l,n}(T_m)|^p \leq K  | l-k| ^{(\alpha-k)p },
\end{equation}
due to \eqref{lemest2} of Lemma \ref{helplem}. 
For all $d\geq 0$ set  $C_d=\sum_{m>d:\, T_m\in [0,1]} |\Delta L_{T_m}|^p$ and note that $C_d\to 0$ a.s.\ as $d\to \infty$ since $L$ is a compound Poisson process. 
By \eqref{eq:7365} we have  
\begin{align}
 |V_{n,\epsilon}- V_{n,\epsilon,d}|
\leq {}& K\Big( C_d+ C_0 \sum_{l=[\varepsilon d]-1}^{\infty} |l-k|^{p(\alpha-k)}\Big) \to 0 \qquad \text{as }d\to \infty 
%\sup_{m\in \N}  \left(n^{\alpha p} 
%\sum_{l=[\varepsilon d/2]-1}^{[\varepsilon n/2]} |g_{i_m+l,n}(T_m)|^p \right)
%\\ \leq {}& K C \sup_{m\in \N} \left(\sum_{l=[\varepsilon d/2]-1}^{[\varepsilon n/2]} \Big|n^{\alpha} n^{-k} \Big(\frac{i_m+l-k}{n}-T_m\Big)^{\alpha-k}\Big|^p\right) 
%\\  \leq {}& K C \sup_{m\in \N} \left(\sum_{l=[\varepsilon d/2]-1}^{[\varepsilon n/2]} \Big|\Big( l-k + n \big(\frac{i_m}{n}-T_m\big)\Big)^{\alpha-k}\Big|^p\right)
%\\ \leq {}& K C \sum_{l=[\varepsilon d/2]-1}^{\infty} |l-k|^{p(\alpha-k)}
%\to 0\qquad \text{as }d\to \infty
\end{align}
since $p(\alpha-k)<-1$. 
Due to the fact that  $n^{\alpha p} \sum_{i=k}^n | M_{i,n,\varepsilon}|^p=V_{n,\epsilon}$ a.s.\ on $\Omega_\epsilon$ and 
$V_{n,\epsilon}\stab Z$, it follows that $n^{\alpha p} \sum_{i=k}^n | M_{i,n,\varepsilon}|^p\stab Z$ on $\Omega_\epsilon$, since $\Omega_\epsilon\in \f$. This proves \eqref{main-con}.

\noindent 
\textit{The rest term:}
In the following we will show that 
\begin{equation}\label{rest-1}
n^{\alpha p} \sum_{i=k}^n | R_{i,n,\epsilon}|^p\toop 0\qquad \text{as } n\to \infty. 
\end{equation}
 The fact that  the random variables in \eqref{rest-1} are  usually not integrable makes the proof of \eqref{rest-1} considerably   more complicated.
Similarly to  \eqref{lemest3} of Lemma~\ref{helplem} we have that 
\begin{equation}\label{est-g-34}
n^k |g_{i,n}(s)| \1_{\{s\leq i/n-\epsilon\}}\leq K\big(\1_{\{s\in [-1,1]\}}+
\1_{\{s<-1\}} |g^{(k)}(-s)|\big)=:\psi(s) 
\end{equation}
where $K=K_\epsilon$. 
We will use the function $\psi$ several times in the proof of  \eqref{rest-1}, which will be divided into the two special cases $\theta\in (0,1]$ and $\theta\in (1,2]$. 

Suppose first that  $\theta\in (0,1]$. 
To show \eqref{rest-1} it suffices to prove  that 
\begin{equation}\label{we-want-52-1}
\sup_{n\in \N,\, i\in \{k,\dots,n\}} n^k|R_{i,n,\epsilon}|<\infty\qquad \text{a.s.}
\end{equation}
 since $\alpha<k-1/p$. 
 To show \eqref{we-want-52-1} 
we will  first prove that 
\begin{equation}\label{lshlfsjl}
\int_\R \int_{\R} \Big(|\psi(s)x|\wedge 1\Big)\,\nu(dx)\,ds<\infty.
\end{equation}
Choose $\tilde K$ such that $\psi(x)\leq \tilde K$ for all $x\in \R$. For $u\in [-\tilde K,\tilde K]$ we have that 
\begin{align} 
\int_{\R} \Big(|u x|\wedge 1\Big)\,\nu(dx) \leq   {}&   K \int_1^\infty \Big( |xu|\wedge 1 \Big)x^{-1-\theta}\,dx
\\ 
\leq {}& 
\begin{cases}
K |u|^\theta   & \theta\in (0,1) \\ 
K |u|^\theta\log(1/u)  & \theta =1, 
\end{cases}\label{lsfjdslfhl}
 \end{align}
 where we have used that $\theta\leq 1$. 
By \eqref{lsfjdslfhl} applied 
to $u=\psi(s)$ and assumption (A) it follows that \eqref{lshlfsjl} is satisfied. 
Since $L$ is a symmetric compound Poisson process we can find a Poisson random measure $\mu$ with compensator $\lambda \otimes \nu$ such that for all $-\infty<u<t<\infty$, $L_t-L_u=\int_{(u,t]\times \R} x\,\mu(ds,dx)$.
Due to \cite[Theorem~10.15]{k83}, \eqref{lshlfsjl} ensures the existence  of the stochastic 
integral $\int_{\R\times \R} |\psi(s)x| \,\mu(ds,dx)$. Moreover, $\int_{\R\times \R} |\psi(s)x| \,\mu(ds,dx)$ can be regarded as an $\omega$ by $\omega$  integral with respect to the measure $\mu_\omega$. 
Now, we have that 
\begin{align}\label{eq:883}
 | n^k R_{i,n,\epsilon} |\leq      \int_{(-\infty,i/n-\epsilon]\times \R}  \big| n^k g_{i,n}(s)x \big|\, \mu(ds,dx)  \leq \int_{\R\times \R} |\psi(s)x| \,\mu(ds,dx)<\infty, \qquad 
\end{align} 
which shows \eqref{we-want-52-1}, since the right-hand side of \eqref{eq:883} does not depend on $i$ and $n$.  
 
Suppose that  $\theta\in (1,2]$. 
Similarly as  before it suffices to show that 
\begin{equation}\label{we-want-52}
\sup_{n\in \N,\, i\in \{k,\dots,n\}} \frac{n^k|R_{i,n,\epsilon}| }{(\log n)^{1/q}}<\infty\qquad \text{a.s.}
\end{equation}
where $q>1$ denotes the conjugated number to $\theta>1$ determined by $1/\theta+1/q=1$.  
In the following we will show \eqref{we-want-52}  using the majorizing measure techniques developed in \cite{Marcus-Rosinski}. In fact, our  arguments are closely related to their Section~4.2.   
Set $T=\{(i,n)\!: n\geq k,\,i=k,\dots, n\}$. For $(i,n)\in T$ we have 
\begin{equation}
\frac{n^k|R_{i,n,\epsilon}| }{(\log n)^{1/q}}=\Big|\int_\R \zeta_{i,n}(s)\,dL_s\Big|, \qquad 
\zeta_{i,n}(s):=\frac{n^k}{(\log n)^{1/q}} g_{i,n}(s)\1_{\{s\leq i/n-\epsilon\}}.
\end{equation}
For $t=(i,n)\in T$ we will sometimes write $\zeta_t(s)$ for $\zeta_{i,n}(s)$. 
%Let $q>2$ be the conjugate number to $\beta\in (1,2)$, i.e.\ $1/\beta+1/q=1$. 
Let  $\tau\!: T\times T\to \R_+$ denote the metric given by 
\begin{equation}
\tau\big((i,n),(j,m)\big)=\begin{cases}
\log(n-k+1)^{-1/q}+\log(m-k+1)^{-1/q} & (i,n)\neq (j,l) \\ 0 & (i,n)=(j,l). 
\end{cases}
\end{equation}
Moreover, let $m$ be the probability measure on $T$ given by  $m(\{(i,n)\})=K n^{-3}$ for a suitable constant $K>0$.   
Set $B_\tau(t,r)=\{s\in T\!: \tau(s,t)\leq r\}$ for $t\in T$, $r>0$,  $D=\sup\{\tau(s,t)\!: s,t\in T\}$ and 
\begin{align}
I_q(m,\tau;D)=\sup_{t\in T}\int_0^D \Big(\log \frac{1}{m(B_\tau(t,r))}\Big)^{1/q} dr.
\end{align}
In the following we will show that $m$ is a so-called majorizing measure, which means that  
$I_q(m,\tau,D)<\infty$. 
For $r< (\log (n-k+1))^{-1/q}$ we have $B_\tau((i,n),r)=\{(i,n)\}$. Therefore, $m( B_\tau((i,n),r))=Kn^{-3}$ and 
\begin{align}\label{slfj-2}
 \int_0^{(\log(n-k+1))^{-1/q}}  \Big(\log \frac{1}{m(B_\tau((i,n),r))}\Big)^{1/q} dr
 = \int_0^{(\log(n-k+1))^{-1/q}} \Big(3\log n+\log K\Big)^{1/q}\,dr.
\end{align} 
For all $r\geq (\log(n-k+1))^{-1/q}$, $(k,k)\in B_\tau((i,n),r)$  and hence $m(B_\tau((i,n),r))\geq m(\{(k,k)\})=K (k+1)^{-3}$. Therefore, 
\begin{align}\label{slfj-3}
{}& \int_{(\log(n-k+1))^{-1/q}}^D  \Big(\log \frac{1}{m(B_\tau((i,n),r))}\Big)^{1/q} \,dr 
\\ {}& \qquad 
\leq
\int_{(\log(n-k+1))^{-1/q}}^D  \Big(3\log(k+1)+\log K\Big)^{1/q}\,dr.
\end{align}
By \eqref{slfj-2} and \eqref{slfj-3} it follows that 
$I_q(m,\tau,D)<\infty$. 
For  $(i,n)\neq (j,l)$ we have that 
\begin{align}\label{est-g-35}
\frac{| \zeta_{i,n}(s)-\zeta_{j,l}(s)|}{\tau\big((i,n),(j,l)\big)}\leq {}& 
n^k |g_{i,n}(s)| \1_{\{s\leq i/n-\epsilon\}}+l^k |g_{j,l}(s)| \1_{\{s\leq j/l-\epsilon\}}\leq K \psi(s).
%\\ \leq {}& 2C_2 \Big(\1_{\{s\in [-1,1]\}}+
%\1_{\{s<-1\}} |g^{(k)}(-s)|\Big).
\end{align}
Fix $t_0\in T$ and consider the following Lipschitz
type norm of $\zeta$, 
\begin{equation}
\|\zeta \|_\tau(s)= D^{-1}|\zeta_{t_0}(s)|+ \sup_{\substack{t_1,t_2\in T:\\ \, \tau(t_1,t_2)\neq 0}} \frac{|\zeta_{t_1}(s)-\zeta_{t_2}(s)|}{\tau(t_1,t_2)}.
\end{equation}
By \eqref{est-g-35} it follows  that 
$\| \zeta \|_{\tau}(s)\leq K  \psi(s) $
and hence   
\begin{equation}\label{eq-tau-norm}
\int_\R \|\zeta \|^\theta_\tau(s)\,ds\leq K\Big(2+\int_1^\infty |g^{(k)}(s)|^\theta\,ds\Big)<\infty.
\end{equation}
By \cite[Theorem~3.1, Eq.~(3.11)]{Marcus-Rosinski} together with $I_q(m,\tau,D)<\infty$ and \eqref{eq-tau-norm} we deduce  \eqref{we-want-52}, which completes the proof of \eqref{rest-1}.

\noindent
\textit{End of the proof:}
 Recall the decomposition $\Delta^n_{i,n}X= M_{i,n,\epsilon}+R_{i,n,\epsilon}$ in  \eqref{rep-X-42}. Eq.~\eqref{main-con}, \eqref{rest-1} and an application of Minkowski inequality yield that  
 \begin{equation}\label{eqyert}
n^{\alpha p} V(p;k)_n\stab Z\qquad \text{ on $\Omega_\epsilon$ as $n\to \infty$.}
 \end{equation}
 Since $\P(\Omega_\epsilon)\uparrow 1$ as $\epsilon\downarrow 0$, \eqref{eqyert} implies that 
 \begin{equation}
n^{\alpha p} V(p;k)_n\stab  Z. 
 \end{equation}
We have now completed the proof for a particular  choice 
of stopping times $(T_m)_{m\geq 1}$. However, the result remains valid for any choice of $\F$-stopping times, since the distribution 
of $Z$ is invariant with respect to reordering of stopping times. 
\end{proof}

\noindent
{\em Step~(ii): An approximation.}
 To prove Theorem~\ref{maintheorem}(i) in the general case we need the following approximation result.
Consider a general symmetric L\'evy process $L=(L_t)_{t\in \R}$ as in Theorem~\ref{maintheorem}(i)
and let $N$ be the corresponding Poisson random measure $N(A):= \sharp\{t: (t,\Delta L_t)\in A\}$ for all measurable  $A\subseteq \R\times (\R\setminus\{0\})$.
%
%
% For a fixed $j\in \N$ let  $L(j)=(L_t(j))_{t\in \R}$ denote the truncated   L\'evy process defined by 
%\begin{align}\label{trunca-sdfl}
%L_t(j)-L_s(j)={}& L_t-L_s-\sum_{u\in (s,t]} \Delta L_u\1_{\{|\Delta L_u|>\frac{1}{j}\}}, \qquad s<t.
%\end{align}
%We have that 
%\begin{equation*}
%\E[e^{i u  L_1(j)}]=\exp\Big(\int_{|x|\leq 1/j} \big(e^{i u x} -1-ix u \big)\,\nu(dx)\Big),\qquad u\in\R.
%\end{equation*}
By our assumptions (in particular, by  symmetry),  the  process $X(j)$ given by 
\begin{equation}\label{def-X-jaj}
X_t(j)=\int_{(-\infty,t]\times [-\frac{1}{j},\frac{1}{j}]} \big\{(g(t-s)-g_0(-s))x \big\}\,N(ds,dx)
\end{equation}
is well-defined. The following estimate on the processes $X(j)$ will be  crucial.

\begin{lem}\label{lem-1}
Suppose that   $\alpha<k-1/p$  and  $\beta<p$. 
%$\int_{|x|\leq 1} |x|^p\,\nu(dx)<\infty$ (in particular, if $p>\beta$). 
Then  
\begin{equation*}
\lim_{j\to \infty} \limsup_{n\to \infty}\P\big(n^{\alpha p}V(X(j))_n>\epsilon\big)=0\qquad \text{for all }\epsilon>0.
\end{equation*}
\end{lem}

\begin{proof}
 By Markov's inequality and the stationary increments of $X(j)$ we have that 
\begin{align}
\P\big(n^{\alpha p}V(X(j))_n>\epsilon\big)
 \leq \epsilon^{-1}n^{\alpha p}  \sum_{i=k}^n \E [|\Delta^n_{i,k} X(j)|^p]\leq  \epsilon^{-1}n^{\alpha p +1} \E[|\Delta_{k,k}^n X(j)|^p].
\end{align}
Hence, it is enough to show that 
\begin{equation}\label{con-Y}
\lim_{j\to \infty} \limsup_{n\to \infty}  \E[|Y_{n,j} |^p] = 0\qquad \text{with}\quad Y_{n,j}:=n^{\alpha +1/p} \Delta^n_{k,k} X(j).
\end{equation}
To show \eqref{con-Y} it suffices  to prove that
\begin{align}
\label{def-h-2}
{}& \lim_{j\to \infty} \limsup_{n\to \infty} \xi_{n,j}= 0\qquad \text{where}\qquad   \xi_{n,j}=   \int_{|x|\leq 1/j} \chi_n(x)\,\nu(dx)\quad \text{and} \\ \label{def-h}
{}& \chi_n(x)=\int_{-\infty}^{k/n} \Big(|n^{\alpha+1/p} g_{k,n}(s)x|^p\1_{\{|n^{\alpha+1/p} g_{k,n}(s)x |\geq 1\}}
\\ {}& \phantom{\chi_n(x)=\int_{-\infty}^{k/n} \Big(  }
+|n^{\alpha+1/p} g_{k,n}(s)x|^2\1_{\{|n^{\alpha+1/p} g_{k,n}(s)x |\leq 1\}}\Big)\,ds,
\end{align}
which follows from the  representation  
\begin{equation*}
Y_{n,j}=\int_{(-\infty,k/n]\times  [-\frac{1}{j},\frac{1}{j}]}\big(n^{\alpha+1/p} g_{k,n}(s)x\big)\, N(ds,dx),
\end{equation*}
and   by   \cite[Theorem~3.3 and the remarks  above it]{RajRos}. 
Suppose for the moment that there exists a finite constant $K>0$ such that 
\begin{equation}\label{est-h-ewrlj}
\chi_n(x)\leq  K (|x|^p+x^2) \qquad \text{for all } x\in [-1,1].
\end{equation}
Then, 
\begin{equation}\label{last-est-sldfj}
\limsup_{j\to\infty}\big\{\limsup_{n\to \infty} \xi_{n,j}\big\}\leq K \limsup_{j\to\infty}\int_{|x|\leq 1/j} (|x|^p+x^2)\,\nu(dx)=0
\end{equation}
since $p>\beta$. Hence it suffices to show the estimate \eqref{est-h-ewrlj}, which we will do in the following.

Let $\Phi_p\!:\R\to\R_+$ denote the function  $\Phi_p(y)= |y|^2\1_{\{|y|\leq 1\}} +|y|^p\1_{\{|y|> 1\}}$.
We split $\chi_n$ into the following three terms which need different treatments
 \begin{align}
  \chi_n(x) {}& = \int_{-k/n}^{k/n} \Phi_p\Big( n^{\alpha+1/p} g_{k,n}(s)x\Big)\,ds +\int_{-1}^{-k/n} \Phi_p\Big(n^{\alpha+1/p} g_{k,n}(s)x\Big)\,ds \\{}
 {}&\phantom{=}+\int_{-\infty}^{-1} \Phi_p\Big(n^{\alpha+1/p} g_{k,n}(s) x\Big)\,ds \\ {}
{}& =:   I_{1,n}(x)+I_{2,n}(x)+I_{3,n}(x).
 \end{align}
%We claim that 
%\begin{equation}\label{claim-2est}
%\int_{-\infty}^{k/n} |xf_n(s)|^2\1_{\{|xf_n(s)|\leq 1\}}\,ds\leq c |x|^p.
%\end{equation}
%To  show \eqref{claim-2est}  we  split the integral on the left-hand side  into three parts. 
\emph{Estimation of $I_{1,n}$}: By \eqref{lemest1} of Lemma \ref{helplem} we have that 
\begin{equation}\label{rep-g-slfj}
|g_{k,n}(s)|\leq K (k/n-s)^\alpha, \qquad s\in [-k/n,k/n].
\end{equation}  Since  $\Phi_p$ is increasing on $\R_+$, 
 \eqref{rep-g-slfj} implies that  
\begin{align}\label{est-v_n-1}
I_{1,n}(x)\leq K \int_{0}^{2k/n}\Phi_p\Big(x n^{\alpha+1/p}s^\alpha\Big)\,ds.
\end{align} 
By basic calculus it follows that 
\begin{align}
{}& \int_{0}^{2k/n} |xn^{\alpha+1/p}s^\alpha |^2\1_{\{|xn^{\alpha+1/p}s^\alpha|\leq 1\}}\,ds \nonumber
\\  
{}& \quad\leq K \Big( \1_{\{|x|\leq (2k )^{-\alpha}n^{-1/p}\}} x^2 n^{2/p-1} + \1_{\{|x|>(2k )^{-\alpha}n^{-1/p}\}} |x|^{-1/\alpha} n^{-1-1/(\alpha p)}\Big)
\\ {}& \quad\leq K (|x|^p +x^2). \label{lhsflh}
\end{align}
%where the last inequality is obtained by considering the case $p\leq 2$ and $p>2$ separately. 
%The inequalities 
%\begin{align}\label{lhsdfgh-1}
%{}& \1_{\{|x|\leq (2k )^{-\alpha}n^{-1/p}\}} x^2 n^{2/p-1} \leq K
%\begin{cases}  \1_{\{|x|\leq (2k )^{-\alpha}n^{-1/p}\}} |x|^p & p\leq 2 \\ 
%\1_{\{|x|\leq (2k )^{-\alpha}n^{-1/p}\}} x^2 &  p>2,
%\end{cases}
%\intertext{and} \label{lhsdfgh-2}
%{}& \1_{\{|x|>(2k )^{-\alpha}n^{-1/p}\}} |x|^{-1/\alpha} n^{-1-1/(\alpha p)}\leq K |x|^p,
%\end{align}
%show  that \eqref{lhsflh} is less than or equal to 
%\begin{align}\label{lhgsg}
% {}&  K \Big( \1_{\{|x|\leq (2k )^{-\alpha}n^{-1/p}\}} (|x|^p+x^2) + \1_{\{|x|>(2k )^{-\alpha}n^{-1/p}\}} |x|^{p} \Big)\leq K(|x|^p+x^2).
%\end{align}
Moreover, 
\begin{align}\label{est-h-v2}
{}&  \int_{0}^{2k/n} |x  n^{\alpha+1/p} s^\alpha|^p \1_{\{|x  n^{\alpha+1/p} s^\alpha|>1\}}\,ds\leq \int_{0}^{2k/n} |x  n^{\alpha+1/p} s^\alpha|^p 
  \,ds 
  \leq K |x|^p. 
\end{align}
Combining \eqref{est-v_n-1}, \eqref{lhsflh} and \eqref{est-h-v2}  show the estimiate 
$I_{1,n}(x)\leq K(|x|^p+x^2)$.
 
 \noindent
\emph{Estimation of $I_{2,n}$}:
By \eqref{lemest2} of Lemma \ref{helplem} it holds  that 
\begin{align}\label{est-ljsdflj}
|g_{k,n}(s)|\leq K n^{-k} |s|^{\alpha-k}, \qquad s\in (-1,-k/n). 
\end{align}
Again, due to the fact that $\Phi_p$ is increasing on $\R_+$, \eqref{est-ljsdflj} implies that 
\begin{align}\label{kjhsfg-231}
 I_{2,n}(x) \leq K \int_{k/n}^{1} \Phi_p(x n^{\alpha+1/p-k}
s^{\alpha-k})\,ds.
\end{align}
For  $\alpha\neq k-1/2$ we have  
\begin{align}
 {}& \int_{k/n}^{1} |xn^{\alpha+1/p-k} s^{\alpha-k}|^2\1_{\{|xn^{\alpha+1/p-k}s^{\alpha-k}|\leq 1\}}\,ds \nonumber
\\ {}&\quad \leq 
 K \Big(x^2 n^{2(\alpha+1/p-k)} 
+\1_{\{|x|\leq n^{-1/p}k^{-(\alpha-k)}\}} |x|^2 n^{2/p-1} 
 \\ {}& \qquad \qquad + 
 \1_{\{|x|> n^{-1/p}k^{-(\alpha-k)}\}}
|x|^{1/(k-\alpha)}n^{1/(p(k-\alpha))-1}\Big)
 \\ \label{lsdjflsh-2}{}& \quad \leq K \Big(x^2+|x|^p\Big),
\end{align}
where we have used that  $\alpha< k-1/p$. 
 For $\alpha=k-1/2$ we have 
\begin{align}
{}& \int_{k/n}^{1} |xn^{\alpha+1/p-k} s^{\alpha-k}|^2\1_{\{|xn^{\alpha+1/p-k}s^{\alpha-k}|\leq 1\}}\,ds \nonumber
\\ \label{lsdjflsh-4}{}&\qquad \leq 
x^2 n^{2(\alpha+1/p-k)} \int_{k/n}^1 s^{-1} \,ds=x^2 n^{2(\alpha+1/p-k)}\log(n/k)\leq  K x^2,
\end{align}
where we again have used   $\alpha<k-1/p$ in the last inequality. 
Moreover, 
\begin{align}\label{lsdjflsh-5}
{}& \int_{k/n}^{1} |xn^{\alpha+1/p-k} s^{\alpha-k}|^p\1_{\{|xn^{\alpha+1/p-k}s^{\alpha-k}|> 1\}}\,ds 
\\ {}& \qquad \leq 
K |x|^p n^{p(\alpha+1/p-k)} \Big(1+(1/n)^{p(\alpha-k)+1}\Big)\leq K |x|^p.  \label{lsdjflsh-6}
\end{align}
By \eqref{kjhsfg-231}, \eqref{lsdjflsh-2}, \eqref{lsdjflsh-4} and \eqref{lsdjflsh-6} we obtain the estimate $I_{2,n}(x)\leq K(|x|^p+x^2)$. 

\noindent
\emph{Estimation of $I_{3,n}$}:
For $s<-1$ we have that 
$
|g_{k,n}(s)|\leq K n^{-k} |g^{(k)}(-k/n-s)|,
$
by \eqref{lemest3} of Lemma \ref{helplem},
and hence 
\begin{align}\label{klhshlkslh}
{}&I_{3,n}(x)
\leq  K \int_{1}^{\infty} 
\Phi_p\big(n^{\alpha+1/p-k} g^{(k)}(s)\big)\,ds.
\end{align}
We have that 
\begin{equation}\label{iuyhswrhjl}
 \int_{1}^{\infty} |xn^{\alpha+1/p-k} g^{(k)}(s)|^2\1_{\{|xn^{\alpha+1/p-k}g^{(k)}(s)|\leq 1\}}\,ds
 %\\ \label{iuyhswrhjl-1}{}& \qquad 
 \leq 
 x^2 n^{2(\alpha+1/p-k)} \int_{1}^\infty |g^{(k)}(s)|^2\,ds.
\end{equation}
Since $|g^{(k)}|$ is decreasing on $(1,\infty)$ and $g^{(k)}\in L^\theta((1,\infty))$ for some $\theta\leq 2$, the integral on the right-hand side of \eqref{iuyhswrhjl} is finite. For $x\in [-1,1]$ we have  
\begin{align}
{}& \int_{1}^{\infty} |xn^{\alpha+1/p-k} g^{(k)}(s)|^p\1_{\{|xn^{\alpha+1/p-k}g^{(k)}(s)|> 1\}}\,ds \nonumber
 \\ {}& \qquad \leq 
 |x|^p n^{p(\alpha+1/p-k)} \int_{1}^\infty |g^{(k)}(s)|^p\1_{\{|g^{(k)}(s)|> 1\}}\,ds. \label{skdhhfho}
\end{align}
From  our assumptions it follows that the integral in \eqref{skdhhfho} is finite. By \eqref{klhshlkslh}, \eqref{iuyhswrhjl} and \eqref{skdhhfho} 
we have  that 
$I_{3,n}(x)\leq K (|x|^p+x^2)$ for all $x\in [-1,1]$,
which completes the proof of \eqref{est-h-ewrlj} and therefore also the proof of the lemma.  
\end{proof}

\noindent
{\em Step~(iii): The general case.}
In the following we will prove Theorem~\ref{maintheorem}(i) in the general case  by combining the above Steps~(i) and (ii). 

 \begin{proof}[Proof of Theorem~\ref{maintheorem}(i)]
Let $(T_m)_{m\geq 1}$ be a sequence of $\F$-stopping times that exhausts the jumps of $(L_t)_{t\geq 0}$. For each $j\in \N$ let $\hat L(j)$ be the L\'evy process given by 
\begin{equation}
\hat L_t(j)-\hat L_u(j)=\sum_{u\in (s,t]} \Delta L_u\1_{\{|\Delta L_u|>\frac{1}{j}\}},\qquad s<t,
\end{equation}
and set 
\begin{equation}
\hat X_t(j)= \int_{-\infty}^t \big(g(t-s)-g_0(-s)\big)\,d\hat L_s(j). 
\end{equation}
Moreover, set 
\begin{equation}
T_{m,j}=\begin{cases}
T_{m} & \text{if }|\Delta L_{T_m}|>\frac{1}{j} \\ \infty & \text{else}, 
\end{cases}
\end{equation}
and note that  $(T_{m,j})_{m\geq 1}$ is a sequence of $\F$-stopping times that exhausts the jumps of  $(\hat L_t(j))_{t\geq 0}$.
Since  $\hat L(j)$ is a compound Poisson process,  Step~(i) shows that 
\begin{align}\label{app-12}
n^{\alpha p}V(\hat X(j))_n \stab Z_j:=\sum_{m:\,T_{m,j}\in [0,1]} |\Delta \hat L_{T_{m,j}}(j)|^p V_{m} \qquad \text{as } n\rightarrow \infty,
\end{align}
where $V_m$, $m\geq 1$, are defined in \eqref{part1}. 
By definition of $T_{m,j}$ and monotone convergence we have as $j\to \infty$, 
\begin{equation}\label{app-22}
Z_j=\sum_{m:\,T_{m}\in [0,1]} |\Delta L_{T_{m}}|^p V_{m} \1_{\{|\Delta L_{T_m}|>\frac{1}{j}|\}} \toas \sum_{m:\,T_{m}\in [0,1]} |\Delta L_{T_{m}}|^p V_m=:Z.
\end{equation}
Suppose first that $p\geq 1$ and decompose
 \begin{align}
\big(n^{\alpha p} V(X)_n\big)^{1/p}{}& = \big(n^{\alpha p}V(\hat X(j))_n\big)^{1/p} +\Big(\big(n^{\alpha p}V(X)_n\big)^{1/p}-\big(n^{\alpha p}V(\hat X(j))_n\big)^{1/p}\Big)
\\ {}& =: Y_{n,j}+U_{n,j}.\label{eq:746}
\end{align}
Eq.~\eqref{app-12} and \eqref{app-22} show 
\begin{equation}\label{eq:3456}
Y_{n,j}\xrightarrow[n\to \infty]{\mathcal{L}-s} Z_j^{1/p}\qquad \text{and}\qquad Z_j^{1/p}\xrightarrow[j\to \infty]{\P} Z^{1/p}. 
\end{equation}
Note that $X-\hat X(j)=X(j)$, where $X(j)$ is defined in \eqref{def-X-jaj}. For all $\epsilon>0$ we have by Minkowski's inequality  
\begin{align}\label{eq:63653}
{}& \limsup_{j \to \infty}\limsup_{n\to \infty} \P\big( |U_{n,j}|>\epsilon\big)  \leq 
\limsup_{j \to \infty}\limsup_{n\to \infty} \P\big(n^{\alpha p}V(X(j))_n>\epsilon^p\big)=0,
\end{align}
where the last equality follows by Lemma~\ref{lem-1}. By a standard argument, see e.g.\ \cite[Theorem~3.2]{Billingsley}, \eqref{eq:3456} and \eqref{eq:63653} implies that  
$(n^{\alpha p} V(X)_n)^{1/p}\stab Z^{1/p}$ which completes the proof of Theorem~\ref{maintheorem}(i) when $p\geq 1$. For  $p<1$, Theorem~\ref{maintheorem}(i) follows by \eqref{app-12}, \eqref{app-22}, the inequality  
$| V(X)_n- V(\hat X(j))_n |\leq V(X(j))_n$  and \cite[Theorem~3.2]{Billingsley}. 
\end{proof}

%%%%%%%%%%%
\subsection{Proof of Theorem~\ref{maintheorem}(ii)}\label{main-thm-2}
%%%%%%%%%%%%%%%%

Suppose that $\alpha <k-1/\beta$, $p<\beta$ and $L$ is a symmetric $\beta$-stable L\'evy proces. In the proof of Theorem~\ref{maintheorem}(ii) we will use  the following notation: For all $n\geq 1$, $r\geq 0$  set
%\label{def-rho}
\begin{equation}\label{eq:23412343}
\phi_{r}^n(s)= D^kg_n(r-s),\qquad \phi_r^\infty(u)=h_k(r-u),
\end{equation}
where 
$g_n$ and $D^k$ are defined at \eqref{def-g-n} and \eqref{dkdef}, and the function $h_k$ is defined in \eqref{def-h-13}.
% = n^{\alpha} \Big( g\big(\frac{r-s}{n}\big)-g\big(\frac{r-1-s}{n}\big)\Big),\\ 
% ,
For all $n\in \N\cup\{\infty\}$ and $t\geq 0$ set 
\begin{equation}\label{eq:232452}
  Y_{t}^n=  \int_{-\infty}^t \phi_{t}^n(s)\,dL_s.
\end{equation}
By self-similarity of $L$ of index $1/\beta$ we have for all $n\in \N$, 
\begin{equation}\label{self-similar}
\{n^{\alpha+1/\beta} \Delta^n_{i,k} X\!:i=k,\dots,n\}\eqschw \{Y_{i}^n\!:i=k,\dots,n\},
\end{equation}
%where 
%\begin{equation} \label{gndef} 
%V_{i,n}= \int_{-\infty}^i    D^k g_{n}(i-s)\,dL_s,
%%=\phi(x)f(x/n) 
%\end{equation}
  where  $\eqschw$ means equality in distribution. 
For $\alpha<1-1/\beta$,  $Y^\infty$ is the $k$-order increments of a linear fractional stable motion. For $\alpha\geq 1-1/\beta$ the linear fractional stable motion is not well-defined, but $Y^\infty$ is well-defined since the function $h_k$ is locally bounded and satisfies  $|h_k(x)|\leq K x^{\alpha-k}$ for all $x\geq k+1$,  
which implies that $h_k\in L^\beta(\R)$. We are now ready to prove Theorem~\ref{maintheorem}(ii), which is done by approximate  $Y^n_t$ by $Y^\infty_t$ and applying  the   ergodic properties of $Y^\infty_t$. 

%
%the process $V=(V_t)_{t\geq 0}$ given by 
%\begin{equation}\label{def-pro-Ysf}
%V_t=\int_{-\infty}^t h_k(t-s)\,dL_s
%\end{equation}
%to approximate the scaled version of  $k$-order increments of $X$.

To show that $Y_{k}^n\to Y^\infty_k$  in $L^p$ as $n\to \infty$ we  recall  that for $\phi:s\mapsto s^\alpha_+$ we have $D^k \phi=h_k \in L^\beta(\R)$.
For $s\in \R$ let  
$\psi_n(s)=g_n(s)-s_+^{\alpha}$. Since $p<\beta$,   
\begin{align}\label{est-lkhgwkg}
\E[|Y_{k}^n-Y_k^\infty|^p]=K \Big(\int_0^\infty | D^k \psi_n(s)  |^\beta \,ds\Big)^{p/\beta}.
\end{align}
To show that the right-hand side of \eqref{est-lkhgwkg} converge to zero we   note that 
\begin{align}\label{sdkkls}
{}& \int_{n+k}^\infty |D^k g_n(s)|^\beta\,ds \leq K n^{\beta(\alpha-k)} 
\int_{n+k}^\infty |g^{(k)}\big((s-k)/n\big)|^\beta\,ds
\\ \label{sdkkls-1} {}& \qquad = 
K n^{\beta(\alpha-k)+1} \int_{1 }^\infty |g^{(k)}(s)|^\beta\,ds\to 0\qquad \text{as }n\to \infty,
\end{align}
which implies that 
\begin{align}\label{skdhks}
 \int_{n+k}^\infty |D^k \psi_n(s)|^\beta\,ds\leq {}& K\Big( \int_{n+k}^\infty |D^k g_n(s)|^\beta\,ds+\int_{n+k}^\infty |D^k \phi(s)|^\beta\,ds\Big)\xrightarrow[n\to\infty]{} 0.
\end{align}
By \eqref{lemest2} of Lemma \ref{helplem} it holds that 
\begin{equation}
|D^k g_n(s)|\leq  K(s-k)^{\alpha-k}
\end{equation}
for  $s\in (k+1,n)$.
Therefore,   for $s\in (0,n]$ we have 
\begin{equation}\label{dom-slfj}
|D^k\psi_n(s)|\leq K \big(\1_{\{s\leq k+1\}}+\1_{\{s>k+1\}}(s-k)^{\alpha-k}\big),  
\end{equation}
where  the function on the right-hand side of \eqref{dom-slfj} is in $L^{\beta}(\R_+)$. 
For fixed $s\geq 0$, $\psi_n(s)\to 0 $ as $n\to \infty$ by  assumption \eqref{kshs}, and hence $D^k \psi_n(s)\to 0$ as $n\to\infty$. By \eqref{dom-slfj} and the dominated convergence theorem this shows that 
\begin{equation}\label{con-low}
\int_0^{ n} |D^k \psi_n(s)|^\beta\,ds\to 0. 
\end{equation}
By \eqref{est-lkhgwkg}, \eqref{skdhks} and \eqref{con-low}   we have 
\begin{equation}\label{eq:12380273}
 \E[|Y_{k}^n-Y^\infty_k|^p]\to 0 \qquad \text{as } n\to \infty,
\end{equation}
 which implies that 
  \begin{equation}\label{con-L-1}
\E\Big[ \frac{1}{n} \sum_{i=k}^n |Y_{i}^n-Y^\infty_i|^p\Big]=  \frac{1}{n} \sum_{i=k}^n \E[|Y_{i}^n-Y_i^\infty |^p]\leq \E[|Y_{k}^n-Y_k^\infty|^p]\to 0 
 \end{equation} 
 as $n\to \infty$. 
Moreover,   $(Y_t^\infty)_{t\in \R}$ is mixing since it is a symmetric stable moving average, see e.g.\ \cite{Ergodic-Cam}. This implies, in particular,  that the discrete time stationary sequence $\{Y_j\}_{j\in \mathbbm Z}$ is mixing and hence ergodic. According to Birkhoff's ergodic theorem (cf.\  \cite[Theorem~10.6]{k83}) 
\begin{equation}\label{Ergodic}
\frac{1}{n} \sum_{i=k}^n |Y_i^\infty|^p \toas \E[|V_k|^p]=:m_p \in (0,\infty)\qquad \text{as }n\to \infty.
\end{equation}
We note that $m_p$ defined at \eqref{Ergodic} coincide with the definition 
 in  Theorem \ref{maintheorem}(ii), cf.\  \cite[Property~1.2.17 and 3.2.2]{SamTaq}.  
By  \eqref{con-L-1}, Minkowski's inequality and \eqref{Ergodic}, we deduce 
\begin{equation}\label{con-gsdfhk}
\frac{1}{n}\sum_{i=k}^n |Y_{i}^n|^p\toop m_p \qquad \text{as }n\to\infty.
\end{equation}
By  \eqref{self-similar} it shows that 
\begin{align}
 n^{-1+p(\alpha+1/\beta)} V(X)_n={}& \frac{1}{n} \sum_{i=k}^n |n^{\alpha+1/\beta} \Delta^n_{i,k} X|^p   \eqschw \frac{1}{n}\sum_{i=k}^n |Y^n_i|^p\toop m_p
\end{align}
as $n\to \infty$.  This completes the proof of Theorem~\ref{maintheorem}(ii).  \qed

%%%%%%%%%%%%%
\subsection{Proof of Theorem~\ref{maintheorem}(iii)}
%%%%%%%%%%%%%%

We will derive Theorem~\ref{maintheorem}(iii) from   the  two lemmas below. For $k\in \N$ and $p\in [1,\infty)$ let $W^{k,p}$ denote the Wiener space of functions $\zeta\!:[0,1]\to\R$ which are $k$-times differentiable with  $\zeta^{(k)}\in L^p([0,1])$ where $\zeta^{(k)}(t)=\partial^k \zeta(t)/\partial t^k$ $\lambda$-a.s. First we  will show that, under the conditions in Theorem~\ref{maintheorem}(iii),    $X\in W^{k,p}$ almost surely.

\begin{lem}\label{abs-cont-sdf}
Suppose  that $p\neq \theta$, $p\geq 1$ and (A) holds.
% $g$ is $k$-times continuous differentiable on $(0,\infty)$ and for all $j=1,\dots, k$, $|g^{(j)}(t)|\leq c t^{\alpha-j}$ for all $t\in (0,1)$,  $g^{(j)}\in L^\theta((1,\infty))$ and $|g^{(j)}|$ is decreasing on $(1,\infty)$. 
 If $\alpha>k-1/(p\vee \beta)$ then  
 \begin{equation}\label{sdhjlshjl}
X\in W^{k,p}\text{ a.s.}\qquad \text{and}\qquad \frac{\partial^{k}}{\partial t^k} X_t=\int_{-\infty}^t g^{(k)}(t-s)\,dL_s\qquad \lambda\otimes \P\text{-a.s.}
 \end{equation}
 Eq.~\eqref{sdhjlshjl} remains valid for  $p=\theta$  if, in  addition, (A-log) holds. 
 % $\int_1^\infty |g^{(j)}(s)|^\theta \log(1/|g^{(j)}(s)|)\,ds<\infty$ for all $j=1,\dots,k$. 
 % \begin{align}\label{klhsdfgw}
%{}& X_t=\int_0^t \cdots\int_0^{s_{3}}\Big(\int_0^{s_{2}} F_{s_1}\,d s_1\Big)\,ds_{2}\cdots ds_k\qquad \text{where} \\ \label{klhsdfgw-1}{}& \int_0^1 |F_t|^p\,dt<\infty, \qquad    F_t=\int_{-\infty}^t g^{(k)}(t-s)\,dL_s
%\end{align}
%and all  integrals are well-defined. 
\end{lem}
\begin{proof}
We will not need the assumption \eqref{kshs} on $g$ in the proof.  For notation simplicity  we only consider the case $k=1$, since the general case follows by similar arguments.   
To prove \eqref{sdhjlshjl} it is sufficient 
to show that the three conditions (5.3), (5.4) and (5.6) from \cite[Theorem~5.1]{SamBra} are satisfied
(this result uses the condition $p\geq 1$). In fact, the representation \eqref{sdhjlshjl} of $(\partial/\partial t)X_t$ 
follows by the equation below (5.10) in  \cite{SamBra}.
In our setting the  function $\dot \sigma$ defined in \cite[Eq.~(5.5)]{SamBra} is constant and hence  (5.3), (5.4) and (5.6) in \cite{SamBra}  simplifies to 
\begin{align}
{}& \label{sldfhogh} \int_{\R} \nu\Big(\Big(\frac{1}{\| g'\|_{L^p([s,1+s])}},\infty\Big)\Big)\,ds<\infty,\\ 
{}& \label{sdflhkhgs} \int_0^\infty \int_\R \Big(|x g'(s)|^2\wedge 1\Big)\,\nu(dx)\,ds<\infty, \\ 
\label{khshjl}
{}& \int_0^1 \int_\R |g'(t+s)|^p\Big(\int_{r/|g'(t+s)|}^{1/\|g'\| _{L^p([s,1+s])}} x^p \,\nu(dx)\Big)\,ds\,dt<\infty
\end{align}
for all $r>0$. When the lower bound in the inner integral in \eqref{khshjl} exceed the upper bound  the integral is set to zero.
Since $\alpha>1-1/\beta$ we may choose $\epsilon>0$ such that $(\alpha-1)(\beta+\epsilon)>-1$. 
 To show \eqref{sldfhogh} we use the estimates
\begin{align}
\| g'\|_{L^p([s,1+s])}\leq {}& K \Big(\1_{\{s\in [-1,1]\}}  +\1_{\{s>1\}} |g'(s)|\Big),\qquad s\in \R,  
\intertext{and}
\nu((u,\infty))\leq {}& 
\begin{cases} K u^{-\theta} & u\geq 1 \\ 
K u^{-\beta-\epsilon} & u\in (0,1],
\end{cases}
\end{align}
which both follow from assumption (A). Hence, we deduce that  
\begin{align}
{}& \int_{\R} \nu\Big(\Big(\frac{1}{\| g'\|_{L^p([s,1+s])}},\infty\Big)\Big)\,ds
\\ {}& \qquad \leq 
 \int_{-1}^1  \nu\Big(\Big(\frac{1}{K},\infty\Big)\Big)\,ds
 +  \int_{1}^{\infty}  \nu\Big(\Big(\frac{1}{K | g'(s)|},\infty\Big)\Big)\,ds
 \\ {}& \qquad \leq 
 2 \nu\Big(\Big(\frac{1}{K},\infty\Big)\Big)
 + K\int_{1}^{\infty}  \Big(|g'(s)|^\theta\1_{\{K |g'(s)|\leq 1\}}+|g'(s)|^{\beta+\epsilon}\1_{\{K|g'(s)|> 1\}}\Big)\,ds<\infty
\end{align}
which shows \eqref{sldfhogh} (recall that $|g'|$ is decreasing on $(1,\infty)$). To show  \eqref{sdflhkhgs} we  will use   the following two  estimates: 
\begin{align}\label{hgsdghl}
 {}&   \int_0^1 \Big(|s^{\alpha-1}x|^2\wedge 1\Big)\,ds
\leq 
\begin{cases} K\Big(\1_{\{|x|\leq 1\}} |x|^{1/(1-\alpha)}+  \1_{\{|x|> 1\}} \Big) & \alpha<1/2 \\ 
K\Big(\1_{\{|x|\leq 1\}} x^2\log(1/x)+  \1_{\{|x|>1\}} \Big)  & \alpha = 1/2 \\   K\Big(\1_{\{|x|\leq 1\}}x^2 +  \1_{\{|x|> 1\}} \Big) & \alpha>1/2,
\end{cases}
\intertext{and} 
\label{howeerho}
{}&  \int_{\{|x|>1\}}  \Big(|x g'(s)|^2\wedge 1\Big)\,\nu(dx) 
 \leq	  K \int_{1}^\infty  \Big(|x g'(s)|^2\wedge 1\Big) x^{-1-\theta}\,dx
 \leq K |g'(s)|^\theta.\qquad
\end{align}
For $\alpha<1/2$ we have 
\begin{align}
{}& \int_0^\infty \int_\R \Big(|x g'(s)|^2\wedge 1\Big)\,\nu(dx)\,ds\\
{}& \quad 
\leq K\Big\{\int_\R  \int_0^1 \Big(|x s^{\alpha-1}|^2\wedge 1\Big)\,ds\,\nu(dx)+
\int_1^\infty \int_{\{|x|\leq 1\}}   \Big(|x g'(s)|^2\wedge 1\Big)\,\nu(dx)\,ds  \\  {}&\qquad \qquad +\int_1^\infty   \int_{\{|x|> 1\}}  \Big(|x g'(s)|^2\wedge 1\Big)\,\nu(dx)\,ds\Big\}\\ {}&\quad 
 \leq K\Big\{ \int_\R  \big( \1_{\{|x|\leq 1\}} |x|^{1/(1-\alpha)}+\1_{\{|x|>1\}}\big)\, \nu(dx)\\ {}& \qquad \qquad +\Big(\int_1^\infty |g'(s)|^2\,ds\Big)\Big(\int_{\{|x|\leq 1\}} x^2\,\nu(dx) \Big) +\int_1^\infty |g'(s)|^\theta\,ds\Big\}<\infty, 
\end{align}
where  the first  inequality follows by  assumption (A), 
the second inequality follows by \eqref{hgsdghl} and \eqref{howeerho},  and   the last inequality is due to the fact that    $1/(1-\alpha)>\beta$ and $g'\in L^\theta((1,\infty))\cap L^2((1,\infty))$. 
This shows \eqref{sdflhkhgs}. The two remaining  cases $\alpha=1/2$ and $\alpha>1/2$ follow similarly.

Now, we will prove that \eqref{khshjl} holds.
Since $|g'|$ is decreasing on $(1,\infty)$ we have for all $t\in [0,1]$ 
that 
 \begin{align}\label{sdfkgkhsfh}
{}& \int_1^\infty |g'(t+s)|^p\Big(\int_{r/|g'(t+s)|}^{1/\|g'\| _{L^p([s,1+s])}} x^p \,\nu(dx)\Big)\,ds
\\ {}& \qquad  \leq 
 \int_1^\infty |g'(s)|^p\Big(\int_{r/|g'(1+s)|}^{1/|g'(s)|} x^p \,\nu(dx)\Big)\,ds
 \\ \label{khshf}{}& \qquad  \leq \frac{K}{p-\theta}
 \int_1^\infty |g'(s)|^p\Big(|g'(s)|^{\theta-p}-|g'(s+1)/r|^{\theta-p}\Big)\1_{\{r/|g'(1+s)|\leq 1/|g'(s)|\}}\,ds. \qquad 
  \end{align}
For $p>\theta$, \eqref{khshf} is less than or equal to 
  \begin{equation}
 \frac{K}{p-\theta} \int_1^\infty |g'(s)|^\theta\,ds<\infty,
  \end{equation}
 and for $p<\theta$,
 \eqref{khshf} is less than or equal to 
 \begin{align}
\frac{K r^{p-\theta}}{\theta-p}
 \int_1^\infty |g'(s)|^p |g'(s+1)|^{\theta-p}\,ds\leq 
 \frac{K r^{p-\theta}}{\theta-p} \int_1^\infty |g'(s)|^\theta\,ds<\infty,
 \end{align}
 where the first inequality is due to the fact that $|g'|$ is decreasing on $(1,\infty)$. 
 Hence we have shown that 
 \begin{equation}\label{khsk}
 \int_0^1 \int_1^\infty |g'(t+s)|^p\Big(\int_{r/|g'(t+s)|}^{1/\|g'\| _{L^p([s,1+s])}} x^p \,\nu(dx)\Big)\,ds\,dt<\infty
 \end{equation}
 for $p\neq \theta$. 
Suppose that $p>\beta$. For $t\in [0,1]$ and $s\in [-1,1]$ we have 
 \begin{align}\label{klhkl}
 {}& \int_{r/|g'(t+s)|}^{1/\|g'\| _{L^p([s,1+s])}} x^p \,\nu(dx)\leq  \int_{1}^{1/\|g'\| _{L^p([s,1+s])}} x^p \,\nu(dx)+ \int_{r/|g'(t+s)|}^{1} x^p \,\nu(dx)\qquad \\
 {}& \qquad \leq K\Big( \|g'\| _{L^p([s,1+s])}^{\theta-p}
+1 \Big)\label{klhkl-1}
 \end{align}
 and hence 
 \begin{align}\label{eq93e4u7}
  {}&  \int_0^1 \int_{-1}^1 |g'(t+s)|^p\Big(\int_{r/|g'(t+s)|}^{1/\|g'\| _{L^p([s,1+s])}} x^p \,\nu(dx)\Big)\,ds\,dt\\ \label{eq93e4u7-1}{}&\qquad  
   \leq K\Big(\int_{-1}^1 \|g'\|_{L^p([s,s+1])}^{\theta} \,ds+ \int_{-1}^1
    \|g'\|_{L^p([s,1+s])}^p\,ds\Big)<\infty.
 \end{align}
Suppose that  $p\leq \beta$. For  $t\in [0,1]$ and $s\in [-1,1]$ we have 
 \begin{align}
  \int_{r/|g'(t+s)|}^{1/\|g'\| _{L^p([s,1+s])}} x^p \,\nu(dx)\leq K\Big( \|g'\| _{L^p([s,1+s])}^{\theta-p}
+|g'(t+s)|^{\beta+\epsilon-p} \Big)
 \end{align}
 and hence 
 \begin{align}\label{skdfh}
{}&  \int_0^1 \int_{-1}^1 |g'(t+s)|^p\Big(\int_{r/|g'(t+s)|}^{1/\|g'\| _{L^p([s,1+s])}} x^p \,\nu(dx)\Big)\,ds\,dt\\ \label{skdfh-1}{}&\qquad  
   \leq K\Big(\int_{-1}^1 \|g'\|_{L^p([s,s+1])}^{\theta} \,ds+ \int_{-1}^1
    \|g'\|_{L^{\beta+\epsilon}([s,1+s])}^{\beta+\epsilon}\,ds\Big)<\infty
 \end{align}
 since $(\alpha-1)(\beta+\epsilon)>-1$. Thus,  \eqref{khshjl} follows by \eqref{khsk}, \eqref{eq93e4u7}--\eqref{eq93e4u7-1} and \eqref{skdfh}--\eqref{skdfh-1}.  
 
For  $p=\theta$ the above proof  remains valid  except for \eqref{khsk}, where we need the additional assumption (A-log).  This completes the proof. 
 \end{proof}

\begin{lem}\label{lem-f-65}
For all $\zeta\in W^{k,p}$ we have  as $n\to \infty$, 
\begin{equation}\label{con-g-1}
n^{-1+pk}V(\zeta,p;k)_n\to \int_0^1 |\zeta^{(k)}(s)|^p\,ds.
\end{equation}
\end{lem}

\begin{proof}
First we will assume that $\zeta\in C^{k+1}(\R)$ and afterwards we will prove the lemma by approximation. Successive applications of Taylor's theorem gives
\begin{equation}\label{est-g-sf}
\Delta^n_{i,k} \zeta= \zeta^{(k)}\Big(\frac{i-k}{n}\Big) \frac{1}{n^k} +a_{i,n}, \qquad n\in \N, \ k\leq i\leq n
%\Delta^k g(t)
\end{equation}
where $a_{i,n}\in \R$ satisfies 
\begin{equation}
|a_{i,n}| \leq K n^{-k-1}, \qquad n\in \N, \ k\leq i\leq n. 
\end{equation}
By Minkowski's inequality, 
\begin{align}
 {}& \Big|\Big(n^{kp-1} V(\zeta)_n\Big)^{1/p} - \Big(n^{kp-1}
  \sum_{j=k}^n \Big|\zeta^{(k)}\Big(\frac{i-k}{n}\Big) \frac{1}{n^k}\Big|^p\Big)^{1/p}\Big|\\ {}& \qquad  \leq 
\Big(n^{pk-1} \sum_{j=k}^n |a_{i,n}|^p\Big)^{1/p}
\leq K n^{-1-1/p}\to 0. 
\end{align}
By continuity of $\zeta^{(k)}$ we have 
\begin{equation}
n^{kp-1} \sum_{i=k}^n \Big|\zeta^{(k)}\Big(\frac{i-k}{n}\Big) \frac{1}{n^k}\Big|^p \to \int_0^1 |\zeta^{(k)}(s)|^p\,ds
\end{equation}
as $n\to\infty$,  which   shows  \eqref{con-g-1}.

The statement of the lemma  for a general  $\zeta \in W^{k,p}$ follows  by approximating $\zeta$ through a sequence of $C^{k+1}(\R)$-functions 
and Minkowski's inequality.  This completes the proof. 
\end{proof}

The Lemmas~\ref{abs-cont-sdf} and \ref{lem-f-65} yield Theorem~\ref{maintheorem}(iii). 
\qed

\section{Proof of Theorem \ref{sec-order}} \label{sec5}
\setcounter{equation}{0}
\renewcommand{\theequation}{\thesection.\arabic{equation}}

%We recall the definition of $m_p$ in Theorem \ref{maintheorem}(ii) and 
 %of $g_{i,n}$ introduced  at \eqref{def-g-i-n}.
Throughout this section we suppose that the assumptions stated in  Theorem~\ref{sec-order} hold, which in particular means that $p<\beta/2$. Suppose in addition that $\alpha\in (0,k-1/\beta)$ which will be satisfied in both (i) and (ii) of Theorem~\ref{sec-order}, and note that this condition is  equivalent to $(\alpha-k)\beta>-1$. 
Without loss of generality we will assume that  the symmetric $\beta$-stable L\'evy process  $L$
has  scale parameter $\sigma=1$ and  (A) holds with $\delta=c_0=1$. 
%\subsection{Proof of Theorem~\ref{sec-order}(i)}  \label{sec5.1}
%Throughout the following subsection we assume  that the conditions  of 
%Theorem~\ref{sec-order}(i) hold,  in particular,  $k=1$. 
 In the following subsection we will consider some notation and decompositions to be used in the proof of Theorem~\ref{sec-order}.

%%%%%%%%%%%%%%%%%%%
\subsection{Notation and outline of the proof} \label{sec6.1}
%%%%%%%%%%%%%%%%

In addition to the notation introduced  in Subsection~\ref{main-thm-2} we define the following truncated version 
of  $Y^n_r$ in \eqref{eq:232452} by 
\begin{equation}
Y_{r}^{n,m}=   \int_{r-m}^r \phi_{r}^n(s)\,dL_s, \qquad n\in \N\cup\{\infty\}, m,r\geq 0, 
\end{equation}
where the function $\phi_{r}^n$ has been introduced in \eqref{eq:23412343}.
For $n,m\in\N$ we set 
%(function $h_k$ is defined in \eqref{def-h-13}).
%For all $n\geq 1$, $r\geq 0$ and $m\geq 0$ set
%%\label{def-rho}
%\begin{equation}\label{eq:23412343}
%%\phi_{r}^n(s)={}& D^kg_n(r-s),\qquad \phi_r^\infty(u)=h_k(j+u),\\ 
%%% = n^{\alpha} \Big( g\big(\frac{r-s}{n}\big)-g\big(\frac{r-1-s}{n}\big)\Big),\\ 
%%% ,
%%\shortintertext{For $n\in \N\cup\{\infty\}$ }
%% Y_{r}^n= {}& 
%% \int_{-\infty}^r \phi_{r}^n(s)\,dL_s,\qquad 
% Y_{r}^{n,m}=   \int_{r-m}^r \phi_{r}^n(s)\,dL_s,
\begin{equation}
 S_{n} =  \sum_{r=k}^n \Big(  | Y^{n}_r|^p- \E[ | Y^{n}_r|^p]\Big)\qquad \text{and}\qquad S_{n,m} = \sum_{r=k}^n \Big(  | Y^{n,m}_r|^p- \E[ | Y^{n,m}_r|^p]\Big) .
 \end{equation}
%, \qquad  
%m_p^n=  \E[| Y^n_1 |^p].%, \qquad F_n(x)= |x|^p-m_p^n.
By \eqref{self-similar} we have that 
\begin{align} \label{statdec}
n^{p(\alpha + 1/\beta)}V(p;k)_n\eqschw S_n + (n-k+1)\E[ |Y^{n}_1|^p],
\end{align}
and hence when proving Theorem~\ref{sec-order} we may instead analyse the right-hand side of \eqref{statdec}. 
For all $n\in \N\cup\{\infty\}$, $j\geq 1$ and $m\geq 0$ we also set 
\begin{align}
 \rho^n_j = {}& \| \phi^n_j \|_{L^\beta(\R\setminus [0,1])},  \qquad  
\rho^{n,m}_j =   \| \phi^n_j \|_{L^\beta([j-m,j]\setminus [0,1])},\\
 U_{j,r}^n={}& \int_{r}^{r+1} \phi_{j}^n(u)\,dL_u. \label{udef}
 \end{align}
For all $r\in \R$ we consider  the following $\sigma$-algebras
\begin{equation}
  \g_r=\sigma(L_s-L_u: s,u\leq r)\qquad\text{and}\qquad  \g_{r}^1=\sigma(L_s-L_u: r\leq s,u\leq r+1 ).
\end{equation}
We note  that $(\g^1_r)_{r\geq 0}$ is not a filtration. 
%%%%%%%%%%%%%%
Let $W$ denote  a symmetric $\beta$-stable random variable with scale parameter $\rho\in (0,\infty)$ and $\Phi_\rho:\R\to \R$ be defined by  
\begin{equation}\label{def-H-rho}
\Phi_\rho(x)= \E[| W+x|^p]- \E[ | W |^p],\qquad  x\in \R. 
\end{equation}
For all $n\geq 1, m,r\geq 0$ let 
\begin{align}
V^{n,m}_r= {}&  | Y_r^n |^p - |Y^{n,m}_r|^p  - \E\big[  | Y_r^n |^p - |Y^{n,m}_r|^p\big],
\\[10pt]
%\shortintertext{and}
\label{eq-S-S'-2}
 \zeta_{r,j}^{n,m}= {}&  \E[ V^{n,m}_r |\g_{r-j+1}] - \E[ V^{n,m}_r |\g_{r-j}] - \E[ V^{n,m}_r | \g_{r-j}^1], \\[10pt]
%\eta^{m}_{j,r} = {}& \Big\{\Phi_{\rho_j^\infty} ( U_{j,r}^\infty)-\E[\Phi_{\rho_j^\infty} (U_{j,r}^\infty)]\Big\} - 
%\Big\{\Phi_{\rho_j^{\infty,m}} ( U_{j,r}^\infty)-\E[\Phi_{\rho_j^{\infty,m}} (U_{j,r}^\infty)]\Big\}\1_{\{j\leq m\}},
%\shortintertext{and also }
\label{def-r-q-z}
R^{n,m}_r={}& \sum_{j=1}^\infty \zeta_{r,j}^{n,m} \qquad \quad  \text{and}\qquad \quad
 Q^{n,m}_r  = \sum_{j=1}^\infty  \E[ V^{n,m}_r | \g_{r-j}^1].
\end{align}
According to  Remark~\ref{rem-well-defined} below the  two series  $R^{n,m}_r$ and  $Q^{n,m}_r $  converge with probability one, and the following decomposition of $S_n-S_{n,m}$ holds with probability one
\begin{equation}\label{eq:283723}
S_{n}- S_{n,m} = \sum_{r=k}^n R^{n,m}_r +\sum_{r=k}^n Q^{n,m}_r.
\end{equation}  
%Furthermore, we let $V_r^n=V^{n,0}_r, \zeta_{r,j}^{n}= \zeta_{r,j}^{n,0}$ and $\eta_{j,r} =\eta^{0}_{j,r}$. 
Decompositions of the type  \eqref{eq:283723} has been successfully used in theory 
of discrete time moving averages, see e.g.\ Ho and Hsing~\cite{hh97},  and   will also play a crucial role in the proof of  Theorem~\ref{sec-order}. 
Indeed,  for the proof of Theorem~\ref{sec-order}(i) we will  choose $m=0$ in  \eqref{eq:283723} and since $S_{n,0}=0$ we have the   following decomposition of $S_n$: 
\begin{equation} \label{eq:7836638-22}
S_{n} = \sum_{r=k}^n R^{n,0}_r +\sum_{r=k}^n \Big( Q^{n,0}_r- Z_r\Big)+   \sum_{r=k}^n Z_r,
 \end{equation}
 where 
 \begin{equation}
Z_r =  \sum_{j=1}^\infty \Big\{\Phi_{\rho_j^\infty} ( U_{j,r}^\infty)-\E[\Phi_{\rho_j^\infty} (U_{j,r}^\infty)]\Big\}. 
\end{equation}
After suitable scaling we show that the first two  sums on the right-hand side of \eqref{eq:7836638-22} are  negligible, see \eqref{eq:827284-1}.  To analyse the third sum we note the random variables    $\{Z_r: r\geq 1\}$ are independent and identically distributed, which follows from their  definition. 
Hence, to complete the proof of Theorem~\ref{sec-order}(i), it is enough to show that the  common law of $\{Z_r: r\geq 1\}$  belong to the domain of attraction of an $(k-\alpha)\beta$-stable random variable, which is done in \eqref{eq:93673}. 

The main part of the proof of Theorem~\ref{sec-order}(ii) consists in showing that 
\begin{equation}\label{eq:29372937}
\lim_{m\to \infty} \limsup_{n\to \infty} \big(n^{-1} \E[(S_{n}-S_{n,m})^2]\big)  = 0,
\end{equation}
see \eqref{eq:189362720}. We prove  \eqref{eq:29372937} by estimating each of the two sums on the right-hand side of 
\eqref{eq:283723} separately. We note that for a fixed $m\geq 1$  the sequences $\{S_{n,m}\!:n\geq 1\}$ are $m$-dependent, which means that for all $k\geq 1$ the random variables 
$\{S_{1,m},\dots,S_{k,m} \}$ are independent of  $\{S_{n,m}\!:n\geq k+1+m\}$. Hence, using  a standard result for $m$-dependent sequences one can deduce a central limit theorem for the sequences $\{S_{n,m}: n\geq 1\}$ and by using \eqref{eq:29372937}  transfer this result to $S_n$, which will prove Theorem~\ref{sec-order}(ii).  
In the next subsection we  present some estimates which play a key role in the proof of Theorem~\ref{sec-order}.

%%%%%%%%%%%%%%%
\subsection{Preliminary estimates}
%%%%%%%%%%%%%%%%%%

The assumption $|g^{(k)}(x)|\leq K x^{\alpha-k}$ for all $x>0$ implies that  
\begin{equation}\label{eq:72138}
\| \phi^n_j\|_{L^\beta([0,1])}\leq K j^{\alpha-k}
\end{equation}
for some finite constant $K$ which do not depend on $j\in \N$, $n\in \N\cup \{\infty\}$. The estimate \eqref{eq:72138} will be used 
repeatedly throughout the proof. In the following we will collect some estimates on the functions $\Phi_\rho$ defined in \eqref{def-H-rho} which will be used various places in the proofs. We first observe the identity
\begin{align} \label{xp}
|x|^p = a_p^{-1} \int_{\R} (1-\exp(iux)) |u|^{-1-p} du \qquad \text{for } p\in (0,1),
\end{align}
with $a_p= \int_\R (1-\exp(iu))|u|^{-1-p}\,du\in \R_+$, 
which can be shown by substitution $y=ux$. Secondly, for any 
deterministic function $\varphi :\R\rightarrow \R$ satisfying $\varphi\in L^{\beta}(\R)$, it holds that
\begin{align} \label{charfun}
\E\left[ \exp \left(iu \int_{\R} \varphi(s) \,dL_s \right)\right] = \exp \left(-|u|^{\beta} \int_{\R} |\varphi(s)|^{\beta} \,ds \right).
\end{align}
Applying the identities \eqref{xp} and \eqref{charfun}, we obtain the representation
\begin{equation}\label{rep-H-est-2}
\Phi_\rho(x)= a_p^{-1} \int_\R \big(1-\cos(ux) \big) e^{-\rho^\beta | u |^\beta}|u|^{-1-p}\,du. 
\end{equation} 
From \eqref{rep-H-est-2}, we deduce that $\Phi_\rho \in C^3(\R)$  and 
it holds that 
\begin{align}\label{est-H''}
\Phi_\rho^{'}(x) {}& =  a_p^{-1} \int_\R \sin(ux) |u|^{-p}e^{-\rho^\beta | u |^\beta}\,du \\
\Phi_\rho^{''}(x) {}& =   a_p^{-1} \int_\R \cos(ux) |u|^{1-p}e^{-\rho^\beta | u |^\beta}\,du 
 % \leq a_p^{-1} \int_\R  |u|^{1-p}e^{-\rho^\beta | u |^\beta}\,du,  
 \\ \Phi_\rho^{'''}(x) {}& = - a_p^{-1} \int_\R \sin(ux) |u|^{2-p}e^{-\rho^\beta | u |^\beta}\,du
\end{align}
In the following we let $\epsilon>0$ be a fixed number. The identities at \eqref{est-H''}
imply that for $v=1,2,3$ there exists a finite constant $K_\epsilon$ such that for all $\rho\geq \epsilon$ and all $x\in \R$
%\footnote{Recall restriction on $p$. $p<1$?}
\begin{equation}\label{est-H-28217}
 |\Phi_\rho^{(v)}(x)|\leq K_\epsilon .
\end{equation}
By  \eqref{rep-H-est-2} we also deduce the following estimate by several applications of the mean
value theorem 
\begin{equation}\label{est-H-inf-2}
|\Phi_\rho(x)-\Phi_\rho(y)| \leq K_\epsilon \Big(   \big(|x|\wedge 1+|y|\wedge 1\big)|x-y| \1_{\{|x-y|\leq 1\}} + |x-y|^p\1_{\{|x-y|>1\}}\Big) 
\end{equation} 
which holds for all $\rho\geq \epsilon$ and all $x,y\in\R$. Eq.~\eqref{est-H-inf-2} used on $y=0$ yields that 
\begin{align} \label{hestimate}
|\Phi_\rho (x)|\leq K_\epsilon (|x|^{p}\wedge |x|^{2} ).
\end{align}
In particular,  it implies that 
\begin{equation}\label{est-H-beta}
 |  \Phi_\rho(x)|\leq K_\epsilon |x|^l\qquad \text{for all }l\in (p,\beta).
\end{equation}
Moreover, for all $r\in [p,2]$ and   $\rho_1, \rho_2\geq \epsilon$  we  deduce by
 \eqref{rep-H-est-2}  that %\footnote{Check this result}
\begin{equation}\label{est-H-rho-1-2}
 | \Phi_{\rho_1}(x)- \Phi_{\rho_2}(x) |\leq K_\epsilon | \rho_1^\beta - \rho_2^\beta |\cdot |x|^{r} \qquad \text{for all }x\in \R. 
\end{equation}

 \begin{rem}\label{rem-well-defined} In the following we will show that the three series $R^{n,m}_r, Q^{n,m}_r$ and $Z^m_r$  defined in \eqref{def-r-q-z} converge almost surely, and the identity \eqref{eq:283723} holds almost surely. 
To show the above claim we will first prove that for all $n\geq 1$ and $m\geq 0$ the two series 
 \begin{equation}\label{eq:232323234}
   (a)\!:\sum_{j=1}^\infty   \E[ V^{n,m}_r | \g_{r-j}^1] \qquad \qquad (b)\!:\sum_{j=m+1}^\infty \Big(  \Phi_{\rho_j^\infty} ( U_{j,r}^\infty)-\E[\Phi_{\rho_j^\infty} (U_{j,r}^\infty)]\Big)
  \end{equation}
  converge absolutely with probability one. For $j\geq 1$ we have that 
  \begin{align}\label{eq:cal-Vsdfs}
  \E[ V^{n,m}_r | \g_{r-j}^1] = {}& \Phi_{\rho^n_{j}}(U^n_{r,r-j})- \Phi_{\rho^{n,m}_{j}} (U^n_{r,r-j} ) \1_{\{j\leq m\}} \\ {}&  - \E\Big[\Big( \Phi_{\rho^n_{j}}(U^n_{r,r-j})- \Phi_{\rho^{n,m}_{j}} (U^n_{r,r-j} ) \1_{\{j\leq m\}}\Big)\Big].\label{eq:1212124432} 
  \end{align}
  We have that  $\rho^n_j\to \| \phi_1^n \|_{L^\beta(\R)}>0$ as $j\to \infty$, and hence $\{\rho^n_{j} :j\geq N \}$ is bounded away from zero 
  for $N$ large enough.   
For all $j>\max\{m,N\}$ and all  $\gamma\in (p,\beta)$ we have 
\begin{align}
\E\big[ \big| \E[ V_{r}^{n,m}  |\g^{1}_{r-j}]\big| \big]\leq  {}& 2\E\big[ \big|\Phi_{\rho^n_{j}}(U^n_{r,r-j})\big|\big]\leq K \E[ |U^n_{r,r-j}|^{\gamma} ] 
\\ \leq {}& K \| \phi^n_j \|_{L^\beta([0,1])}^\gamma
\leq K j^{(\alpha-k)\gamma},
\end{align}
where the first inequality follows by  
\eqref{eq:1212124432}, the second inequality follows by  \eqref{est-H-beta} and the last inequality follows by \eqref{eq:72138}. 
By choosing $\gamma$ close enough to $\beta$ and using the assumption $(\alpha-k)\beta<-1$, it follows that the series (a) in \eqref{eq:232323234} converges absolutely almost surely.  
A similar application of  \eqref{est-H-beta} and \eqref{eq:72138}  also shows that the series (b) in \eqref{eq:232323234} converge absolutely almost surely.  Next we note that $V^{n,m}_r= \E[ V_r^{n,m} | \g_r]$ and $\E[ V_r^{n,m} | \g_{-j}]\to \E[ V^{n,m}_r] =0$ almost surely as $j\to \infty$. The latter  claim follows from Kolmogorov's 0-1 law and the backward martingale convergence theorem.   From these two properties we deduce that $V^{n,m}_r$ has the following  telescoping sum representation 
  \begin{equation}\label{eq:23232323}
  V^{n,m}_r = \sum_{j=1}^\infty \Big( \E[ V^{n,m}_r |\g_{r-j+1}] - \E[ V^{n,m}_r |\g_{r-j}] \Big), 
  \end{equation}
  where sum converge almost surely.  Convergence of the three  series in \eqref{eq:232323234} and  \eqref{eq:23232323} show the claim in the remark together with the observation that $S_n - S_{n,m} = \sum_{r=k}^n V^{n,m}_r$.  \qed
   \end{rem}
 
The following estimates will play a key role in the proof of Theorem~\ref{sec-order}.
 
 \begin{prop} \label{prop-key-est}
 Suppose that the  conditions of Theorem~\ref{sec-order} hold, and hence in particular $p<\beta/2$ and $\alpha<k-1/\beta$.  For all $\epsilon>0$ there exists a finite constant $K$ such that for all $n\geq 1$ and $m\geq 0$ we have the  following estimates: 
 \begin{align}
{}&  \E\Big[ \Big(\sum_{r=k}^n R^{n,m}_r\Big)^2\Big] \\ {}&  \leq K \Big( n \Big[(m+1)^{(\alpha-k)\beta+1}\log^2(m+1) + (m+1)^{2(\alpha-k)\beta +3}\Big] 
+ n^{2(\alpha-k)\beta +4}+\log(n)\Big).   \\
\label{eq:112132-1}
\shortintertext{If in addition $\alpha<k-2/\beta$ then the estimate \eqref{eq:112132-3} holds:  }  
\label{eq:112132-3}
 {}& \E\Big[ \Big(\sum_{r=k}^n Q^{n,m}_r\Big)^2\Big]\leq   K\Big(   n^{ (\alpha-k)\beta+3+\epsilon}   + n (m+1)^{(\alpha-k)\beta +2+\epsilon}+1 \Big).
 %\\  \label{eq:112132-4}
% {}& \E\Big[ \Big(\sum_{r=k}^n Z^{m}_r\Big)^2\Big]\leq K n (m+1)^{(\alpha-k)\beta+1+\epsilon}.
\shortintertext{On the other hand, if  $\alpha>k-2/\beta$ then  the following estimate holds: }   \label{eq:112132-2}
 {}& 
 \E\Big[ \Big|\sum_{r=k}^n \Big( Q^{n,0}_r-Z_r\Big)\Big|\Big]\leq K \Big( n^{(\alpha-k)\beta +2+\epsilon}+ n^{1-\beta+\epsilon}\Big).
 \end{align}
% \marginpar{Check \eqref{eq:112132-2}--\eqref{eq:112132-3}.}
%  In \eqref{eq:112132-1}--\eqref{eq:112132-4},  $K$ denotes a finite constant not depending on $n$ and $m$. 
 \end{prop}
 
 The proof of Proposition~\ref{prop-key-est} is carried out in Subsections~\ref{proof-pro-key} and \ref{proof-prop-232}. 
We will also need the following inequality.
\begin{lem} \label{helplem1}
%\begin{itemize} 
Assume that the conditions of Theorem~\ref{sec-order} hold. Then there exists a finite constant $K$ such that for all $j,n\geq 1$ we have 
\begin{align*}
\left| \|\phi^n_j\|_{L^\beta(\R)}^{\beta} - \| \phi^\infty_j\|_{L^\beta(\R)}^{\beta} \right | \leq K 
\begin{cases}
n^{-1} \qquad\qquad  & \text{when } \alpha \in (0,k-2/\beta) \\
n^{(\alpha-k)\beta+1} & \text{when } \alpha \in (k-2/\beta, k-1/\beta)
\end{cases}
\end{align*}
where the functions $\phi^n_j$ and $\phi^\infty_j$ has been introduced at \eqref{eq:23412343}.
%\end{itemize}
\end{lem}
The proof  of Lemma~\ref{helplem1} is   postponed to Subsection~\ref{help-lemma1}. We are now ready to show Theorem~\ref{sec-order}(i).

%%%%%%%%%%%%%%%
\subsection{Proof of Theorem~\ref{sec-order}(i)} \label{sec6.3}
%%%%%%%%%%%%%%%

To prove Theorem~\ref{sec-order}(i) we will first  state and prove the following lemma:

\begin{lem}\label{est-H-inf}
There exists $\delta>0$ and a finite constant $K>0$ such that for all $\epsilon\in (0,\delta)$, $\rho>\delta$, $\kappa,\tau\in L^\beta([0,1])$ with   $\| \kappa \|_{L^\beta([0,1])} , \| \tau \|_{L^\beta([0,1])}\leq 1$ the following inequality holds
\begin{align}
{}& \Big\| \Phi_\rho\Big(\int_0^1 \kappa(s)\,dL_s\Big)-\Phi_\rho\Big(\int_0^1 \tau(s)\,dL_s\Big) \Big\|_{L^{(k-\alpha)\beta+\epsilon}}\\ {}&\qquad  \leq K 
\Big( \| \kappa - \tau \|_{L^\beta([0,1])}+
\| \kappa - \tau \|_{L^\beta([0,1])}^{\frac{1}{k-\alpha+\epsilon/\beta}} \Big).\label{eq:97236554340}
\end{align}
Moreover, 
\begin{align}
& \Big\| \Phi_\rho\Big(\int_0^1 \kappa(s)\,dL_s\Big)-\Phi_\rho\Big(\int_0^1 \tau(s)\,dL_s\Big) \Big\|_{L^{1}} \\ 
&  \leq K 
\begin{cases}\Big( \|\kappa\|_{L^\beta([0,1])}^{\beta-1-\epsilon} + \|\tau\|_{L^\beta([0,1])}^{\beta-1-\epsilon} \Big) \| \kappa - \tau \|_{L^\beta([0,1])}+
\| \kappa - \tau \|_{L^\beta([0,1])}^{\beta} & \beta>1 \medskip \\
\| \kappa - \tau \|_{L^\beta([0,1])}^{\beta-\epsilon}  & \beta\leq 1. 
\end{cases}
\label{est-234l1j32lh3}
\end{align}
%%set $U=\int_0^1 \kappa(s)\,dL_s$ and $V=\int_0^1 \tau(s)\,dL_s$. 
%%Let $r=(k-\alpha)\beta+\epsilon$ and   $\Phi_\rho$ be given by \eqref{def-H-rho}.  For all $\rho\geq \epsilon$ we have 
%\begin{itemize}
%\item For $k-\alpha\geq 1$ we have that 
%\begin{equation}
% \Big\| \Phi_\rho\Big(\int_0^1 \kappa(s)\,dL_s\Big)-\Phi_\rho\Big(\int_0^1 \tau(s)\,dL_s\Big) \Big\|_{L^{(k-\alpha)\beta+\epsilon}}\leq K \| \kappa - \tau \|_{L^\beta}^{\frac{1}{k-\alpha+\epsilon/\beta}} .
%\end{equation}
%\item For $k-\alpha<1$ we have that 
%\begin{align}\label{est-we-want-3}
% {}& \| \Phi_\rho(U)-\Phi_\rho(V) \|_{L^r}\  \\ {}& \qquad \leq  K_\rho\Big\{ \| \kappa-\tau \|_{L^\beta([0,1])}       +   
% \| \kappa-\tau \|_{L^\beta([0,1])}^{\frac{1}{k-\alpha+\epsilon/\beta}} \Big\}. 
%\end{align}
%\end{itemize}
\end{lem}

To prove Lemma~\ref{est-H-inf} we will among others use the  following simple estimates. 
 
\begin{lem}\label{sim-est-W}
 Let $W$ be a  symmetric $\beta$-stable random variable with scale parameter 
 $\rho$.  
 \begin{enumerate}\label{eq:7366829-1}
 \item [(i)] \label{eq:7366829-1} Let  $\gamma<\beta$. For all $\rho \leq 1$    we have that  
\begin{equation}\label{eq:836}
 \E[ |  W |^\gamma\1_{\{| W |\geq 1\}}]\leq K \rho^{\beta}.
\end{equation}
\item [(ii)] \label{eq:7366829-2} Let  $\gamma>\beta$. For  all $\rho \leq 1$ we have that 
\begin{equation}\label{eq:8276y26}
\E[( |W|\wedge 1)^\gamma ] \leq K \rho^\beta.
\end{equation}
 \end{enumerate}
\end{lem}

\begin{proof}[Proof of Lemma~\ref{sim-est-W}] Let  $\eta$ be the  density of a standard symmetric $\beta$-stable random variable. According to \cite[Theorem~1.1]{Wat}
we have that      $\eta(x)\leq	 K(1+|x|)^{-1-\beta}$, $x\in \R$. To prove (i) we use  substitution to get 
\begin{align}
 \E[ |  W |^\gamma\1_{\{| W |\geq 1\}}] = {}&\int_\R |\rho x|^\gamma \1_{\{ |\rho x|\geq 1\}}\,\eta(x)\,dx
\\ \leq {}& K \rho^{-1} \int_\R |x|^\gamma \1_{\{|x|\geq 1\}} |\rho^{-1} x|^{-1-\beta}\,dx \leq K \rho^\beta, 
\end{align}
where we use that $\gamma<\beta$ in the last inequality. 
To show (ii) we  note that  the assumption $\gamma>\beta$ implies that  
\begin{align}
\E[ |  W |^\gamma\1_{\{| W |\leq  1\}}] = {}&\int_\R |\rho x|^\gamma \1_{\{ |\rho x|\leq 1\}}\,\eta(x)\,dx
\\ \leq {}& K \rho^{-1} \int_\R |x|^\gamma \1_{\{|x|\leq 1\}} |\rho^{-1} x|^{-1-\beta}\,dx 
\leq  K \rho^\beta. \label{eq:72627} 
\end{align}
Moreover, if $W_0$ denotes symmetric $\beta$-stable random variable with scale parameter $1$ then 
\begin{align}
\E[ \1_{\{| W|\geq 1\}}] = \P(| W_0|\geq \rho^{-1} ) \leq K \rho^{\beta} ,  
\end{align}
which together with \eqref{eq:72627} completes the proof of \eqref{eq:8276y26}. 
\end{proof}

\begin{proof}[Proof of Lemma~\ref{est-H-inf}] 
For notation simplicity set  $U=\int_0^1 \kappa(s)\,dL_s$, $V=\int_0^1 \tau(s)\,dL_s$ and  $r_\epsilon= (k-\alpha )\beta+\epsilon$ for all $\epsilon>0$. To prove \eqref{eq:97236554340}  fix   $\delta>0$ and let $\rho\geq \delta$. According to  \eqref{est-H-inf-2} and Minkowski inequality we have that 
%on $q$ and $q':=rq/(q-r)$, which satisfies the equality $1/r=1/q+1/q'$, we obtain  
\begin{equation}\label{eq:7238}
\Big \| \Phi_\rho(U)- \Phi_\rho(V)\Big\|_{L^{r_\epsilon}} \\    \leq     K \Big(\Big\| | U-V|\1_{\{|U-V|<1\}}\Big\|_{L^{r_\epsilon}}
+ \Big\| | U-V |^p  \1_{\{|  U-V |\geq 1\}}\Big\|_{L^{r_\epsilon}}\Big),%\leq {}& K\Big( \Big\{\| f\|_{L^\beta([0,1])}^{(q-r)/r}+\| h  \|_{L^\beta([0,1])}^{(q-r)/r}\Big\} \| f-h\|_{L^\beta([0,1])}   + \| f-h\|_{L^\beta([0,1])}^{\beta/r}\Big)
\end{equation}
where $K=K_\delta$ is a finite constant only depending on $\delta$. 
To estimate the second term on the right-hand side of \eqref{eq:7238} we note that $p \beta (k-\alpha)< 2 p < \beta$ by our assumptions, and hence for all $\epsilon>0$ small enough we have that $p r_\epsilon < \beta$.   Therefore, according to   
Lemma~\ref{sim-est-W}(i),  we have  
\begin{equation}\label{eq:2srlqwer}
\Big\| | U-V |^p  \1_{\{|  U-V |\geq 1\}}\Big\|_{L^{r_\epsilon}} \leq 
K \| \kappa - \tau \|_{L^\beta([0,1])}^{\beta/r_\epsilon}  = 
K \| \kappa - \tau \|_{L^\beta([0,1])}^{\frac{1}{k-\alpha+\epsilon/\beta}}.
\end{equation}
To estimate the first term on the right-hand side of \eqref{eq:7238} we assume first that  $k-\alpha\geq 1$  which implies that   $r_\epsilon >\beta$ for all $\epsilon>0$,  and hence by  Lemma~\ref{sim-est-W}(i) 
\begin{equation}
\Big\| | U-V|\1_{\{|U-V|<1\}}\Big\|_{L^{r_\epsilon}}\leq K \| \kappa - \tau\|_{L^\beta([0,1])}^{\beta/r_\epsilon} =    K \| \kappa - \tau\|_{L^\beta([0,1])}^{\frac{1}{k-\alpha+\epsilon/\beta}}.
\end{equation}
On the other hand, if  $k-\alpha<1$ then $r_\epsilon<\beta$ for all $\epsilon>0$ close enough to 0  which implies that 
\begin{equation}
\Big\| | U-V|\1_{\{|U-V|<1\}}\Big\|_{L^{r_\epsilon}} \leq \|  U-V\|_{L^{r_\epsilon}}\leq K \| \kappa - \tau\|_{L^\beta([0,1])}, 
\end{equation} 
and  completes the proof of \eqref{eq:97236554340}. 

To prove \eqref{est-234l1j32lh3} we are applying  \eqref{est-H-inf-2} to get  
\begin{align}
{}&  \Big \| \Phi_\rho(U)- \Phi_\rho(V)\Big\|_{L^{1}} \\  {}&\quad    \leq     K \Big(\Big\| \Big( |U|\wedge 1+ |V|\wedge 1\Big) | U-V|\1_{\{|U-V|<1\}}\Big\|_{L^{1}}
+ \Big\| | U-V |^p  \1_{\{|  U-V |\geq 1\}}\Big\|_{L^{1}}\Big).\qquad \label{eq:2038203}
\end{align}
By using that $p<\beta$ we have by Lemma~\ref{sim-est-W}(i)  
\begin{equation}
\Big\| | U-V |^p  \1_{\{|  U-V |\geq 1\}}\Big\|_{L^{1}} \leq K \| \kappa -\tau \|_{L^\beta([0,1])}^\beta. 
\end{equation}
Suppose first that $\beta>1$. To estimate the  first term in \eqref{eq:2038203} we let $r\in (1,\beta)$ and $q=r/(r-1)$ denote the conjugated number to $r$.  By H\"older's inequality we have 
\begin{align}
{}& \Big\| \Big( |U|\wedge 1+ |V|\wedge 1\Big) | U-V|\1_{\{|U-V|<1\}}\Big\|_{L^{1}} \\ 
{}& \qquad \leq \Big( \| |U|\wedge 1 \|_{L^q}+ \| |V|\wedge 1 \|_{L^q}\Big)\Big\| | U-V|\1_{\{|U-V|<1\}}\Big\|_{L^{r}}
\\ {}& \qquad \leq K \Big( \| \kappa  \|_{L^\beta([0,1])}^{\beta/q}+ \|\tau \|_{L^\beta([0,1])}^{\beta/q}\Big) \| \kappa - \tau \|_{L^\beta([0,1])},\label{eqerqwe}
\end{align}
where we have used Lemma~\ref{sim-est-W}(ii) and  $r<\beta<q$  in the second inequality. By \eqref{eqerqwe} we obtain \eqref{est-234l1j32lh3} by choosing $r$ close enough to $\beta$. For  $\beta\leq 1$ and all $\tilde \epsilon>0$  the first term in \eqref{eq:2038203}  is less than or equal to 
\begin{equation}
2 \E[ | U-V|\1_{\{|U-V|\leq 1\}}]\leq 2 \E[ | U-V|^{1+\tilde \epsilon}\1_{\{|U-V|\leq 1\}}]^{1/(1+\tilde\epsilon)}\leq K  \| \kappa - \tau\|_{L^\beta([0,1])}^{\beta/(1+\tilde \epsilon)}
\end{equation}
where we have used Lemma~\ref{sim-est-W}(ii) in the last inequality.   Hence 
choosing $\tilde \epsilon$ small enough  yields \eqref{est-234l1j32lh3}. 
\end{proof}

To prove Theorem~\ref{sec-order}(i) we  use \eqref{statdec} to obtain the decomposition 
\begin{align}
{}& n^{1-\frac{1}{(k-\alpha)\beta}}\Big(n^{-1+p(\alpha + 1/\beta)}V(p;k)_n- m_p\Big)\\ {}& \qquad 
 \eqschw  n^{\frac{1}{(\alpha-k)\beta}}S_n + n^{1-\frac{1}{(k-\alpha)\beta}}\Big( \frac{n-k+1}{n} \E[ |Y^{n}_1|^p]-m_p\Big).
 \label{eq:decom77272}
\end{align}
First we will prove that 
\begin{equation}\label{eq:78368222}
 n^{\frac{1}{(\alpha-k)\beta}}S_n\schw S\qquad  \text{as } n\to \infty,
\end{equation}
where the random variable $S$ is defined in Theorem~\ref{sec-order}(i). Afterwards we show that the second term on the right-hand side of \eqref{eq:decom77272} converges to zero.  To show \eqref{eq:78368222}   we will use the  decomposition \eqref{eq:7836638-22}, which shows that it suffices to prove that 
%\begin{equation} \label{eq:7836638}
%S_{n} = \sum_{r=k}^n R^{n,0}_r +\sum_{r=k}^n Q^{n,0}_r+   \sum_{r=k}^n Z_r^0.
% \end{equation}
%%where 
%%\begin{align}
%%R^{n}_r={}& \sum_{j=1}^\infty \zeta_{r,j}^{n}, \qquad  
%% Q^{n}_r  = \sum_{j=1}^\infty  \Big(\E[ V^{n}_r | \g_{r-j}^1]-  \eta_{j,r}\Big),\qquad 
%%  Z_r =  \sum_{j=1}^{\infty} \eta_{j,r} .
%%\end{align}
%%and 
%%\begin{align}
%%W_{m,r}= \int_{r}^{r+1} \phi_{m}(u)\,dL_u,\qquad \eta_{j,r} = {}& \Phi_{\rho_j} ( W_{j+r,r})-\E[\Phi_{\rho_j} (W_{j+r,r})].
%%\end{align}
%To show \eqref{eq:78368222} it suffices, cf.\ \eqref{eq:7836638},  to show \eqref{eq:827284-1} below 
\begin{align}\label{eq:827284-1}
{}& n^{\frac{1}{(\alpha-k)\beta}}  \sum_{r=k}^n R^{n,0}_r  \toop 0,\qquad
%\label{eq:827284-2}   
 n^{\frac{1}{(\alpha-k)\beta}} \sum_{r=k}^n \Big(Q^{n,0}_r -Z_r\Big)\toop 0,
%\\  \label{eq:827284-3}
\\  \label{eq"213213123} {}& n^{\frac{1}{(\alpha-k)\beta}}\sum_{r=k}^n  Z_r\schw S\quad 
\end{align}
as $n\to \infty$. 
According to \eqref{eq:112132-1} of Proposition~\ref{prop-key-est} we have that 
\begin{equation}
 \E\Big[ \Big( n^{\frac{1}{(\alpha-k)\beta}}  \sum_{r=k}^n R^{n,0}_r \Big)^2\Big] \leq K \Big( n^{\frac{2}{(\alpha-k)\beta}+1}
+ n^{2(\frac{1}{(\alpha-k)\beta}+(\alpha-k)\beta +2)}+n^{\frac{2}{(\alpha-k)\beta}}\log(n)\Big)\to 0
\end{equation}
as $n\to \infty$, where we have used  the inequality $2<x+1/x$ for all $x>1$ and the fact that $(k-\alpha)\beta>1$ by assumption. 
Furthermore, for all $\epsilon>0$ we have according to  \eqref{eq:112132-2} of Proposition~\ref{prop-key-est} and the assumption $\alpha>k-2/\beta$  that as $n\to \infty$
\begin{equation}
\E\Big[ \Big|n^{\frac{1}{(\alpha-k)\beta}}\sum_{r=k}^n \Big( Q^{n,0}_r-Z_r\Big)\Big|\Big]\leq K \Big( n^{\frac{1}{(\alpha-k)\beta}+(\alpha-k)\beta +2+\epsilon}+n^{\frac{1}{(\alpha-k)\beta}-\beta+1+\epsilon}\Big)\to 0\label{eq:2323qew}
\end{equation}
 for all $\epsilon$ close enough to zero. The first term on the right-hand side of \eqref{eq:2323qew} converge to zero  by the inequality $2<x+1/x$ for all $x>1$ and the assumption $(k-\alpha)\beta>1$. Convergence of the  second term on the right-hand side of \eqref{eq:2323qew} to zero is equivalent to $\alpha>k-\frac{1}{\beta(1-\beta)}$. But this is satisfied by the  assumption $\alpha<k-2/\beta$ for $\beta\geq 1/2$ and by explicit assumption for $\beta<1/2$.

%
%
%It now remains to prove that $A_n\rightarrow 0$, and to do this we distinguish two cases (recall that $1<\gamma<\beta$ close enough to $\beta$). 
%Split $A_n= \sum_{-n<s<n} + \sum_{s\leq -n}=: A^{\prime}_n + A^{\prime \prime}_n$ and assume for the moment that $1/\beta  +2(\alpha -1)<-1$ 
%holds. Since the inner sum is summable (for $\gamma$ close enough to $\beta$) we immediately see that $A^{\prime}_n \leq K n^{\gamma/(\alpha -1)\beta + 1}$.
%Since $\beta >1$ and $\alpha \in (0, 1-1/\beta)$ we deduce that $
%A^{\prime}_n \rightarrow 0$.  
%On the other hand, a direct computation shows that
%\[
%A^{\prime \prime}_n \leq K n^{\gamma/(\alpha -1)\beta + 3 + 2\gamma(\alpha -1)},
%\]
%and since $\beta (\alpha -1)<-1$, we readily obtain that 
%$
%A^{\prime \prime}_n \rightarrow 0$.
%Now, assume that  
%$
%1/\beta  +2(\alpha -1)\geq -1$.
%Then
%\[
%A^{\prime}_n\leq K n^{\gamma/(\alpha -1)\beta + 2 +\gamma +2\gamma(\alpha-k) }.
%\]
%Since $1/\beta  +2(\alpha -1)\geq -1$, which is equivalent to $1  +2\beta (\alpha -1)\geq -\beta$, we obtain for $\gamma$ close enough to $\beta$
%\[
%A^{\prime}_n\leq K n^{1/(\alpha -1) + 1}\rightarrow 0.
%\] 
%The convergence $A^{\prime \prime}_n\rightarrow 0$ is shown as above.
%

In the following we will show the last statement of \eqref{eq:827284-1}. 
Since $(Z_r)_{r\geq k}$ are i.i.d.\ with mean zero it is enough  to show that   
\begin{equation}\label{eq:93673}
\lim_{x\to\infty}x^{(k-\alpha)\beta}\P( Z>x)= \gamma \qquad \text{and}\qquad \lim_{x\to\infty}x^{(k-\alpha)\beta} \P(Z<-x)=0
\end{equation}
with $Z:=Z_k$, cf.\  %  \cite[Theorem 5.25]{hj94} or
 \cite[Theorem~1.8.1]{SamTaq}. The 
constant $\gamma$ is defined in \eqref{eq:24245} below. 
To show \eqref{eq:93673} 
let us define the function $\overline{\Phi}:\R\to \R_+$ via
\[
\overline{\Phi}(x):= \sum_{j=1}^{\infty} \Phi_{\rho_j^\infty} (\phi^\infty_j(0) x).
% \qquad \text{where} \qquad a_j:=h_k(j)\text{ for }j\geq 1. 
\]
Note that \eqref{rep-H-est-2} implies that $\Phi_{\rho_j^\infty} (x)\geq 0$ and hence $\overline \Phi$ is positive. Note that  $\rho_j^\infty\to \rho_\infty^\infty:= \| h_k \|_{L^\beta(\R)}>0$ which implies that  $(\rho_j^\infty)_{j\geq 1}$ is bounded away from 0, and hence by by \eqref{est-H-beta}  and for $l\in (p,\beta)$ with $(\alpha-k)l <-1$ we have 
\begin{equation}\label{eq:232}
|\overline{\Phi}(x)|\leq K |x|^l \sum_{j=1}^{\infty} \phi^\infty_j(0)^l \leq K |x|^l \sum_{j=1}^{\infty} j^{l(\alpha -k)}<\infty, 
\end{equation}
 which shows that $\overline \Phi$ is well-defined. 
 Eq.~\eqref{eq:232} shows moreover that $\E[ \overline \Phi(L_{k+1}-L_k)]<\infty$, and hence we can define a random variable $Q$ via 
 \begin{align}
Q ={}&  \overline \Phi(L_{k+1}-L_k)- \E[ \overline \Phi(L_{k+1}-L_k)] \\ ={}&  \sum_{j=1}^{\infty} \Big(\Phi_{\rho_j^\infty} \Big(\phi^\infty_j(0) (L_{k+1}-L_k) \Big)- \E\Big[ \Phi_{ \rho_j^\infty}\Big(\phi^\infty_j(0) (L_{k+1}-L_k)\Big)\Big]\Big),
 \end{align}
 where the last sum converges absolutely almost surely.
%Now, we need the exact asymptotic 
%behaviour of $\overline{H}(x)$ when $x\rightarrow \infty$. 
Since 
 $Q\geq - \E[ \overline \Phi(L_{k+1}-L_k)] $, we have that 
 \begin{equation}\label{eq:836b}
 \lim_{x\to\infty}x^{(k-\alpha)\beta} \P(Q<-x)=0. 
 \end{equation}
By the substitution $t=(x/u)^{1/(k-\alpha)}$ we have that 
\begin{align}
x^{1/(\alpha-k)} \overline \Phi(x)= {}& x^{1/(\alpha-k)} \int_0^\infty \Phi_{ \rho^\infty_{1+[t]}}(\phi^\infty_{1+[t]} (0)x)\,dt
\\ = {}& (k-\alpha)^{-1}  \int_0^\infty \Phi_{\rho^\infty_{1+[( x  / u)^{1/(k-\alpha)}]}}(\phi^\infty_{1+[( x  / u)^{1/(k-\alpha)}]}(0) x)u^{-1+1/(\alpha-k)}\,du \\ \label{eq:7268712}
\to {}& (k-\alpha)^{-1} \int_0^\infty \Phi_{\rho_\infty^\infty}(k_\alpha u)u^{-1+1/(\alpha-k)}\,du=:\kappa\qquad \text{as }x\to \infty,
\end{align}
where $k_\alpha = \alpha (\alpha-1)(\alpha-2)\cdots (\alpha-k+1)$. 
Here we have used that $(\rho_j^\infty)_{j\geq 1}$ are bounded away from zero together with the estimate  \eqref{hestimate} on $\Phi_{\rho_j^\infty}$ and  Lebesgue's dominated convergence theorem. Note that the constant $\kappa$ defined in \eqref{eq:7268712} coincides with the $\kappa$ defined  in 
Remark~\ref{rem-const}. The connection between the tail behaviour 
of a symmetric $\rho$-stable random variable $S_{\rho}$, $\rho\in (1,2)$,  and its  scale parameter $\bar\sigma$ is given via 
\[
\mathbb{P}(S_{\rho}>x) \sim  \tau_{\rho} \bar\sigma^{\rho} x^{-\rho} /2\qquad \text{as } x\rightarrow \infty, 
\]  
where the function $\tau_{\rho}$ has been defined in \eqref{def-tau-rho}
(see e.g.\ \cite[Eq.~(1.2.10)]{SamTaq}). Hence,  $\mathbb P (|L_{k+1}-L_k|>x)\sim \tau_{\beta}  x^{-\beta}$ as $x\rightarrow \infty$, and by \eqref{eq:7268712} we readily deduce  that as $x\to \infty$
\begin{equation}\label{eq:24245}
\mathbb P(Q>x)\sim \gamma x^{(k-\alpha)\beta} \quad \text{with} \quad \gamma = 
\tau_{\beta}  \kappa ^{(k-\alpha )\beta }.
\end{equation}
Next we will show that 
for some $r>(k-\alpha)\beta$ we have 
\begin{equation}\label{eq:434}
\mathbb P(| Z - Q|>x)\leq Kx^{-r} \qquad \text{for  all } x\geq 1,
\end{equation}
which implies  \eqref{eq:93673}, cf.\ \eqref{eq:836b} and \eqref{eq:24245}.
To show \eqref{eq:434}  it is sufficient to find  $r>(k-\alpha)\beta$ such that 
\[
\E[| Z - Q|^r]<\infty  
\]
by Markov's inequality. Furthermore,   by Minkowski inequality and the definitions of $Q$ and $Z$  it suffices to show that 
\begin{equation}
 \label{eq:6739}
\sum_{j=1}^{\infty} \Big\| \Phi_{\rho_j^\infty}(U_{j,k}^\infty )
- \Phi_{\rho_j^\infty}\Big(  \phi^\infty_j(0) (L_{k+1}-L_k)\Big)\Big\|_{L^r}<\infty  
\end{equation}
(recall  that $(k-\alpha)\beta>1$).
To show \eqref{eq:6739} we note that for all $x\in [0,1]$ and $j\in \N$ there exists $\theta_{j,x}\in [j,j+x]$  such that  
\begin{equation}\label{eq:2390723}
 |  \phi_j^\infty(x) - \phi^\infty_j(0)| = | h_k(j+x) - h_k(j)| \leq |h_k'(\theta_{j,x})| \leq K j^{\alpha-k-1}. 
\end{equation}
%$\phi_j(0)=a_j$, and hence obtain the estimates 
%\marginpar{{\color{blue} $j^{\alpha-k-1}$??}}
%\begin{align} \label{estimate-262}
%\| \phi_j-a_j\|_{L^{\beta}([0, 1])}\leq K j^{\alpha -2},\qquad j\geq 1.
% \end{align}
Choose $\delta>0$ according to Lemma~\ref{est-H-inf} and let $r_\epsilon=(k-\alpha)\beta+\epsilon$ for all  $\epsilon\in (0,\delta)$. 
By Lemma~\ref{est-H-inf} and \eqref{eq:2390723} we have that 
\begin{align}
 {}& \Big\| \Phi_{\rho_j^\infty}( U_{j,k}^\infty )
- \Phi_{\rho_j^\infty}\Big(  \phi^\infty_j(0) (L_{k+1}-L_k)\Big)\Big\|_{L^{r_\epsilon}} \\ {}& \leq  K \Big( \| \phi_j^\infty-\phi^\infty_j(0)\|_{L^{\beta}([0, 1])}+ \| \phi_j^\infty-\phi^\infty_j(0)\|_{L^{\beta}([0, 1])}^{\frac{1}{k-\alpha+\epsilon/\beta}}\Big)
\leq   K \Big( j^{\alpha-k-1}+  j^{\frac{\alpha-k-1}{k-\alpha+\epsilon/\beta}}\Big).\qquad \label{eq:83767473}
\end{align}
Our assumption $\alpha<k-1/\beta$ implies that $\alpha-k<0$. Furthermore, since 
\begin{equation}
 \frac{\alpha-k-1}{k-\alpha+\epsilon/\beta}\to -1-1/(k-\alpha)<-1\qquad \text{as }\epsilon\to 0, 
\end{equation}
 we may, according to \eqref{eq:83767473},   choose $\epsilon>0$ such that \eqref{eq:6739} holds for $r=r_\epsilon$ which satisfies the condition $r>(k-\alpha)\beta$. This completes the proof of \eqref{eq:434} and hence also of 
 \eqref{eq:827284-1}.

%%
%%$r>1$ such that $r>(1-\alpha)\beta$,  $p r <\beta$ and $(\alpha-2)\beta/r<-1$. 
%
%For all $\epsilon>0$ we have by   Lemma~\ref{est-H-inf}, \eqref{eq:72138} and  \eqref{estimate-262} that 
%\begin{equation}
% \| \Phi_{\rho_j}( W_{j,0} )
%- \Phi_{\rho_j}(  a_j L_1)\|_{L^r}\leq K\Big( j^{ (\alpha-k)(\beta/r-\epsilon)-1}+ j^{(\alpha-2)\beta/r}\Big)
%\end{equation}
%which shows \eqref{eq:6739}, by choosing  $\epsilon= \beta/(2r)$ and noticing that $(\alpha-2)\beta/r<\alpha-2<-1$.  
%This completes the proof of \eqref{eq:827284-3} and hence of \eqref{eq:78368222}. 
%%\qed
%%\end{proof}

To complete the proof of Theorem~\ref{sec-order}(i)  we  show that the second term in \eqref{eq:decom77272} converges to zero. For this purpose it is enough to show  that 
\begin{equation}\label{eq:7823927}
 n^{1-\frac{1}{(k-\alpha)\beta}} \Big( \E[| Y^n_1 |^p] -m_p\Big)\to 0\qquad \text{as }n\to\infty,  
\end{equation}
since $1-\frac{1}{(k-\alpha)\beta}<1$. 
Recall that $m_p= \| h_k \|_{L^\beta(\R)}^p \E[ |Z|^p]$, where $Z$ is a standard symmetric $\beta$-stable random variable and $ \| h_k \|_{L^\beta(\R)}= \|\phi^\infty_1\|_{L^\beta(\R)}$. 
 By   Lemma~\ref{helplem1} we have that 
 \begin{equation}\label{eq:873663}
 \Big| \| \phi^n_j \|_{L^\beta(\R)}^\beta - \| \phi^\infty_j \|_{L^\beta(\R)}^\beta \Big|  \leq K n^{(\alpha-k)\beta+1}\to 0, 
  \end{equation}
where the  convergence to zero is due to the fact that $(k-\alpha)\beta>1$ under our assumptions. Since the function $x\mapsto x^{p/\beta}$ is continuously differentiable on $(0,\infty)$ and $\| h_k \|_{L^\beta(\R)}^\beta>0$, it follows by the mean value theorem that 
\begin{equation}
 \Big| \| \phi^n_1 \|_{L^\beta(\R)}^p - \| h_k \|_{L^\beta(\R)}^p \Big| \leq K \Big| \| \phi^n_1 \|_{L^\beta(\R)}^\beta - \| h_k \|_{L^\beta(\R)}^\beta \Big|, 
\end{equation}
which together with \eqref{eq:873663} and the definition of $Y^n_1$ in \eqref{eq:23412343} shows that   
  \begin{align}
 n^{1-\frac{1}{(k-\alpha)\beta}}   \Big| \E[| Y^n_1 |^p] -m_p\Big| = {}& n^{1-\frac{1}{(k-\alpha)\beta}}   \E[ |Z|^p]
 \Big|\| \phi^n_1\|_{L^\beta(\R)}^p - \| h_k \|_{L^\beta(\R)}^p\Big|\\
 \leq   {}& K n^{1-\frac{1}{(k-\alpha)\beta}}  
 \Big|\| \phi^n_1\|_{L^\beta(\R)}^\beta - \| h_k \|_{L^\beta(\R)}^\beta\Big| 
 \\  \leq   {}& K  n^{1-\frac{1}{(k-\alpha)\beta}+1-(k-\alpha)\beta} = K  n^{2-\frac{1}{(k-\alpha)\beta}-(k-\alpha)\beta}. \label{eq:023723}
 \end{align}
By   \eqref{eq:023723} and the assumption $(k-\alpha)\beta>1$ we obtain \eqref{eq:7823927}, and the proof of 
Theorem~\ref{sec-order}(i) is complete.
\qed
%
%where the last inequality follows from the fact that   $p<\beta$.  Eq.~\eqref{eq:7370} implies that  
%$r_n\to 0,$ which completes the proof. 

%
%
%
%%%%%%%%%%%%%%%
%\subsection{Proof of Theorem~\ref{thm-21392947}}\label{thm-5.1-232}
%%%%%%%%%%%%%%%%%%%%%%%
%
%

%To show \eqref{eq:189362720} we will show that 
%\begin{align}\label{eq:123401282781}
%{}&  \lim_{m\to \infty} \limsup_{n\to \infty} \Big(n^{-1} \E\Big[ \Big(\sum_{r=1}^n 
%R^{n,m}_r\Big)^2\Big] \Big) = 0 ,
% \\ \label{eq:123401282782}{}& 
%  \lim_{m\to \infty} \limsup_{n\to \infty} \Big(n^{-1} \E\Big[ \Big( \sum_{r=1}^n Q^{n,m}_r\Big)^2\Big] \Big) = 0, \\ \label{eq:123401282783}
%{}&     \lim_{m\to \infty} \limsup_{n\to \infty} \Big(n^{-1}\E\Big[ \Big( \sum_{r=1}^n Z_r^m\Big)^2\Big]  \Big)= 0 .
%\end{align}
%In the proves of \eqref{eq:123401282781}--\eqref{eq:123401282783} we let $m\geq 1$. 
%
%
% \qed
%  %\end{proof}

%%%%%%%%%%%%%%%%%%%%%%%%%%%%%%%%%%%%%%%%%%
\subsection{Proof of Theorem~\ref{sec-order}(ii)}
%\begin{proof}[\textbf{Proof of Theorem~\ref{sec-order}(ii)}]
To prove Theorem~\ref{sec-order}(ii) we start by noticing that 
\begin{align}\label{eq:97136-4}
\sqrt{n} \Big( n^{-1+p(\alpha + 1/\beta)}V(p;k)_n - m_p \Big) \eqschw \frac{1}{\sqrt{n}} S_n + 
\sqrt{n} \Big( \frac{n-k+1}{n}\E[ | Y^{n}_1|^p]-m_p\Big)\qquad 
\end{align}
due to  \eqref{statdec}. 
First we will show that 
\begin{equation}\label{eq:076388452}
 \frac{1}{\sqrt{n}} S_n \schw \mathcal N(0,\eta^2) \qquad \text{as } n\rightarrow \infty, 
\end{equation}
where $\eta^2\in (0,\infty)$ is given in Theorem~\ref{sec-order}(ii).  Afterwards we will show that the second term on the right-hand side of \eqref{eq:97136-4} converges to zero, which will complete the proof of Theorem~\ref{sec-order}(ii).   To prove \eqref{eq:076388452} it is according to a standard result (see e.g.\ \cite[Theorem~3.2]{Billingsley})  enough to show the following (i)--(iii): 

\noindent
(i): We have that 
\begin{equation}\label{eq:189362720}
\lim_{m\to \infty} \limsup_{n\to \infty} \big(n^{-1} \E[(S_{n}-S_{n,m})^2]\big)  = 0.
\end{equation}
%Theorem~\ref{thm-21392947} 
 \noindent
(ii): For all $m\geq 1$ there exists  $\eta^2_m\in [0,\infty)$ such that 
 \begin{equation}\label{eq:8272627282}
 \frac{1}{\sqrt{n}} S_{n,m} \schw \mathcal N(0,\eta^2_m) \qquad \text{as } n\rightarrow \infty. 
 \end{equation}

\noindent
 (iii): We have that 
  \begin{equation}
  \eta_m^2\to \eta^2 \qquad \text{as }m\to \infty. 
  \end{equation}
 
\noeqref{eq:112132-3-99}
To prove (i) we use 
Proposition~\ref{prop-key-est} and the assumption $\alpha<k-2/\beta$ to obtain that 
\begin{align}\label{eq:8287320973}
{}& \frac{1}{n}\E\Big[ \Big(\sum_{r=k}^n R^{n,m}_r\Big)^2\Big]\leq K \Big(  (m+1)^{(\alpha-k)\beta/4+1/2} + n^{2(\alpha-k)\beta +3}+n^{-1}\log n\Big),   \\
\label{eq:112132-3-99}
 {}& \frac{1}{n}\E\Big[ \Big(\sum_{r=k}^n Q^{n,m}_r\Big)^2\Big]\leq   K\Big(   n^{ (\alpha-k)\beta+2+\epsilon}   +  (m+1)^{(\alpha-k)\beta +2+\epsilon}+n^{-1} \Big).
% ,\\  \label{eq:112132-4-33}
% {}& \frac{1}{n}\E\Big[ \Big(\sum_{r=k}^n Z^{m}_r\Big)^2\Big]\leq K (m+1)^{(\alpha-k)\beta+1+\epsilon}.
\end{align}
Thus, by the  decomposition \eqref{eq:7836638-22} of $S_n - S_{n,m}$, \eqref{eq:8287320973},\eqref{eq:112132-3-99} and the assumption $\alpha<k-2/\beta$ we deduce \eqref{eq:076388452}, which  completes the proof of (i).

To prove (ii)   we note that for fixed $n,m\geq 1$, $\{|Y_i^{n,m}|^p\!: i=k,\dots,n\}$ is a stationary  $m$-dependent sequence, and hence 
%Hence using the main result of \cite{Berk} 
%\[
%Y_{n,m}^{(1)} = \frac{1}{\sqrt{n}} \sum_{i=k}^n Y_i^{n,m} \quad \text{where}\quad Y_i^{n,m} = n^{p(\alpha +1/\beta)} \Big(|\Delta_{i,k}^{n,m} X|^p - \mathbb E[|\Delta_{i,k}^{n,m} X|^p] \Big)
%\]
%and that for fixed $n\geq1$, the random variables $\{Y_i^{n,m}: i=k,\dots, n\}$ is a stationary $m$-dependent sequence, in particular, $\{Y_i^{n,m}: i=k,\dots, r\}$ is indepnendet of $\{Y_i^{n,m}: i=r+m,\dots, n\}$ for all $r= k,\dots, n-m $.
%By stationarity and $m$-dependence of $\{V_i^{n,m}: i=k,\dots,n\}$ we have 
\begin{align}\label{eq:72727-1}
n^{-1}\text{var}(S_{n,m}) = n^{-1}(n-k)\theta^{n,m}_0 + 2n^{-1} \sum_{i=1}^{m}(n-k-i) \theta^{n,m}_i
\end{align}   
where we set  $\theta^{n,m}_i= \text{cov}(|Y^{n,m}_k|^p,|Y_{k+i}^{n,m}|^p)$ for all $n\in \N\cup\{\infty\}$, $m,i\geq 1$.
%Set 
%\begin{equation}
%Y^m_i = \int_{i-m}^i h_k(i-s) \,dL_s\qquad \text{and}\qquad V^m_i = | Y^m_i|^p - \E[  | Y^m_i|^p]. 
%\end{equation}
By the  symmetrisation  inequality we have for all $u>0$ that 
\[
\P(  | Y^{n,m}_i - Y^{\infty,m}_i|>u)\leq 2 \P( | Y^{n}_i - Y^{\infty}_i|>u), 
\] 
where the quantities $Y^{n}_i$ and $Y^{\infty}_i$ have been introduced in \eqref{eq:232452}. By  the  equivalence of moments of  stable random variables we have 
for  all $q<\beta$  that 
\begin{equation}\label{eq:297293723}
\E[ | Y^{n,m}_i - Y^{\infty,m}_i|^q] \leq K_q  \E[ | Y^{n,m}_i - Y^{\infty,m}_i|^p]^{q/p}\leq 
K_q 2^{q/p}  \E[ | Y^{n}_k - Y^{\infty}_k|^p]^{q/p}\to 0
\end{equation}
as $n\to \infty$, where the convergence to zero follows by \eqref{eq:12380273}. Since $p<\beta/2$, \eqref{eq:297293723} implies that $\theta^{n,m}_i\to \theta^{\infty,m}_i$ as $n\to \infty$,  
%
%$\leq 2u^{-p} \E[    | Y^{n}_k - Y^{\infty}|^p] $. 
%
%$\E[ | Y^{n,m}_k - Y^{\infty,m}_k|^p] \leq 2 \E[ | Y^{n}_k - Y^{\infty}|^p]$
%
%
%
%
%\eqref{eq:12380273} used on the function $\tilde g := g \1_{[0,m]}$ we have that it follows that 
%
%
%By  the the line above  \eqref{con-L-1} we deduce that for all $d\geq 1$ and as $n\to\infty$, $(V_i^{n,m})_{i=1}^d \schw (V_i^{\infty,m})_{i=1}^d$ as $n\to \infty$.
% For any $q>0$ with $pq<\beta$, the estimates of the proof of 
% Theorem~\ref{maintheorem}(ii) show that  
%\begin{equation}\label{eq:92882-1}
% \mathbb E[|Y_i^{n,m}|^q] \leq K,
%\end{equation}
%and by the assumption $p<\beta/2$ we may choose $q>2$. Hence,   $\theta^{n,m}_i\to \theta^{\infty,m}_i$ as $n\to\infty$
and  by  \eqref{eq:72727-1} we deduce  that 
\begin{equation}\label{eq:76726-1}
 n^{-1}\text{var}(S_{n,m}) \to \theta^{\infty,m}_0 + 2\sum_{i=1}^{m} \theta^{\infty,m}_i=:\eta^2_m\qquad \text{as }n\to \infty. 
\end{equation} 
By \eqref{eq:297293723}, \eqref{eq:76726-1}, and since for all $n\geq 1$, the sequences $\{|Y_i^{n,m}|^p: i=k,\dots,n\}$ are  $m$-dependent,  the convergence \eqref{eq:8272627282} follows by the main theorem of \cite{Berk}, and the proof of (ii) is complete. 

%Now, the Lindeberg condition holds due to the fact that 
%\begin{align*}  
%\sum_{i=k}^n E[|n^{-1/2} Y_i^{n,m}|^2 \1_{\{|n^{-1/2} Y_i^{n,m}|>\epsilon\}}]
%\leq K n^{-r/2}  E[|Y_i^{n,m}|^{2+r}]  \leq K n^{-r/2} \to 0, 
%\end{align*}
%for some $r>0$ with $(2+r)p<\beta$. 
%Set \begin{equation}
%Q_{i,n}= \sum_{j= m(i-1)+k}^{mi+k} Y_{i}^{n,m}
%\end{equation}
%and note that $\{Q_{i,n}: i , n \}$ is an independent triangular array which satisfies the Lindeberg condition. 
%Hence by the decomposition 
%\begin{align}
%Y_{n,m}^{(1)} =  \frac{1}{\sqrt{n}}  \sum_{i=1}^{ \lceil n/ k \rceil }Q_{i,n} +\frac{1}{\sqrt{n}}  
%\sum_{i=[n/k]n}^n Y_{i}^{n,m}
%\end{align}
%yields the convergence \eqref{eq:6725}. 

The proof of (iii)  follows by the same arguments as in \cite[p.~1650]{hh97}. Indeed, for all $m,j\geq 1$ we have by  the triangle inequality  that 
\begin{align}
\big |  | \eta_m|-|\eta_j|\big | {}& =  \lim_{n\to \infty} \Big( n^{-1/2} \Big|  \| S_{n,m} \|_{L^2} -   \| S_{n,j} \|_{L^2} \Big| \Big)   \leq
\limsup_{n\to \infty} \Big( n^{-1/2}  \| S_{n,m} - S_{n,j} \|_{L^2} \Big) \\ {}& 
\leq \limsup_{n\to \infty} \Big( n^{-1/2}  \| S_{n,m}- S_{n}\|_{L^2}\Big) + \limsup_{n\to \infty} \Big( n^{-1/2}  \| S_n- S_{n,j}\|_{L^2}\Big),
\end{align}
which according to \eqref{eq:189362720} shows that  $(|\eta_m|)_{m\geq 1}$ is a Cauchy sequence in $\R_+$. Hence, $(\eta_m^2)_{m\geq 1}$ is convergent. 

To show that the second term on the right-hand side of \eqref{eq:97136-4} converges to zero it suffices to prove that 
\begin{equation}\label{eq:239723}
\sqrt{n} \Big(\E[ | Y^{n}_1|^p]-m_p\Big)\to 0 \qquad \text{as }n\to \infty. 
\end{equation}
%
%\emph{Include also this part: } Following the same arguments as in Step~4 of the previous subsection, we readily deduce that
%\[
%|Y_n^{(2)}|\leq K n^{\alpha \beta +3/2} \left |\|g_{0,n}\|_{L^{\beta}(\R)}^{\beta }
%-\| h_{0,n}\|_{L^\beta(\R)}^{\beta } \right|. 
%\]
%Applying Lemma \ref{helplem1}(ii) immediately shows that $Y_n^{(2)}\rightarrow 0$ as $n\to \infty$. 
%
%
%
%
%
%To complete the proof of Theorem~\ref{sec-order}(i)  we  show, in the following, that 
%\begin{equation}\label{eq:7823927}
% n^{1-\frac{1}{(1-\alpha)\beta}} \Big( \E[| Y^n_1 |^p] -m_p\Big)\to 0\qquad \text{as }n\to\infty. 
%\end{equation}
%Recall that $m_p= \| h_k \|_{L^\beta(\R)}^p \E[ |Z|^p]$, where $Z$ is a standard symmetric $\beta$-stable random variable.
% Let $g_{0,n}$ and $h_{0,n}$ be  defined as in \eqref{def-g-i-n}. 
By  Lemma~\ref{helplem1} we have that 
 \begin{equation}
 \Big| \| \phi^n_1 \|_{L^\beta(\R)}^\beta - \| \phi^\infty_1 \|_{L^\beta(\R)}^\beta \Big| \leq  K n^{-1}\to 0. \label{eq:873663-33}
  \end{equation}
Since the function $x\mapsto x^{p/\beta}$ is continuously differentiable on $(0,\infty)$ and $\| \phi^\infty_1 \|_{L^\beta(\R)}^\beta>0$ it follows by the mean-value theorem that 
\begin{equation}
 \Big| \| \phi^n_1 \|_{L^\beta(\R)}^p - \| \phi^\infty_1 \|_{L^\beta(\R)}^p \Big| \leq K \Big| \| \phi^n_1 \|_{L^\beta(\R)}^\beta - \| h_k \|_{L^\beta(\R)}^\beta \Big|.
\end{equation}
Together with \eqref{eq:873663-33} and the definition of $Y^n_1$ in \eqref{eq:23412343} it shows that   
  \begin{align}
{}&\sqrt{ n}   \Big| \E[| Y^n_1 |^p] -m_p\Big| =  \sqrt{ n}  \E[ |Z|^p]
 \Big|\| \phi^n_1\|_{L^\beta(\R)}^p - \| \phi^\infty_1 \|_{L^\beta(\R)}^p\Big|\\
{}&\qquad  \leq    K \sqrt{ n}  
 \Big|\| \phi^n_1\|_{L^\beta(\R)}^\beta - \| \phi^\infty_1 \|_{L^\beta(\R)}^\beta\Big| 
  \leq    K  n^{-1/2}\to 0 \label{eq:023723-33}
 \end{align}
 as $n\to \infty$.  Eq.~\eqref{eq:023723-33} shows \eqref{eq:239723} and  completes the proof of Theorem~\ref{sec-order}(ii). 
\qed
%\end{proof}

%%%%%%%%%%%%%%
\subsection{An estimate}\label{proof-pro-key}
%%%%%%%%%%%%%%%%

This subsection is devoted to proving the following lemma, which is used in the proof of \eqref{eq:112132-1} of Proposition~\ref{prop-key-est}. 
\begin{lem}\label{eq:8937739}
Let $ \zeta_{r,j}^{n,m}$ be defined in \eqref{eq-S-S'-2}.  Then there exists a finite constant $K$ such that for all $n\geq 1$, $r=k,\dots,n$, $m\geq 0$ and $j\geq 1$ we have 
 \begin{equation}
 \E[ |  \zeta_{r,j}^{n,m}|^2] \leq K 
 \begin{cases} (m+1)^{(\alpha-k)\beta +1} j^{(\alpha-k)\beta}\qquad  &  j= 1, \dots, m \\ 
 j^{2(\alpha-k)\beta+1} & j>m.
 \end{cases}
 \end{equation} 
\end{lem}

To show Lemma~\ref{eq:8937739} we  will use the following  telescoping sum decomposition of $\zeta^{n,m}_{r,j}$: 
\begin{equation}\label{def-theta-23}
 \zeta_{r,j}^{n,m} = \sum_{l=j}^\infty \vartheta_{r,j,l}^{n,m}  ,\qquad \quad \vartheta_{r,j,l}^{n,m}  := \E[\zeta_{r,j}^{n,m}  | \g_{r-j}^1\vee \g_{r-l}] 
 - \E[\zeta_{r,j}^{n,m}  | \g_{r-j}^1\vee \g_{r-l-1}].
\end{equation}
The series \eqref{def-theta-23} converges almost surely and the representation follows from  the fact that  $\lim_{l\to \infty} \E[ \zeta^{n,m}_{r,j}|\g^1_{r-j}\vee \g_{r-l}]=\E[ \zeta^{n,m}_{r,j}|\g^1_{r-j}]=0$ almost surely, similarly to the argument used in  Remark~\ref{rem-well-defined}.  The next lemma gives a moment estimate for $\vartheta^{n,m}_{r,j,l}$.  

\begin{lem}\label{lem-est-theta}
Let  $\vartheta^{n,m}_{r,j,l}$ be defined in \eqref{def-theta-23} and suppose that  $\beta<\gamma<\beta/p$. Then    there exists  $N\geq 1$ such that  for all 
$n\geq N$, $r=k,\dots,n$, $j\geq 1$ and   $m\geq 0$ we have that 
 \begin{align}\label{est-ss4w5}
  \E[| \vartheta_{r,j,l}^{n,m}|^\gamma]\leq K  
  \begin{cases} j^{(\alpha-k)\beta}l^{(\alpha-k)\beta}   \quad & l\geq m \\
  (m+1)^{(\alpha-k)\beta+1} j^{(\alpha-k)\beta} l^{(\alpha-k)\beta}  \quad & l=j,\dots,m-1.
  \end{cases}
 \end{align}
\end{lem}

To prove Lemma~\ref{lem-est-theta} we use  the following estimate on $\Phi_{\rho}$ defined in \eqref{def-H-rho}.  

\begin{lem}\label{lemma-est-phi''}
There exists a finite constant $K$ such that for all $\rho\in [\epsilon, \epsilon^{-1}]$, all  $x,y, z\geq 0 $ and all $a\in \R$  we have   that   
\begin{align}
{}&  \int_0^z \int_0^x \int_0^y |\Phi_\rho'''(a+  u_1+u_2+u_3)|\,du_1\,du_2\,du_3
\\ {}& \qquad   \leq K\Big( ( x \wedge 1)( y \wedge 1)(z\1_{\{z\leq 1\}} + z^p\1_{\{z> 1\}}\big)\Big).\label{eq:837636}
\shortintertext{
and }
{}&  \int_0^x \int_0^y |\Phi_\rho''( a+u_1+u_2)|\,du_1\,du_2
\\ {}& \qquad   \leq K\Big( (x\wedge 1)(y\1_{\{y\leq 1\}} + y^p\1_{\{y> 1\}}\big)\Big)\label{eq:83637392}
\end{align}
%where  $K$ is a finite constant only depending on $\beta,\epsilon$ and $p$.
\end{lem}

\begin{proof}[Proof of Lemma~\ref{lemma-est-phi''}] 
Throughout the proof $K$ will denote  a finite constant only depending on $\beta,\epsilon$ and $p$, but might change from line to line. 
First we will show that for all $v=1,2,3$, all $a\in \R$ and all $z>0$  we have that 
\begin{equation}\label{eq:est-2273}
 \int_0^z \Phi_\rho^{(v)}(a+u)\,du\leq K (\1_{\{z\leq 1\}} z+\1_{\{z> 1\}}z^p),
\end{equation}
where $\Phi^{(v)}_\rho$ denotes the $v$-th derivative of $\Phi_\rho$.
To this aim we first  show that for $v=1,2,3$ we have that 
\begin{equation}\label{eq:38873}
|\Phi^{(v)}_\rho(x)| \leq K  \big(1\wedge |x|^{p-v}\big)\qquad \text{for all } x\in \R,
\end{equation} 
which, in particular, yields that  
\begin{equation}\label{eq:82652278}
 |\Phi^{(v)}_\rho(x)| \leq K  \big(1\wedge |x|^{p-1}\big)\qquad \text{for all } x\in \R.
\end{equation}
For all $u>0$ we  let 
\begin{equation}
q(u)=  u^{v-1-p}e^{-\rho^\beta u}\qquad \text{and}\qquad 
\psi(u) =  u^{v-1-p}( e^{-\rho^\beta u^\beta}- e^{-\rho^\beta u}) .
\end{equation}
By recalling \eqref{est-H''} we have by the triangle inequality that  
\begin{equation}\label{eq:73736}
|\Phi^{(v)}_\rho(x)|  \leq  2a_p^{-1} \Big( \Big| \int_0^\infty \cos(x u) \psi(u)\,du \Big| +  \Big| \int_0^\infty \cos(x u) q(u)\,du \Big| \Big).
\end{equation}
To estimate the second integral on the right-hand side of \eqref{eq:73736} we note that 
$q(u)\rho^{\beta(v-p)} /\Gamma(v-p)$  is   the density of a gamma distribution with shape parameter $v-p$ and rate parameter $\rho^\beta$. Hence using the expression for the characteristic function for the gamma distribution we get for all $x\neq 0$ that 
\begin{align}\label{eq:78363}
{}& \Big| \int_0^\infty \cos(xu) q(u)\,du \Big| \leq \Big| \int_0^\infty e^{ixu}  q(u)\,du \Big| \\ 
{}&\qquad =  
\frac{\Gamma(v-p)}{\rho^{\beta(v-p)}}
\Big| (1- ix \rho^{-\beta})^{p-v}\Big| = 
\frac{\Gamma(v-p)}{\rho^{\beta(v-p)}}
\left(1+ x^2 \rho^{-2\beta} \right)^{\frac{p-v}{2}} 
 \leq  \Gamma(v-p)|x|^{p-v} .
\end{align}
%Set $\psi(u)=f(u)-g(u)$. We show that $\psi$ is two times absolutely continuous on $\R$ with 
%\begin{equation}
% | \psi''(u)|\leq \rho u^{\beta\wedge 1-1-p}+u^{r} e^{-\epsilon u^{\beta\wedge 1}} \qquad \lambda\text{-a.e.}
%\end{equation}
%where $\lambda$ denotes that the Lebesgue measure. 
To estimate the first integral on the right-hand side of \eqref{eq:73736} we set $\zeta(u) = e^{-\rho^\beta u^\beta}-e^{-\rho^\beta u}$ for $u\geq 0$ such that 
$\psi(u) = u^{v-1-p}\zeta(u)$.  For all $j=0,1,2,3$ we obtain the estimates 
\begin{equation}\label{eq:26627}
 |\zeta^{(j)}(u)|\leq \begin{cases}
 K  u^{\beta\wedge 1- j } & u\in (0,1), \\ 
 Ku^2 e^{-\epsilon^\beta u^{\beta\wedge 1}}  \qquad   & u\geq 1,
 \end{cases}
\end{equation}
which implies that 
\begin{equation}\label{eq:723638}
 |\psi^{(j)}(u)| \leq  \begin{cases}
 K u^{\beta\wedge 1- 1-p } & u\in (0,1), \\ 
 Ku^3 e^{-\epsilon^\beta u^{\beta\wedge 1}}  \qquad   & u\geq 1.
 \end{cases}
\end{equation}
%which shows that $\int_\R | \psi''(s)| \,ds\leq K_{\epsilon,\beta}$. 
%Since $\psi$ is a continuous function on $\R$ which is  continuous differentiable on $\R\setminus\{0\}$ it follows that $\psi$ is absolutely continuous. The estimate \eqref{eq:26627} used on $j=0$ and $j=1$ yields  that  $\psi'$ is continuous and since $\psi'$ is continuous differentiable on $\R\setminus\{0\}$ it follows that $\psi$ is two times absolutely continuous. The estimate \eqref{eq:723638} shows that there exists a finite constant $K_{\epsilon,\beta}$ such that $\int_\R |\psi''(s)| \,ds\leq K_{\epsilon,\beta}$. 
Since $p<\beta \wedge 1$ by assumption,  we deduce by \eqref{eq:723638} used on $j=0$ and $j=1$ that  $\lim_{u\to \infty} \psi(u)=\lim_{u\downarrow 0} \psi'(u)=\lim_{u\to \infty} \psi'(\infty) =0$. Hence,  by integration by parts,   we have for all $x>0$ that 
\begin{align}\label{eq:82827}
{}& \Big|\int_0^\infty \cos(xu) \psi(u)\,du\Big|  = 
\begin{cases}
x^{-v}  \Big| \int_0^\infty \cos(xu) \psi^{(v)}(u)\,du \Big| & \quad v \text{ even} \medskip \\
x^{-v}  \Big| \int_0^\infty \sin(xu) \psi^{(v)}(u)\,du \Big| & \quad v \text{ odd}
\end{cases}
\\ {}&\qquad \leq  x^{-v} \int_0^\infty |\psi^{(v)}(u)|\,du\leq K x^{-v},
\end{align}
where the last inequality follows from \eqref{eq:723638} used on $j=v$. 
The  estimates \eqref{eq:73736},  \eqref{eq:78363} and  \eqref{eq:82827} 
imply  \eqref{eq:38873}.

To show \eqref{eq:est-2273} it suffices, cf.\  \eqref{eq:82652278}, to show the estimate  
\begin{align}\label{eq:12-23}
{}& \int_0^z  \Big( 1\wedge | a+u|^{p-1}\Big)\,du\leq K (\1_{\{z\leq 1\}} z+\1_{\{z> 1\}}z^p). 
\end{align}
It is important that the constant $K$ does not depend on $a\in \R$. To show \eqref{eq:12-23} we may and do assume that $z>1$ since the estimate \eqref{eq:12-23}  holds for $z\leq 1$ by dominating the integrand by 1.
We split the integral in three parts
 \begin{align}
{}&  \int_0^z  \Big( 1\wedge | a+u|^{p-1}\Big)\,du = \int_{(-a-1,1-a)\cap [0,z]} 1\,du \\ {}&\qquad  + \int_{(1-a,\infty )\cap [0,z]} (a+u)^{p-1} \,du + \int_{(-\infty,-a-1)\cap [0,z]} (-a-u)^{p-1} \,du.\label{eq:82626}
 \end{align}
Since $p\in (0,1]$ we have by  subadditivity that   $x^p-y^p \leq (x-y)^p$
for all $0\leq y\leq x$. 
Hence 
\begin{align}
\int_{(1-a,\infty )\cap [0,z]} (a+u)^{p-1} \,du ={}&  \1_{\{z\geq 1-a\}}\frac{1}{p} 
\begin{cases}
(a+z)^p-a^p \quad & a\geq 1 \\ 
(a+z)^p -1 & a<1. 
\end{cases}
\\ \leq {}& \1_{\{z\geq 1-a\}}\frac{1}{p} z^p 
\shortintertext{and}
\int_{(-\infty,-a-1)\cap [0,z]} (-a-u)^{p-1} \,du   = {}& \1_{\{ -a-1\geq 0\}}\frac{1}{p}
\begin{cases}
(-a)^p - 1 & -a-1\leq z \\ (-a)^p - (-a-z)^p \quad & z\leq -a -1 
\end{cases}\\  \leq {}& \1_{\{ -a-1\geq 0\}} \frac{1}{p} z^p. 
\end{align}
Thus, by \eqref{eq:82626} we obtain for $z\geq 1$ that  
\begin{equation}
 \int_0^z  \Big( 1\wedge | a+u|^{p-1}\Big)\,du \leq 2 +  \frac{2}{p} z^p\leq 2\left(1+\frac{1}{p}\right)z^p,
\end{equation}
which implies \eqref{eq:12-23}, and completes the proof of \eqref{eq:est-2273}. 

We will now deduce \eqref{eq:837636} from \eqref{eq:est-2273}.
For $x\geq 1$ we have that with $\bar a= a+x$ 
\begin{align}
 {}&  \Big| \int_0^y\int_0^z \int_0^x  \Phi_\rho'''( a+u_1+u_2+u_3)\,du_1\,du_2\,du_3 \Big|
\\  {}& \qquad \leq \Big| \int_0^z \int_0^y \Phi_\rho''( \bar a+u_1+u_2)\,du_1\,du_2\Big| 
+\Big|  \int_0^z \int_0^y \Phi_\rho''( a+u_1+u_2)\,du_1\,du_2\Big| .
\end{align}
For $x<1$ there exists an $\tilde a\in \R$ such that  
\begin{align}
   {}&    \int_0^x\int_0^y\int_0^z  \Phi_\rho'''( a+u_1+u_2+u_3)\,du_2\,du_3\,du_1  
\\  {}& \qquad = x \int_0^y\int_0^z \Phi_\rho'''( \tilde a+u_2+u_3)\,du_2\,du_3.
\end{align}
Repeating this argument shows that for any $\tilde a\in \R$ and $v=2,3$ we have for $y\geq 1$ that 
with $\bar a= \tilde a + y$ 
\begin{align}
 {}& \int_0^y\int_0^z \Phi_\rho^{(v)}( \tilde a+u_2+u_3)\,du_3\,du_2
 \\  {}& \qquad \leq \Big| \int_0^z \Phi_\rho^{(v-1)}( \bar a+u_2)\,du_3\Big| 
+\Big|  \int_0^z  \Phi_\rho^{(v-1)}( \tilde a+u_3)\,du_3\Big|, \label{eq:8272}
\end{align}
and for $y< 1$ there exists  $\bar a\in \R$ such that 
\begin{equation}
 \int_0^y\int_0^z \Phi_\rho^{(v)}( \tilde a+u_2+u_3)\,du_3 \,du_2\leq y \int_0^z \Phi_\rho^{(v)}( \bar a+u_3)\,du_3.\label{eq:297320937}
\end{equation}
By collecting all the terms and using \eqref{eq:est-2273} we obtain \eqref{eq:837636}. Eq.~\eqref{eq:83637392} follows by similar arguments. In this case, we use \eqref{eq:8272} with $v=2$ and conclude the proof by using  \eqref{eq:est-2273}  as above. 
 \end{proof}

We are now ready to prove Lemma~\ref{lem-est-theta}.

\begin{proof}[Proof of Lemma~\ref{lem-est-theta}] 
For fixed $n,m,j,l$, $\{\vartheta^{n,m}_{r,j,l}:r\geq 1\}$ is a  stationary sequence, and hence we may and do assume that $r=1$. Furthermore, we may  assume that $l\geq j\vee 2$, since the case  $l=j=1$  can be covered by choosing a new  constant $K$. 
% and omit $r$ in the notation by written e.g.\ $\vartheta^n_{j,l}$ instead of $\vartheta^n_{1,j,l}$.  
By definition of $\vartheta^{n,m}_{1,j,l}$ we obtain the representation 
\begin{equation}\label{eq:62627}
\vartheta^{n,m}_{1,j,l} =  \E[ V^{n,m}_r | \g^1_{1-j}\vee \g_{1-l}] - \E[ V^{n,m}_r  |\g_{1-l}] - \E[V^{n,m}_r   | \g_{1-j}^1\vee 
\g_{-l}] + \E[ V^{n,m}_r   | \g_{-l}].
 \end{equation}
Set $\rho_{j,l}^n = \| \phi^n_1 \|_{L^\beta([1-l,1-j]\cup [2-j,1])}$. For large enough $N\geq 1$ there exists $\epsilon>0$ such that $\rho_{j,l}^n\geq \epsilon$ for all $n\geq N, j\geq 1, l\geq j\vee 2$ (we have  $\rho_{j,l}^n=0$ for $l=1$). Hence by \eqref{est-H-28217} there exists a finite constant $K$ such that 
\begin{equation}\label{eq:73698}
|\Phi''_{\rho^n_{j,l}}(x)| \leq K \qquad \text{for all $n\geq N,\, j\geq 1,\, l\geq j\vee 2, \, x\in \R$}. 
\end{equation}
Let 
\begin{equation}
A_{l}^n= \int_{-\infty}^{-l} \phi_{1}^n(s) \,dL_s\qquad \text{and}\qquad   \quad A_{l}^{n,m}= \int_{1-m}^{-l} \phi_{1}^n(s) \,dL_s
%,\quad B_{j}^n=\int_{1-j}^{2-j} \phi_{1}^n(s)\,dL_s, 
\end{equation}
%\marginpar{$B^n_j=U^n_{1,j}$}
 and    $(\tilde U_{1,-l}^n, \tilde U_{1,1-j}^n)$ denote a random vector, which is independent of $L$, and which equals 
  $(U_{1,-l}^n, U_{1,1-j}^n)$ in law (cf. definition \eqref{udef}). Let  moreover $\tilde \E$ denote the expectation with respect to $(\tilde U_{1,-l}^n, \tilde U_{1,1-j}^n)$ only.  For all  $j=1,\dots, m$ and $l=j,\dots, m-1$ we deduce from  \eqref{eq:62627}   that 
\begin{align}
\vartheta_{1,j,l}^{n,m} ={}&  \tilde \E\Big[  \Phi_{\rho^n_{j,l}}( A_{l}^n +U_{1,-l}^n+U_{1,1-j}^n)- 
\Phi_{ \rho^n_{j,l}}(A_{l}^n+\tilde U_{1,-l}^n+U_{1,1-j}^n)\\ {}&
\phantom{   \tilde \E\Big[ }   -\Phi_{\rho^n_{j,l}}(A_{l}^n+U_{1,-l}^n+\tilde U_{1,1-j}^n)+ \Phi_{\rho^n_{j,l}}(A_{l}^n+\tilde U_{1,-l}^n+\tilde U_{1,1-j}^n) \\
{}& \phantom{   \tilde \E\Big[ }  - \Big( \Phi_{\rho^n_{j,l}}( A_{l}^{n,m} +U_{1,-l}^n+U_{1,1-j}^n)- 
\Phi_{ \rho^n_{j,l}}(A_{l}^{n,m}+\tilde U_{1,-l}^n+U_{1,1-j}^n)\\ {}&
\phantom{   \tilde \E\Big[ - \Big(  }  -\Phi_{\rho^n_{j,l}}(A_{l}^{n,m}+U_{1,-l}^n+\tilde U_{1,1-j}^n)+ \Phi_{\rho^n_{j,l}}(A_{l}^{n,m}+\tilde U_{1,-l}^n+\tilde U_{1,1-j}^n)\Big)\Big] \\ 
\label{eq:45} = {}& \tilde \E\Big[ \int^{A^n_l}_{A^{n,m}_l} \int^{U_{1,1-j}^n}_{\tilde U_{1,1-j}^n} \int^{U^n_{1,-l}}_{\tilde U^n_{1,-l}}  \Phi'''_{\rho^n_{j,l}}(u_1+u_2+u_3)\,du_1\,du_2\,du_3\Big], 
\end{align}
where $\int^x_y$ denotes $-\int^{y}_x$ if $x<y$.
Hence, by substitution, there is a random variable $W^{n,m}_{j,l}$ such that 
\begin{align}\label{eq:82363782}
 {}& |\vartheta_{1,j,l}^{n,m}|  \leq   \\ {}&  \tilde \E\Big[ \int^{|A^n_l-A^{n,m}_l |}_0 \int_0^{|\tilde U_{1,1-j}^n-U_{1,1-j}^n|} \int^{|U^n_{1,-l}- \tilde U^n_{1,-l}|}_0  |\Phi'''_{\rho^n_{j,l}}(W^{n,m}_{j,l}+u_1+u_2+u_3)|\,du_1\,du_2\,du_3\Big].
 \end{align}
For $l> m$ we have that 
\begin{align}
\vartheta_{1,j,l}^{n,m} ={}&  \tilde \E\Big[  \Phi_{\rho^n_{j,l}}( A_{l}^n +U_{1,-l}^n+U_{1,1-j}^n )- 
\Phi_{ \rho^n_{j,l}}(A_{l}^n+\tilde U_{1,-l}^n+U_{1,1-j}^n)\\ {}&
\phantom{   \tilde \E\Big[ }   -\Phi_{\rho^n_{j,l}}(A_{l}^n+U_{1,-l}^n+\tilde U_{1,1-j}^n)+ \Phi_{\rho^n_{j,l}}(A_{l}^n+\tilde U_{1,-l}^n+\tilde U_{1,1-j}^n) \Big]\\ 
\label{eq:45} = {}& \tilde \E\Big[  \int^{U^n_{1,1-j}}_{\tilde U^n_{1,1-j}} \int^{U^n_{1,-l}}_{\tilde U^n_{1,-l}}  \Phi''_{\rho^n_{j,l}}(A^n_l+u_2+u_3)\,du_1\,du_2\Big] .
\end{align}
Let $l=j,\dots,m-1$. By \eqref{eq:82363782} and \eqref{eq:837636} we have that 
\begin{align}
\E[ |\vartheta_{1,j,l}^{n,m}|^\gamma] \leq {}&  K \Big(  \E[ |A^n_l-A_l^{n,m}|^{p\gamma} \1_{\{ |A^n_l-A_l^{n,m}|\geq 1\}}] + \E[ |A^n_l-A_l^{n,m}|^{\gamma} \1_{\{|A^n_l-A_l^{n,m}|\leq 1\}}]  \Big) 
\\ {}& \phantom{ K } \times  \E[\tilde\E[ (|\tilde U^n_{1,1-j}-U^n_{1,1-j}|\wedge 1)^\gamma]] \E[\tilde \E[   (|\tilde U^n_{1,-l}-U^n_{1,-l}|\wedge 1)^\gamma] ]
\\ \leq 
{}&   K \| \phi^n_1 \|_{L^\beta((-\infty,1 -m])}^\beta \| \phi^n_1 \|_{L^\beta([1-j,2-j])}^\beta \| \phi^n_1 \|_{L^\beta([-l,1-l])}^\beta  
\\ \leq {}&   K m^{(\alpha-k)\beta+1} j^{(\alpha-k)\beta} l^{(\alpha-k)\beta}. 
\end{align}
We use Lemma~\ref{sim-est-W}(i) and (ii),   $p\gamma <\beta<\gamma$ and  $|x- y|\wedge 1\leq |x|\wedge 1 + | y | \wedge 1$. 
For $l\geq m$ we have by \eqref{eq:83637392} that  
\begin{align}
 \E[ |  \vartheta_{1,j,l}^{n,m}||^\gamma]  \leq {}& K  \E[  \tilde \E[ ( | U^n_{1,1-j}- \tilde U^n_{1,1-j}| \wedge 1)^\gamma]] 
 \\ {}& \times \Big( \E[ \tilde \E[ | U^n_{1,1-j}- \tilde U^n_{1,1-j}|^{p\gamma}\1_{\{| U^n_{1,1-j}- \tilde U^n_{1,1-j}|\geq 1\}}  ] ]
 \\ {}& \phantom{\times \Big(} + 
 \E[ \tilde \E[ | U^n_{1,1-j}- \tilde U^n_{1,1-j}|^{\gamma}\1_{\{| U^n_{1,1-j}- \tilde U^n_{1,1-j}|\leq  1\}}  ]] \Big)
  \\  {}& \leq K  \| \phi^n_1 \|_{L^\beta([1-j,2-j])}^\beta  \| \phi^n_1 \|_{L^\beta([1-j,2-j])}^\beta \leq K j^{(\alpha-k)\beta}l^{(\alpha-k)\beta}  
\end{align}
again using Lemma~\ref{sim-est-W}(i) and (ii),   $p\gamma <\beta<\gamma$ and  $|x- y|\wedge 1\leq |x|\wedge 1 + | y | \wedge 1$. This completes the proof of \eqref{est-ss4w5}. 
\end{proof}

We are now ready to prove Lemma~\ref{eq:8937739}.

\begin{proof}[Proof of Lemma~\ref{eq:8937739}] 
 We will use  Lemma~\ref{lem-est-theta}  for $\gamma=2$ which satisfies $\beta<\gamma<\beta/p$. Suppose that $j=1,\dots,m$. By orthogonality of $\{\vartheta^{n,m}_{r,j,l }: l=1,2,\dots\}$ in $L^2$   we have that 
 \begin{align}
 \E[ |  \zeta_{r,j}^{n,m}|^2] = {}& \sum_{l=j}^\infty \E[ |  \vartheta_{r,j,l}^{n,m}|^2] =K \left(  \sum_{l=j}^{m-1}
  m^{(\alpha-k)\beta+1}  l^{(\alpha-k)\beta} j^{(\alpha-k)\beta}  + \sum_{l=m}^\infty  l^{(\alpha-k)\beta} 
  j^{(\alpha-k)\beta} \right) 
  \\ \leq {}& K \Big(  (m+1)^{(\alpha-k)\beta+1}   j^{2(\alpha-k)\beta+1} + j^{(\alpha-k)\beta}(m+1)^{(\alpha-k)\beta+1} \Big) 
   \\ \leq {}& K (m+1)^{(\alpha-k)\beta+1} j^{(\alpha-k)\beta}
 \end{align}
 since $2(\alpha-k)\beta+1<(\alpha-k)\beta$. Similarly, for $j>m$ we have that 
 \begin{equation}
 \E[ |  \zeta_{r,j}^{n,m}|^2] =  \sum_{l=j}^\infty \E[ |  \zeta_{r,j}^{n,m}|^2] \leq K j^{(\alpha-k)\beta} \sum_{l=j}^\infty l^{(\alpha-k)\beta}\leq K j^{2(\alpha-k)\beta+1}, 
 \end{equation}
 which completes the proof. 
\end{proof}

%By   Lemma~\ref{eq:8937739}  it follows that   $\E[ |  \zeta_{r,j}^{n,m}|^2] \leq K 
% j^{2(\alpha-k)\beta+1} $ for all $j=1,2,\dots $. 

%%%%%%%%%%%%%%%%%%%%
\subsection{Proof of Proposition~\ref{prop-key-est}}\label{proof-prop-232}
%%%%%%%%%%%%%%%%%%%%%%
We will start with the proof of  \eqref{eq:112132-1}.
By rearranging the terms  using  the  substitution $s=r-j$, we have 
\[
\sum_{r=k}^n R^{n,m}_r =  \sum_{s=-\infty}^{n-1} M_{s}^{n,m} \qquad \text{with} \qquad M_{s}^{n,m}:= \sum_{r= 1\vee (s+1)}^n \zeta_{r,r-s}^{n,m}.   
\]
Recalling the definition of $\zeta^{n,m}_{r,j}$ in \eqref{eq-S-S'-2}, we note that $\E[ \zeta^{n,m}_{r,r-s} |\g_s]=0$ for all $s$ and $r$,  showing that 
 that $\{M^{n,m}_s: s\in (-\infty,n)\cap \Z \}$ are  martingale differences.  By orthogonality we 
 have that  
\begin{align}
\E\Big[\Big(\sum_{r=k}^n R^{n,m}_r\Big)^{2}\Big]{}& =  \sum_{s=-\infty}^{n-1} \E[|M_{s}^{n,m}|^{2}] \\
&\leq  \sum_{s=-\infty}^{n-1} \left( \sum_{r= 1\vee (s+1)}^n \E[|\zeta_{r,r-s}^{n,m}|^{2}]^{1/2}  \right)^{2} =:  A_{n,m}.
\label{est-A-n-m}
%\\
%&\leq K n^{\gamma/(\alpha -k)\beta } \sum_{s<n} \left( \sum_{r= 1\vee (s+1)}^n (r-s)^{1/\gamma +2(\alpha -k)}  \right)^{\gamma} =: A_n.
\end{align}
We split $A_{n,m} = \sum_{s=1}^n+\sum_{s=-n}^0+ \sum_{s=-\infty}^{-n}  = A'_{n,m} + A_{n,m}''+A_{n,m}'''$. 
By the substitution $\tilde s=n-s$ and $\tilde r= r-s$ we obtain
\begin{align}
A_{n,m}' = \sum_{s=1}^n \Big(\sum_{r=1}^s \E[ | \zeta^{n,m}_{r+n-s, r}|^2]^{1/2}\Big)^2.
\end{align}
For $s=1,\dots,n$ we have (cf.\ Lemma~\ref{lem-est-theta})
\begin{align}
 \sum_{r=k}^s \E[ | \zeta^{n,m}_{r+n-s, r}|^2]^{1/2} \leq {}&  K\Big( m^{((\alpha-k)\beta+1)/2} \sum_{r=k}^m r^{(\alpha-k)\beta/2} + \sum_{r=m}^s r^{2(\alpha-k)\beta+1}\Big)
 \\ \leq {}& 
  K\Big( m^{((\alpha-k)\beta+1)/2}(\log(m)+ m^{(\alpha-k)\beta/2+1})  +  m^{2((\alpha-k)\beta+1)}\Big)
  \\ \leq {}& K \Big(m^{((\alpha-k)\beta+1)/2}\log(m)+m^{(\alpha-k)\beta+3/2}\Big) \label{eq:239792739}
\intertext{where we have used the assumption $(\alpha-k)\beta<-1$ in the second inequality. Eq.~\eqref{eq:239792739}  shows  that  }
\label{eq:A-121}
A_{n,m}' \leq {}& K n \Big( m^{(\alpha-k)\beta+1}(\log(m))^2 + m^{2(\alpha-k)\beta +3}\Big).
\end{align}
The substitution $\tilde s=-s$ and $\tilde r=r-s $ together with Lemma~\ref{lem-est-theta}  yields that 
\begin{align}
A_{n,m}'' = {}&  \sum_{s=0}^n \Big( \sum_{r=s+1}^{n+s} \E[ | \zeta^{n,m}_{r+s, r}|^2]^{1/2}\Big)^2 
\leq  K \sum_{s=0}^n \Big( \sum_{r=s+1}^{n+s} r^{(\alpha-k)\beta+1/2}\Big)^2.\label{eq:232355}
\end{align}
For $\alpha<k-2/\beta$ the inner sum on the right-hand side of  \eqref{eq:232355} is summable. 
Thus, we deduce
\begin{align}\label{est-3224234124}
A_{n,m}''\leq K \sum_{s=0}^n  s^{2(\alpha-k)\beta+3}\leq K \Big( n^{2(\alpha-k)\beta+4}+\log(n)\Big).
\end{align}
On the other hand, for  $\alpha\geq k-2/\beta$ we have by Jensen's inequality that 
\begin{align}
A_{n,m}'' \leq {}& K n\sum_{s=0}^n \Big( \sum_{r=s+1}^{n+s} r^{2(\alpha-k)\beta+1}\Big) 
\leq K n \sum_{s=0}^n s^{2(\alpha-k)\beta+2}\leq  K n^{2(\alpha-k)\beta+4} ,\label{eq:A-122}
\end{align}
where we have used  the assumption $(\alpha-k)\beta<-1$ in the second inequality and the fact that  $\alpha>k-\frac{3}{2\beta}$ in the third inequality.   
Again by the substitution $\tilde s=-s$ and $\tilde r=r-s $ and Lemma~\ref{lem-est-theta} we have
\begin{align}
A_{n,m}''' =  {}& \sum_{s=n}^\infty \Big( \sum_{r=s+1}^{n+s} \E[ | \zeta^{n,m}_{r+s, r}|^2]^{1/2}\Big)^2  \leq K
\sum_{s=n}^\infty \Big( \sum_{r=s+1}^{n+s} r^{(\alpha-k)\beta +1/2}\Big)^2   
\\ \leq {}& K \sum_{s=n}^\infty \Big( n  s^{(\alpha-k)\beta +1/2}\Big)^2   \leq K n^{2(\alpha-k)\beta+4},
\label{eq:A-123} 
\end{align}
where we have used the assumption $(\alpha-k)\beta<-1$ in the last inequality. 
Combining the  estimates  \eqref{est-A-n-m}--\eqref{eq:A-123} yields \eqref{eq:112132-1}. 
%\begin{equation}
%\limsup_{n\to \infty} \Big( n^{-1} \E\Big[\Big|\sum_{r=1}^n R^{n,m}_r\Big|^{2}\Big]\Big) \leq K m^{(\alpha-k)\beta+1} ,
%\end{equation} 
% which implies \eqref{eq:123401282781}. 
\noeqref{eq:A-122}\noeqref{est-3224234124}\noeqref{eq:A-121}

In the proof of  \eqref{eq:112132-3} and \eqref{eq:112132-2} we will use the following decomposition 
\begin{equation}\label{decompo-W_n}
 \sum_{r=k}^n  Q^{n,m}_r  = \sum_{s=-\infty}^{n-1}  \sum_{j=(k-s)\vee 1}^{n-s} \E[ V_{s+j}^{n,m}  |\g^{1}_{s}] 
\end{equation}
which follows by  the substitution $s=r-j$. 
To prove \eqref{eq:112132-3}  we assume that $\alpha<k-2/\beta$ and let $\epsilon>0$.  
By \eqref{est-H-rho-1-2} we have for all $p\leq \gamma <\beta/2$ that 
\begin{align}
{}&  \E\Big[\Big| \Phi_{\rho^n_j} (U_{s+j,s}^n)- \Phi_{\rho^{n,m}_j}(U^n_{s+j,s})\Big|^2\Big] \leq \Big| | \rho_j^n|^\beta- |\rho^{n,m}_j|^\beta\Big|^2 \E[ |U_{s+j,s}^n|^{2\gamma} ] 
 \\ {}&\qquad \leq  K \Big| | \rho_j^n|^\beta- |\rho^{n,m}_j|^\beta\Big|^2 j^{(\alpha-k) 2\gamma} \leq K \Big| | \rho_j^n|^\beta- |\rho^{n,m}_j|^\beta\Big|^2 j^{(\alpha-k) \beta+2\epsilon}, \label{eqst-sflsf-2}
\end{align}
where the last inequality holds for $\gamma$ close enough to $\beta/2$. 
We have that 
\begin{align}
 \Big| | \rho_j^n|^\beta- |\rho^{n,m}_j|^\beta\Big| {}& =  \Big| \int_{(-\infty,s+j]\setminus [s,s+1]} | \phi^n_{s+j}(u)|^\beta\,du- \int_{(-s+j-m,s+j]\setminus [s,s+1]} | \phi^n_{s+j}(u)|^\beta\,du\Big| 
 \\ {}& \leq   \int_{-\infty}^{-m}  | \phi^n_{0}(u)|^\beta\,du\leq m^{(\alpha-k)\beta+1} . \label{eqst-sflsf}
\end{align}
%Similar estimate holds for the second term on the right-hand side of \eqref{eq:725673489}, and hence  
%\begin{align}
%\big\| \E[ V_{s+j}^{n,m}  |\g^{1}_s] - \eta^{m}_{j,s}\big\|_{L^2}\leq K m^{(\alpha-k)\beta+1}  j^{(\alpha-k)\beta/2+\epsilon}. 
%\end{align}
By recalling the identity  \eqref{eq:1212124432} we  have
\begin{align}
{}&   \Big\| \E[ V_{s+j}^{n,m}  |\g^{1}_s] \Big\|_{L^2}   \leq 2  
\Big\| \Phi_{\rho^n_j} (U_{s+j,s}^n)- \Phi_{\rho^{n,m}_j}(U^n_{s+j,s})\Big\|_{L^2} \\ {}& \leq  K \begin{cases}
m^{(\alpha-k)\beta+1}  j^{(\alpha-k)\beta/2+\epsilon} & j=1,\dots,m \\ 
  j^{(\alpha-k)\beta/2+\epsilon} & j>m
  \end{cases}
\label{eq:725673489}
\end{align}
where the last inequality follows from \eqref{eqst-sflsf-2} and \eqref{eqst-sflsf}. 
By orthogonality in $L^2$ of the inner sums on the right-hand side of \eqref{decompo-W_n} we have that 
\begin{align}
{}& \E\Big[ \Big(\sum_{r=k}^n  Q^{n,m}_r\Big)^2\Big] =   \sum_{s=-\infty}^{n-1} \E\Big[\Big( \sum_{j=(k-s)\vee 1}^{n-s} \E[ V_{s+j}^{n,m}  |\g^{1}_{s}]\Big)^2\Big]
\\ {}&\quad \leq   \sum_{s=-\infty}^{n-1} \Big( \sum_{j=(k-s)\vee 1}^{n-s} \Big\| \E[ V_{s+j}^{n,m}  |\g^{1}_{s}]\Big\|_{L^2} \Big)^2
\\ {}&\quad
=   K \Big[ \sum_{s=-\infty}^{k-1} \Big( \sum_{j=k-s}^{n-s} \Big\| \E[ V_{s+j}^{n,m}  |\g^{1}_{s}]\Big\|_{L^2} \Big)^2\Big)^2+\sum_{s=k}^{n-1} \Big( \sum_{j=  1}^{n-s} \Big\| \E[ V_{s+j}^{n,m}  |\g^{1}_{s}]\Big\|_{L^2} \Big)^2\Big)^2\Big]
\\ {}&\quad
=:   K[A_{n,m}'+A_{n,m}''].\label{eq:232323543}
\end{align}
  By \eqref{eq:725673489} we obtain the following estimate on  $A_{n,m}'$: 
\begin{align}
{}& A_{n,m}'\leq  K \sum_{s=-\infty}^{k-1} \Big( \sum_{j=k-s}^{n-s} j^{(\alpha-k)\beta/2+\epsilon}\Big)^2
\\ {} &\qquad = K\Big(  \sum_{s=-k+1}^n \Big( \sum_{j=k+s}^{n+s} j^{(\alpha-k)\beta/2+\epsilon}\Big)^2+  \sum_{s=n+1}^\infty \Big( \sum_{j=k+s}^{n+s} j^{(\alpha-k)\beta/2+\epsilon}\Big)^2 \Big) \\ {}& \qquad =:  K( B_n+C_n) .\label{eq:237232}
\end{align}
Since  $(\alpha-k)\beta<-2$ we obtain the estimate
\begin{align}\label{eq:2323-1}
B_n\leq {}& K \sum_{s=-k+1}^n s^{(\alpha-k)\beta+2+2\epsilon}\leq 
K  (n^{(\alpha-k)\beta+3+2\epsilon}+1),
\shortintertext{
and by using  $(\alpha-k)\beta<-1$ we get  }\label{eq:2323-2}
C_n\leq {}& K  \sum_{s=n+1}^\infty n^{2} s^{(\alpha-k)\beta+2\epsilon}
\leq K n^{(\alpha-k)\beta +3+2\epsilon}.
\end{align}
 By the substitution $\tilde s=n-s$ and \eqref{eq:725673489} we have 
\begin{align}
A_{n,m}'' {}& =     \sum_{s=1}^{n-k-1}  \Big(\sum_{j=1}^{s} \Big\| \E[ V_{n+s+j}^{n,m}  |\g^{1}_{n+s}] \Big\|_{L^2}\Big)^2\\ 
{} &\leq  \sum_{s=1}^{n-1}  \Big(m^{(\alpha-k)\beta+1}
 \sum_{j=1}^{m} j^{(\alpha-k)\beta/2+\epsilon} + \sum_{j=m+1}^s j^{(\alpha-k)\beta/2+\epsilon}  \Big)^2
 \\ {}& \leq 
 n \Big( m^{2((\alpha-k)\beta+1)}+  m^{(\alpha-k)\beta+2+2\epsilon} \Big)\leq n m^{(\alpha-k)\beta+2+2\epsilon}  \label{eq:7826628}
\end{align}
where the  last inequality follows by the  assumption $(\alpha-k)\beta<-2$. The above estimates 
\eqref{eq:232323543}--\eqref{eq:7826628} yield \eqref{eq:112132-3}. 

\noeqref{eq:237232}  \noeqref{eq:2323-1}  \noeqref{eq:2323-2}

To prove \eqref{eq:112132-2} we suppose that $\alpha>k-2/\beta$.  We will again use the decomposition \eqref{decom-Q-232}, which by the decomposition $\sum_{s=-\infty}^{n-1} = \sum_{s=-\infty}^{k-1}+ \sum_{s=k}^{n-1}$ gives   
\begin{equation}\label{decom-Q-232}
\sum_{r=k}^n \Big(Q^{n,0}_r-Z_r\Big) =  H_{n}^{(1)} - H_{n}^{(2)} + H_{n}^{(3)},
\end{equation}
where
\begin{align*}
{}& H_{n}^{(1)} =  \sum_{s=-\infty}^{k-1} \sum_{j=k-s}^{n-s} \E[ V_{s+j}^{n,0}  |\g^{1}_{s}] , \qquad \quad  H_{n}^{(2)} =   \sum_{s=k}^{n}  \sum_{j=n-s+1}^\infty \Big\{\Phi_{\rho_j^\infty} ( U_{j,r}^\infty)-\E[\Phi_{\rho_j^\infty} (U_{j,r}^\infty)]\Big\}, \\
{}& H_{n}^{(3)} =    \sum_{s=k}^{n-1}  \sum_{j=1}^{n-s}\Big( \E[ V_{s+j}^{n,0}  |\g^{1}_{s}] - \Big\{\Phi_{\rho_j^\infty} ( U_{j,r}^\infty)-\E[\Phi_{\rho_j^\infty} (U_{j,r}^\infty)]\Big\}\Big). 
\end{align*}
%
%
%
%\noindent 
%Proof of \eqref{eq:827284-2}:
%We will use the decomposition 
%\begin{equation}\label{decp,-sfsf}
%\sum_{r=1}^n Q^{n}_r =  H_{n}^{(1)} - H_{n}^{(2)} + H_{n}^{(3)},
%\end{equation}
%where
%\begin{align*}
%{}& H_{n}^{(1)} =  \sum_{s=-\infty}^{0} \sum_{j=1-s}^{n-s} \E[ V_{s+j}^{n}  |\g^{1}_s] , \qquad \quad  H_{n}^{(2)} =  \sum_{s=1}^{n-1}  \sum_{j=n-s+1}^\infty \eta_{j+s,s}, \\
%{}& H_{n}^{(3)} =    \sum_{s=1}^{n-1}  \sum_{j=1}^{n-s}\Big( \E[ V_{s+j}^{n}  |\g^{1}_s] - \eta_{j+s,s}\Big). 
%\end{align*}
 In the following we will show that for all $\epsilon>0$  there exists a finite constant $K$ such all  $i=1,2,3$ and $n\geq 1$ we have 
\begin{equation}\label{eq:726251}
\E[ | H^{(i)}_n|]\leq K  \Big( n^{(\alpha-k)\beta+2+\epsilon} + n^{1-\beta+\epsilon}\Big),
%,\\ \label{eq:726252}
%{}& \leq K  n^{(\alpha-k)\beta+2+\epsilon},\\ \label{eq:726253}
%\E[ | H^{(3)}_n|]{}& \leq K  n^{(\alpha-k)\beta+2+\epsilon},
\end{equation}
which by \eqref{decom-Q-232} yields  \eqref{eq:112132-2}.
To be used in the proof of \eqref{eq:726251} we recall that  according to \eqref{eq:1212124432} we have  
\begin{equation}\label{eq:23097673422}
 \E[ V_{s+j}^{n,0}  |\g^{1}_s] = \Phi_{\rho^n_j}(U^n_{s+j,s})- \E[ \Phi_{\rho^n_j}(U^n_{s+j,s})].
\end{equation}
For all $\gamma\in (p,\beta)$ such that $-2<(\alpha-k)\gamma<-1$ we have by \eqref{eq:23097673422} that 
\begin{align}
{}& \E[ | H^{(1)}_n|]  \leq 2 \sum_{s=-\infty}^{k-1} \sum_{j=k-s}^{n-s} \E[ | \Phi_{\rho^n_j}(U^n_{s+j,s})|] 
\leq K \sum_{s=-\infty}^{k-1}  \sum_{j=k-s}^{n-s} \E[ | U^n_{s+j,s}|^\gamma] \\
{} & \leq K \sum_{s=-k+1}^\infty  \sum_{j=k+s}^{n+s} j^{(\alpha-k)\gamma}=  K\Big(  \sum_{s=-k+1}^n   \sum_{j=k+s}^{n+s} j^{(\alpha-k)\gamma} + \sum_{s=n=1}^\infty   \sum_{j=k+s}^{n+s} j^{(\alpha-k)\gamma}\Big) \\
{}& \leq K \Big(  \sum_{s=-k+1}^n  s^{(\alpha-k)\gamma+1} + \sum_{s=n+1}^\infty n s^{(\alpha-k)\gamma} \Big) 
\leq K n^{(\alpha-k)\gamma +2}, \label{eq:12301802812}
%{}&\quad   \leq K \sum_{j=1}^\infty   \sum_{s=(j-n)\vee 1}^{j} j^{(\alpha-k)\gamma}   = 
%K \Big( \sum_{j=1}^n    \sum_{s=1}^{j} j^{(\alpha-k)\gamma} + \sum_{j=n+1}^\infty   \sum_{s=j-n }^{j} j^{(\alpha-k)\gamma}\Big)
%\\ {}&\quad = 
%K \Big( \sum_{j=1}^n    j^{(\alpha-k)\gamma+1} +n  \sum_{j=n+1}^\infty   j^{(\alpha-k)\gamma}\Big)
% \\ {}& \quad\leq  
%K \Big( n^{(\alpha-k)\gamma+2} + n^{(\alpha-k)\gamma+1}\Big)\leq  K  n^{(\alpha-k)\gamma+2} 
\end{align}
where the second inequality follows by \eqref{est-H-beta}, the third inequality follows by  \eqref{eq:72138}, 
the fourth inequality follows by $(\alpha-k)\gamma<-1$,  and the last inequality follows by $(\alpha-k)\gamma+1>-1$. 
Similarly, we have for  all $\gamma\in (p,\beta)$ with  $-2<(\alpha-k)\gamma<-1$ that
\begin{align}
 \E[ | H_{n}^{(2)}|] {}& \leq   2\sum_{s=k}^{n}  \sum_{j=n-s+1}^\infty \E[ |\Phi_{\rho_j^\infty} ( U_{j+s,s}^\infty) |] \leq   K \sum_{s=0}^{n-k}  \sum_{j=s+1}^\infty \E[ | U_{j+n-s,n-s}^\infty) |^\gamma] \\ 
{}& \leq   K \sum_{s=0}^{n-k}  \sum_{j=s+1}^\infty j^{(\alpha-k)\gamma}\leq K  \sum_{s=1}^{n-1}  s^{(\alpha-k)\gamma+1}\leq K n^{(\alpha-k)\gamma +2}.\label{eq:H-2fjlsjf}
\end{align}
We will need more involved estimates on $H_n^{(3)}$. To this end we start with the following simple inequality 
\begin{align}
\E[ | H^{(3)}_n| ] {}& \leq 2 \sum_{s=k}^{n-1} \sum_{j=1}^{n-s} \E[ | \Phi_{\rho^n_j}(U^n_{s+j,s}) -  \Phi_{\rho_j^\infty} ( U_{s+j,s}^\infty)|] \\
{}& \leq 2 n  \sum_{j=1}^{n} \E[ | \Phi_{\rho^n_j}(U^n_{j,0}) -  \Phi_{\rho_j^\infty} (U_{j,0}^\infty)|] .\label{est-H^3-slf}
\end{align}
By adding and subtracting $\Phi_{\rho_j^n}(U^\infty_{j,0})$ we get the decomposition
\begin{align}\label{decom-=sflj}
 \Phi_{\rho^n_j}(U^n_{j,0}) -  \Phi_{\rho_j^\infty} ( U_{j,0}^\infty) = {}& C^n_j+ D^n_j
 \shortintertext{where} 
 C^n_j =\Phi_{\rho_j^n}(U^n_{j,0}) -   \Phi_{\rho_j^n} ( U_{j,0}^\infty) \qquad{}& \text{and}\qquad  D^n_j =  \Phi_{\rho_j^n} ( U^\infty_{j,0}) - \Phi_{\rho_j^\infty}(U^\infty_{j,0}).
\end{align}
In the following we will show that for all $\epsilon>0$ we have that 
 \begin{equation}\label{eq:7239723}
(a)\!:  \sum_{j=1}^n \E[ |C^n_j|]\leq K \Big(
   n^{  (\alpha-k)\beta+1+\epsilon}+n^{-\beta+\epsilon}\Big) \qquad \qquad (b)\!: 
%\shortintertext{ and }
%\label{est-Dslfjls}
%{}& 
\sum_{j=1}^n \E[ |D^n_j|] \leq K n^{(\alpha-k)\beta+1}.
\end{equation}
To prove \eqref{eq:7239723}(a) we note that   $g_n(s)=n^\alpha g(s/n)$ and $g(s)=s^\alpha f(s)$ we have for $s\geq 0$ that  
\begin{equation}\label{eq:83733}
 \eta_n(s):= g_n(s)-s^{\alpha}= n^{\alpha} (s/n)^{\alpha} \{ f(s/n) -f(0)\} =n^{\alpha} \psi_1(s/n) \psi_2(s/n)
\end{equation}
where  $\psi_1(s)=s^\alpha$ and $\psi_2(s)= f(s)- f(0)$ for $s\geq 0$. 
%We set $g_n(s)=n^\alpha g(s/n)$ and $\psi_n(s) = g_n(s) - c_0 s^\alpha = n^\alpha \psi_1(s/n) \psi_2(s/n)$ for all $s\geq 0$.  
For all $s>k$ there exists, as a consequence of the mean-value theorem, a  $\xi^n_s\in [s-k,s]$ such that 
\begin{equation}\label{eq:82627}
 (D^k \eta_n)(s) = \eta_n^{(k)}(\xi^n_s)= n^{\alpha -k}\sum_{l=0}^k \binom{k}{l} \psi_1^{(l)}(\xi_s^n/n)
\psi_2^{(k-l)}(\xi_s^n/n). 
\end{equation}
Eq.~\eqref{eq:82627}   implies that 
\begin{equation}\label{eq:8366483}
  |(D^k \eta_n)(s)|\leq K\Big[\Big(\sum_{l=0}^{k-1} n^{l-k} | \xi^n_s|^{\alpha-l}\Big)+ |\xi^n_s|^{\alpha-k+1} n^{-1}\Big],
\end{equation} 
where we have used that $\psi_1^{(l)}(t)=\alpha (\alpha-1)\cdots (\alpha - l+1) t^{\alpha- l}$ for $t>0$, that $\psi^{(l)}_2$ is bounded  on $(0,\infty)$ for $l=1,\dots,k$, and that $| \psi_2(t)|\leq K t$ for all $t>0$. 
Since $\phi^n_j(s)-\phi^\infty_j(s) = D^k \eta_n(j-s)$ we obtain by \eqref{eq:8366483} the estimate 
\begin{align}\label{est-a-j}
{}& \| \phi^n_j-\phi^\infty_j \|_{L^\beta([0,1])} \leq K \sum_{l=0}^{k}a_{l,j,n}\qquad 
\text{where  }\\
a_{l,j,n}=  n^{l-k}  j^{\alpha-l}{}&  \quad \text{for } l=0,\dots,k-1,\qquad \text{and}\qquad  a_{k,j,n} = 
n^{-1} j^{\alpha-k+1} .
\end{align}
We note that $a_{k-1,j,n}= a_{k,j,n}$, and for all $l=0,\dots,k-1$ and $j=1,\dots,n$ we have $a_{l,j,n}= (n/j)^{l} n^{-k} j^\alpha\leq (n/j)^{k-1} n^{-k} j^\alpha = n^{-1} j^{\alpha-k+1}$, which by \eqref{est-a-j} shows that 
\begin{equation}\label{eqasfljsf}
 \| \phi^n_j-\phi^\infty_j \|_{L^\beta([0,1])} \leq K n^{-1} j^{\alpha-k+1} \qquad \quad j=1,\dots,n. 
\end{equation}
First we suppose  that $\beta>1$. For all $\tilde \epsilon>0$ we have according to \eqref{est-234l1j32lh3} of Lemma~\ref{est-H-inf},   \eqref{eq:72138} and \eqref{eqasfljsf} that 
\begin{align}\label{eqlrjslfdjlsf}
 \sum_{j=1}^n \E[ |C^n_j|]\leq {}& K \sum_{j=1}^n \Big(
  j^{(\alpha-k)(\beta-1-\tilde \epsilon)}(n^{-1}j^{\alpha-k+1}) +(n^{-1} j^{\alpha-k+1})^{\beta} \Big)\\
 =  {}&  K \Big( n^{-1} \sum_{j=1}^n j^{(\alpha-k)\beta+1-\tilde \epsilon (\alpha-k)}  +n^{-\beta}\sum_{j=1}^n  j^{(\alpha-k+1)\beta} \Big)
 \\ \leq   {}&  K \Big( n^{-1+(\alpha-k)\beta+2-\tilde \epsilon (\alpha-k)}  +n^{-\beta+(\alpha-k)\beta+\beta+1}\Big), \label{eq:232432}
\end{align}
where we have used that $(\alpha-k)\beta>-2$ and $\beta>1$ in the second inequality. Eq.~\eqref{eq:232432}  shows \eqref{eq:7239723}(a) by choosing $\tilde\epsilon$ close enough to zero. 
%we have that 
%Thus, we have 
%\begin{align}\label{eq:23087203}
% \sum_{j=1}^n \E[ |C^n_j|]\leq {}& K \sum_{l=0}^k \Big\{\sum_{j=1}^n \big( a_{l,j,n} +a_{l,j,n}^{\frac{1}{k-\alpha+\epsilon/\beta}} \big)\Big\}.
%\end{align}
% By Lemma~\ref{est-H-inf} used on $r=1$,   \eqref{eq:72138} and  \eqref{est-a-j}  %for  $q=q_\epsilon\in (1,\beta)$ chosen  close enough  to $\beta$  
% we have
%\begin{align}\label{eq:836a}
% \sum_{j=1}^n \E[ |C^n_j|]\leq {}& K \sum_{l=0}^k \Big\{\sum_{j=1}^n \big(  j^{(\alpha-k)(\beta-1-\epsilon)} a_{l,j,n} +a_{l,j,n}^\beta \big)\Big\}.
% %\leq K n^{(\alpha-k)(\beta-\epsilon)+1}\to 0\quad 
%\end{align}
%For all $l=0,\dots,k-1$ we have by the definition of $a_{l,j,n}$ and by choosing $\tilde \epsilon$ close enough to zero we have that 
%\begin{equation}\label{eq:82627893}
%\sum_{j=1}^n \big(  j^{(\alpha-k)(\beta-1-\tilde \epsilon)} a_{l,j,n} +a_{l,j,n}^{\beta} \big)
%\leq K  n^{(\alpha-k)\beta+1+\epsilon}.
%\end{equation}
% Since  $a_{k,j,l}=a_{k-1,j,l}$, \eqref{eq:82627893} also holds for $l=k$,  and hence by \eqref{eqlrjslfdjlsf} we obtain the estimate \eqref{eq:7239723}(a).  
 On the other hand, suppose that $\beta\leq 1$.   For all $\tilde \epsilon>0$ we have according to \eqref{eq:97236554340} of Lemma~\ref{est-H-inf} and \eqref{eqasfljsf} that 
\begin{align}\label{eqlrjslfdjlsf-2}
 \sum_{j=1}^n \E[ |C^n_j|]\leq {}& K \sum_{j=1}^n  (n^{-1}j^{\alpha-k+1})^{\beta-\tilde \epsilon} = K n^{-\beta+\tilde \epsilon} \sum_{j=1}^n j^{(\alpha-k+1)(\beta-\tilde \epsilon)}.
\end{align} 
For $(\alpha-k+1)\beta>-1$ and $\tilde \epsilon$ chosen small enough,  \eqref{eqlrjslfdjlsf-2} implies that  
\begin{equation}
 \sum_{j=1}^n \E[ |C^n_j|]\leq K n^{(\alpha-k)\beta+1+\epsilon},
\end{equation}
 which shows \eqref{eq:7239723}(a). When $(\alpha-k+1)\beta\leq -1$ then the sum on the right-hand side of \eqref{eqlrjslfdjlsf-2} converges and we obtain the estimate
\begin{equation}
  \sum_{j=1}^n \E[ |C^n_j|]\leq K n^{-\beta+\tilde \epsilon}, 
\end{equation}
which completes the proof of \eqref{eq:7239723}(a).

To show \eqref{eq:7239723}(b) we use  
%substitution, Lemma~\ref{helplem1}(i) and \eqref{eq:72138} to obtain
%\begin{align}
%{}&  \Big | |\rho_{j}^n|^\beta-   |\rho_0|^\beta\Big|  \leq  \Big| \| \phi^n_1 \|_{L^\beta(\R)}- \| \phi_1 \|_{L^\beta(\R)} \Big| 
%+ \| \phi^n_1 \|_{L^\beta([-j,-j+1])}
% \\  {}& \leq K\Big(  n^{\alpha \beta+1} \Big| \| g_{0,n}\|_{L^\beta(\R)}^\beta -  \| h_{0,n}\|_{L^\beta(\R)}^\beta  \Big| + j^{\alpha-k}\Big)\leq K\Big( n^{(\alpha-k)\beta+1} +j^{\alpha-k}\Big).\label{eq:63734}
%\end{align} 
%For any $r\in (p,\beta)$ such that $(\alpha-k)r <-1$ we have  by \eqref{est-H-rho-1-2}, \eqref{eq:72138} and \eqref{eq:63734} that 
%\begin{align}
%\E[  | D^n_j | ] = {}&  \E[  | \Phi_{\rho^n_j}( U_{j,0}^n)-\Phi_{ \rho_0}(U_{j,0}^n)| ] 
%\leq K  \big| |\rho^n_j |^\beta-  | \rho |^\beta \big|  \E[ | U_{j,0}^n |^r   ]  
%\\ \leq {}& K  \big| |\rho^n_j |^\beta-  | \rho |^\beta \big|  \| \phi^n_j \|_{L^\beta([0,1])}^r \leq K \Big(n^{(\alpha-k)\beta+1} j^{(\alpha-k)r }+j^{(\alpha-k)(1+r)}\Big),
% \end{align}
%which implies that 
%\begin{equation}\label{est-Q^n_r}
%  \sum_{j=1}^n \E[ |D^n_j|] \leq K\Big( n^{(\alpha-k)\beta+1} + n^{(\alpha-k) (1+r)+1}\Big)\leq K n^{( \alpha-k)\beta +1}.
%\end{equation}
% 
% %%%%%%%%%%%%%%
% 
% 
%  Fix  $\epsilon>0$. 
% By Lemma~\ref{est-H-inf} used on $r=1$,   \eqref{eq:72138} and \eqref{eq:72138-2} 
% %for  $q=q_\epsilon\in (1,\beta)$ chosen  close enough  to $\beta$  
% we obtain
%\begin{align}\label{eq:836a}
%\E[ |C^n_j|]\leq {}& K \big(n^{-1}   j^{(\alpha-k)\beta+1+\epsilon} +n^{-\beta } j^{\alpha \beta} \big).
%\end{align}
%To estimate $D^n_j$ we use substitution, 
Lemma~\ref{helplem1} to obtain 
\begin{equation}
%{}&  
\Big | |\rho_{j}^n|^\beta-   |\rho_j^\infty |^\beta\Big|  \leq  \Big| \| \phi^n_j \|_{L^\beta(\R)}^\beta- \| \phi_j^\infty \|_{L^\beta(\R)}^\beta \Big| 
% \\  {}&\qquad  \leq K\Big(  n^{\alpha \beta+1} \Big| \| g_{0,n}\|_{L^\beta(\R)}^\beta -  \| h_{0,n}\|_{L^\beta(\R)}^\beta  \Big| \Big)
 \leq K  n^{(\alpha-k)\beta+1} .\label{eq:63734}
\end{equation} 
For any $\gamma\in (p,\beta)$ such that $(\alpha-k)\gamma <-1$ we have  
\begin{align}
\E[  | D^n_j | ] = {}&  \E[  | \Phi_{\rho^n_j}( U_{j,0}^n)-\Phi_{ \rho_j^\infty}(U_{j,0}^n)| ] 
\leq K  \Big| |\rho^n_j |^\beta-  | \rho^\infty_j |^\beta \Big|  \E[ | U_{j,0}^n |^\gamma   ]  
\\ \leq {}& K  \Big| |\rho^n_j |^\beta-  | \rho^\infty_j |^\beta \Big|  \| \phi^n_j \|_{L^\beta([0,1])}^\gamma \leq K n^{(\alpha-k)\beta+1} j^{(\alpha-k)\gamma},
\end{align}
where the first inequality follows by  \eqref{est-H-rho-1-2}, the second inequality follows by 
 \eqref{eq:72138} and the last inequality follows by \eqref{eq:63734}. 
Since $(\alpha-k)\gamma<-1$, the estimate \eqref{eq:7239723}(b) follows. 

The estimates 
\eqref{est-H^3-slf} and \eqref{eq:7239723} yields that 
\begin{equation}\label{est-Q^n_r}
 \E[|H^{(3)}_n|]  \leq 2n   \sum_{j=1}^{n} \Big( \E[ | C^n_j|] + \E[|D^n_j|]\Big) \leq K  \Big( n^{(\alpha-k)\beta+2}+ n^{1-\beta+\epsilon}\Big) ,
\end{equation}
and completes the proof of the proposition. \qed

\subsection{Proof of Lemma~\ref{helplem1}}\label{help-lemma1}
%%%%%%%%%%%%%%%%%%%%%
We have  that $f(x)=g(x)x^{-\alpha}$ for $x>0$, $f(0)=1$ and $f$ is assumed to be right differentiable at zero, and hence we  may and do extend $f$ to a differentiable function from $\R$ which also will be denoted $f$. We recall the notation from \eqref{eq:23412343}: 
\begin{align*} 
\phi_{j}^n(s)= D^kg_n(j-s),\qquad g_n(x)=n^{\alpha} g(x/n), \qquad \phi_j^\infty(u)=h_k(j-u).
\end{align*}
By substitution we have that 
\begin{align*}
\|\phi^n_j\|_{L^\beta(\R)}^{\beta} = \int_0^{\infty} |D^kg_n(x)|^{\beta}\, dx 
\qquad \text{and} \qquad \|\phi^{\infty}_j\|_{L^\beta(\R)}^{\beta} = 
\int_0^{\infty} |h_k(x)|^{\beta} \,dx. 
\end{align*}
From Lemma \ref{helplem} and condition $\alpha < k- 1/\beta$ we obtain for all $n\geq 1$ that 
\begin{align} \label{est1}
A_n:=\int_{n}^{\infty} |h_{k} (x)|^{\beta}\, dx \leq K  \int_{n}^{\infty} x^{(\alpha -k)\beta} \,dx 
\leq K n^{(\alpha -k)\beta+1}.
\end{align}
The same estimate holds for the quantity $ \int_n^{\infty} |D^kg_n(x)|^{\beta}\, dx$.
On the other hand, we have that 
\begin{align} \label{est2}
B_n:=\left| \int_{0}^{k} |D^kg_n(x)|^{\beta}\, dx  - \int_{0}^{k} |h_k(x)|^{\beta}\, dx \right|
\leq K n^{-1}.
\end{align}
This follows by the estimate $| |x|^\beta - |y|^\beta|\leq K \max\{ |x|^{\beta-1},|y|^{\beta-1}\} |x-y|$ for all $x,y>0$, and that for all  $x\in [0,k]$ we have by differentiability of $f$ at zero that 
\begin{equation}
| D^kg_n(x)- h_{k}(x)| \leq Kn^{-1}x^{\alpha}.
\end{equation}
Recalling that $g(x)=x^{\alpha}_{+} f(x)$ and using $k$th order Taylor expansion of $f$ at $x$, we deduce the following identity
\begin{align*}
D^kg_n(x)={}&  n^{\alpha} \sum_{j=0}^k (-1)^j \binom{k}{j} g\big((x-j)/n\big) \\
= {}& \sum_{j=0}^k (-1)^j \binom{k}{j} \big(x-j\big)_{+}^{\alpha}
\left( \sum_{l=0}^{k-1} \frac{f^{(l)}(x/n)}{l!} (-j/n)^l + \frac{f^{(k)}(\xi_{j,x})}{k!} (-j/n)^k \right) \\
={}&  \sum_{l=0}^{k-1} \frac{f^{(l)}(x/n)}{l!} \left( \sum_{j=0}^k (-1)^j \binom{k}{j} (-j/n)^l \big(x -j\big)_{+}^{\alpha} \right) \\
{}&+ \left( \sum_{j=0}^k  \frac{f^{(k)}(\xi_{j,x})}{k!} (-1)^j \binom{k}{j} (-j/n)^k \big(x-j\big)_{+}^{\alpha} \right),  
\end{align*}
where $\xi_{j,x}$ is a certain intermediate point. Now, by rearranging terms we can find coefficients $\lambda_0^n,\cdots,\lambda_k^n:[k,n] \to \R$
and $\tilde{\lambda}_0^n,\cdots,\tilde{\lambda}_k^n :[k,n] \to \R$
(which are in fact bounded functions in $x$ uniformly in $n$)
such that  
\begin{align*}
D^kg_n(x)&= \sum_{l=0}^k \lambda_l^n(x) n^{-l} \left( \sum_{j=l}^k (-1)^j \binom{k}{j} j(j-1) \cdots (j-l+1) \big(x-j\big)_{+}^{\alpha} \right) \\
&= \sum_{l=0}^k \tilde{\lambda}_l^n(x) n^{-l} \left( \sum_{j=l}^k (-1)^j \binom{k-l}{j-l}  \big(x-j\big)_{+}^{\alpha} \right) 
=: \sum_{l=0}^k r_{l,n} (x). 
\end{align*}
At this stage we remark that the term $r_{l,n} (x)$ involves $(k-l)$th order differences of the function 
$x_{+}^\alpha$  and $\lambda_0^n(x)=\tilde{\lambda}_0^n(x)=f(x/n)$. 
Now, observe that 
\begin{align*}
 C_n:={}& 
\int_{k}^{n} \left| |D^kg_n(x)|^{\beta} - |h_{k}(x)|^{\beta} \right|  \,
dx \\
 \leq {}&  K\int_{k}^{n} \max\{|D^kg_n(x)|^{\beta-1}, |h_{k}(x)|^{\beta-1}\} |D^kg_n(x)- h_{k}(x)| \,dx.
\end{align*}
Since $r_{0,n}(x)=f(x/n) h_{k}(x)$ and $f(0)=1$,  it holds  that 
\begin{equation}
|r_{0,n}(x) - h_{k}(x) | \leq K  (x/n)|h_k(x)|.
\end{equation}
We deduce that 
\begin{align}
&\int_{k}^{n} \max\{|D^kg_n(x)|^{\beta-1}, |h_{k}(x)|^{\beta-1}\} |r_{0,n}(x) - h_{k}(x)| \,dx 
 \leq K n^{-1} \int_{k}^{n} x^{(\alpha-k) \beta +1}\, dx \\
\label{est4} & \leq K
\begin{cases}
n^{-1} & \text{when } \alpha \in (0,k-2/\beta) \\
n^{(\alpha-k) \beta+1} & \text{when } \alpha \in (k-2/\beta, k-1/\beta) 
\end{cases}
\end{align}  
For $1\leq l\leq k$,  we readily obtain the approximation
\begin{align*} 
\int_{k}^{n} \max\{|D^kg_n(x)|^{\beta-1}, |h_{k}(x)|^{\beta-1}\} |r_{l,n}(x)| \,dx 
\leq K n^{-l} \int_{k}^{n} x^{(\alpha -k)\beta  +l } \,dx. 
\end{align*}
If $\alpha \in (k-2/\beta, k-1/\beta) $, then 
$(\alpha -k)\beta  +l > -1$ and we have 
\begin{align} \label{est3a} 
\int_{k}^{n} x^{(\alpha -k)\beta  +l } \,dx \leq Kn^{(\alpha -k)\beta  +l+1}.
\end{align}
When $\alpha \in (0,k-2/\beta)$
it holds that 
\begin{align} \label{est3}
 \int_{k}^{n} x^{(\alpha -k)\beta  +l } \,dx \leq 
\begin{cases}
K & (\alpha -k)\beta  +l < -1 \\
K \log(n)n^{(\alpha -k)\beta  +l+1} & (\alpha -k)\beta  +l \geq -1.
\end{cases}
\end{align}
By \eqref{est4}, \eqref{est3a} and \eqref{est3} we conclude that  
\begin{align*} 
C_n
\leq K
\begin{cases}
 n^{-1} & \text{when } \alpha \in (0,k-2/\beta) \\
n^{(\alpha -k) \beta+1} & \text{when } \alpha \in (k-2/\beta, k-1/\beta) 
\end{cases}
\end{align*}
Since $(\alpha -k) \beta+1<-1$ if and only if $\alpha < k-2/\beta$, the result readily follows from \eqref{est1} and \eqref{est2}. 
\qed

\section*{Acknowledgements}
We would like to thank Donatas Surgailis, who helped us with the proof  of
Theorem~\ref{sec-order}(i).

\bibliographystyle{chicago}

\begin{thebibliography}{}
\bibitem{Aldous} D.J. Aldous and G. K. Eagleson (1978). On mixing and stability of limit theorems. 
{\em Ann. Probab.}~6(2), 325--331.

\bibitem{at87} F. Avram and M.S. Taqqu (1987). Noncentral limit theorems and Appell polynomials. 
{\em Ann. Probab.}~15(2), 767--775.

\bibitem{BNBV} O.E.  Barndorff-Nielsen, F.E. Benth, and A.E.D. Veraart (2013). Modelling energy spot prices by volatility modulated L\'evy-driven Volterra processes. {\em Bernoulli}~19(3), 803--845.

\bibitem{BNCP09} O.E. Barndorff-Nielsen, J.M. Corcuera and M. Podolskij (2009). Power variation
for Gaussian processes with stationary increments.  \textit{Stochastic Process. Appl.}~119(6), 1845--1865. 

\bibitem{BNCPW09} O.E. Barndorff-Nielsen, J.M. Corcuera, M. Podolskij and J.H.C. Woerner (2009). Bipower variation
for Gaussian processes with stationary increments.  \textit{J. Appl. Probab.}~46(1), 132--150.

\bibitem{BGJPS}  O.E. Barndorff-Nielsen, S.E. Graversen, 
J. Jacod, M. Podolskij and N. Shephard (2005). A central limit theorem
for realised power and bipower variations of continuous
semimartingales. In: Kabanov, Yu., Liptser, R., Stoyanov, J.~(eds.),
{\em From Stochastic Calculus to Mathematical Finance. Festschrift in
Honour of A.N.~Shiryaev}, 33--68, Springer, Heidelberg.

%\bibitem{bs07} O.E. Barndorff-Nielsen and J. Schmiegel (2007). Ambit processes; with applications
%to turbulence and cancer growth. In F.E. Benth, Nunno, G.D., Linstr{\o}m, T., {\O}ksendal,
%B. and Zhang, T. (Eds.): \textit{Stochastic Analysis and Applications: The Abel Symposium
%2005},  93--124,  Springer, Heidelberg. 
 
%\bibitem{BasRosFV} A. Basse-O'Connor and J. Rosi{\'n}ski (2013). Characterization of the finite variation property for a class
%of stationary increment infinitely divisible processes. {\em Stochastic Processes and Their Applications}~123(6), 1871--1890.

\bibitem{BasRosSM}
A.  Basse-O'Connor and J. Rosi\'nski (2016). On infinitely divisible semimartingales. 
     {\em   Probab. Theory Relat. Fields}~164(1-2), 133--163. 
     
     \bibitem{BCI} A. Benassi, S. Cohen and J. Istas (2004). On roughness indices for fractional fields.
\emph{Bernoulli}~10(2), 357--373. 

\bibitem{bls12} C. Bender, A. Lindner and M. Schicks (2012).
Finite variation of fractional L\'evy processes. \textit{J. Theoret. Probab.}~25(2), 595--612. 

\bibitem{bm08} C. Bender and T. Marquardt (2008).
Stochastic calculus for convoluted L\'evy processes. \textit{Bernoulli}~14(2), 499--518.

\bibitem{Berk}  K.N. Berk (1973). A central limit theorem for $m$-dependent random variables with unbounded $m$. 
{\em Ann. Probab.}~1(2),   352--354.

\bibitem{Billingsley} P. Billingsley (1999). {\em Convergence of Probability Measures (second edition)}.
Wiley Series in Probability and Statistics: Probability and Statistics.

\bibitem{SamBra} M. Braverman  and G. Samorodnitsky (1998). Symmetric infinitely divisible processes 
with sample paths in {O}rlicz spaces and absolute continuity of infinitely
divisible processes. {\em Stochastic Process. Appl.}~78(1), 1--26.

%\bibitem{BM83} P. Breuer and P. Major (1983). Central limit theorems for nonlinear functionals of
%Gaussian fields. \textit{J. Multivariate Anal.}~13(3), 425--441.

\bibitem{Ergodic-Cam} S. Cambanis, C.D. Hardin, Jr.,  and A. Weron (1987). 
 Ergodic properties of stationary stable processes. {\em Stochastic Process. Appl.}~24(1), 1--18.
 
\bibitem{Rosenblatt} A. Chronopoulou,  C.A. Tudor  and F.G. Viens  (2009).  Variations and Hurst index estimation for a Rosenblatt process using longer filters. {\em Electron. J. Stat.}~3, 1393--1435. 

\bibitem{CTV11} A. Chronopoulou,  C.A. Tudor  and F.G. Viens  (2011).  Self-similarity parameter estimation and reproduction property for non-Gaussian Hermite processes. 
{\em Communications on Stochastic Analysis}~5, 161--185.


 
\bibitem{Coeu} J.-F. Coeurjolly (2001). Estimating the parameters of a fractional Brownian motion by discrete variations of its sample paths.
{\em Stat. Inference Stoch. Process.}~4(2), 199--227.

%\bibitem{cnw06} J.M. Corcuera, D. Nualart and J.H.C. Woerner (2006). 
%Power variation of some integral fractional processes. \textit{Bernoulli}~12(4) 713--735. 

\bibitem{Jacod-round-off}  S. Delattre and J. Jacod (1997).  A central limit theorem for normalized functions of the increments of a diffusion process, in the presence of round-off errors.  {\em Bernoulli}~3(1), 1--28.

\bibitem{sg14} S. Glaser (2015). A law of large numbers for the power variation of fractional L\'evy processes.
 {\em Stoch. Anal. Appl.}~33(1), 1--20.

\bibitem{glt15} D. Grahovac, N.N. Leonenko and M.S. Taqqu (2015). 
Scaling properties of the empirical structure function of linear fractional stable motion and estimation of its parameters. {\em J. Stat. Phys.}~158(1), 105--119.

\bibitem{gl89} L. Guyon and J. Leon (1989):  Convergence en loi des $H$-variations d'un processus
gaussien stationnaire sur $\R$. {\em 
Ann. Inst. H. Poincar\'e Probab. Statist.}~25(3), 265--282. 

\bibitem{hh97} H.-C. Ho and T. Hsing (1997). Limit theorems for functionals of moving averages. 
{\em Ann. Probab.}~25(4), 1636--1669.

%\bibitem{hj94} J. Hoffman-J{\o}rgensen (1994). {\em Probability with a View towards Statistics: Volume 1}. Chapmann and Hall, Probability Series.

\bibitem{h99} T. Hsing (1999). On the asymptotic distributions of partial sums of functionals of infinite-variance moving averages.
{\em Ann. Probab.}~27(3), 1579--1599.

\bibitem{J} J. Jacod (2008). Asymptotic properties of realized power variations and related functionals of semimartingales. 
{\em Stochastic Process. Appl.}~118(4), 517--559.

\bibitem{JP} J. Jacod and P. Protter (2012). {\em Discretization of Processes.} Springer, Berlin.

\bibitem{k83} O. Kallenberg (2002). {\em Foundations of Modern Probability (second edition)}.  Springer-Verlag, New York.
   

\bibitem{FK}  
F.B. Knight (1992). {\em Foundations of the Prediction Process.} Oxford Science Publications, New York. 

\bibitem{ks01} H.L. Koul and D. Surgailis (2001). Asymptotics of empirical processes of long memory moving averages with infinite variance. {\em Stochastic Process. Appl.}~91(2), 309--336.

\bibitem{Marcus-Rosinski} M.B. Marcus and J. Rosi\'nski (2005). Continuity and boundedness of infinitely divisible processes: a Poisson point process approach. 
{\em J. Theoret. Probab.}~18(1),  109--160.

\bibitem{nr09} I. Nourdin and A. R\'eveillac (2009). Asymptotic behavior of weighted quadratic variations of fractional Brownian motion: the critical case 
$H=1/4$. \textit{Ann. Probab.}~37(6), 2200--2230.

%\bibitem{PecSolTaqUtz} G. Peccati, J.L. Sol\'e, M. Taqqu and F. Utzet (2010). Stein's method and normal approximation of Poisson functionals.
%{\em Annals of Probability}~38(2), 443--478.

\bibitem{PV} M. Podolskij and M. Vetter (2010). Understanding limit theorems for semimartingales: a short survey. 
{\em Stat. Neerl.}  64(3), 329--351.

\bibitem{RajRos} B. Rajput and J. Rosi{\'n}ski (1989). Spectral representations of infinitely divisible processes.
 {\em Probab. Theory Relat. Fields}~82(3), 451--487.

\bibitem{ren} A. Renyi (1963). On stable sequences of events. \textit{Sankhy\=a Ser. A}~25, 293--302.

%\bibitem{Ro-sto-int}  J. Rosi\'nski (1986): On stochastic integral representation of stable processes with
%sample paths in {B}anach spaces. {\em Journal of Multivariate Analysis}~20(2), 277--302.

%\bibitem{Ro-On-path} J. Rosi\'nski (1989). On path properties of certain infinitely divisible processes. {\em Stochastic Processes and Their Applications}~33(1), 73--87. 

\bibitem{SamTaq} G. Samorodnitsky and M.S. Taqqu (1994). \emph{Stable Non-Gaussian Random Processes: Stochastic Models with Infinite Variance.}
Chapmann and Hall, New York.

%\bibitem{Sato} K. Sato (1999). {\em L\'evy processes and Infinitely Divisible Distributions.} Cambridge Studies in Advanced Mathematics 68, Cambridge
%University Press.

\bibitem{Serfozo} R. Serfozo (2009).
 {\em Basics of Applied Stochastic Processes}.
Probability and its Applications (New York),
Springer-Verlag, Berlin.

\bibitem{s02} D. Surgailis (2002). Stable limits of empirical processes of moving averages with infinite variance.  
{\em Stochastic Process. Appl.}~100(1--2), 255–--274.

\bibitem{s04} D. Surgailis (2004). Stable limits of sums of bounded functions of long-memory moving averages with
finite variance. {\em Bernoulli} 10(2),  327--355.

\bibitem{Takashima} K. Takashima (1989). Sample path properties of ergodic self-similar processes. 
   {\em Osaka J. Math.}~26(1), 159--189.	

%\bibitem{T79} M.S. Taqqu (1979). Convergence of integrated processes of arbitrary Hermite rank. 
%\textit{Z. Wahrsch. Verw. Gebiete}~50(1), 53--83. 

\bibitem{Rosenblatt-1}
C.A. Tudor and F.G. Viens (2009). Variations and estimators for self-similarity parameters via Malliavin calculus. 
{\em Ann. Probab.}~37(6), 2093--2134.



\bibitem{Tukey}  J.W. Tukey (1938). On the distribution of the fractional part of a statistical variable.
Rec. Math. [Mat. Sbornik] N.S., 4(46):3, 561--562. 

%\bibitem{Turner-Rosinski}
%M.D. Turner (2011): Explicit {$L^p$-norm} estimates of infinitely divisible random
%vectors in {Hilbert} spaces with applications. {\em PhD dissertation, University of Tennessee}.

\bibitem{Wat}
T. Watanabe (2007). 
Asymptotic estimates of multi-dimensional stable densities and
              their applications. 
              {\em  Trans. Amer. Math. Soc.}~359(6), 2851--2879.
           	


\end{thebibliography}
 
\end{document}